\documentclass[11pt,letterpaper]{amsart}
\usepackage{graphicx}
\usepackage{amsmath}
\usepackage{amssymb}
\usepackage{amsfonts}
\usepackage{amsthm}
\usepackage{cancel}
\usepackage[all]{xy}
\usepackage[neveradjust]{paralist}
\usepackage{enumitem}
\usepackage{float}
\usepackage{multicol}
\usepackage{mathrsfs}
\usepackage{mathdots}
\usepackage{mathtools}
\usepackage{todonotes}
\usepackage[normalem]{ulem} 
\usepackage{caption} 
\usepackage{bbm}
\usepackage[pagebackref,colorlinks=true,linkcolor=blue,citecolor=blue, urlcolor=blue]{hyperref}
\usepackage{tikz-cd}
\usepackage{tikz}
\tikzcdset{scale cd/.style={every label/.append style={scale=#1},
    cells={nodes={scale=#1}}}}
\usetikzlibrary{matrix, positioning}

\newlength{\wdth}

\usepackage{tabularray}

\newcommand{\Z}{\mathbb{Z}}

\newcommand{\R}{\mathbb{R}}

\newcommand{\Hom}{{\mathrm{Hom}}}

\newcommand{\Eseg}{{\mathrm{Eseg}}}
\newcommand{\Vseg}{{\mathrm{Vseg}}}
\newcommand{\Rep}{{\mathrm{Rep}}}
\newcommand{\VRep}{{\mathrm{VRep}}}
\newcommand{\up}{{\mathrm{up}}}
\newcommand{\down}{{\mathrm{down}}}

\newcommand{\B}{{\mathcal B}}

\newcommand{\inv}{^{-1}}
\newcommand{\sm}{\setminus}

\newcommand{\GL}{\mathrm{GL}}
\newcommand{\SO}{\mathrm{SO}}
\newcommand{\OO}{\mathrm{O}}
\newcommand{\SL}{\mathrm{SL}}
\newcommand{\Sp}{\mathrm{Sp}}

\newcommand{\BC}{\mathbb{C}}

\newcommand{\BB}{\mathcal{B}}
\newcommand{\CC}{\mathcal{C}}
\newcommand{\EE}{\mathcal{E}}
\newcommand{\FF}{\mathcal{F}}

\newcommand{\Block}{\mathrm{Block}}
\newcommand{\ABlock}{\mathrm{ABlock}}

\newcommand{\rc}{\operatorname{rc}}
\newcommand{\tc}{\operatorname{tc}}

\renewcommand{\implies}{\Rightarrow}

\newcommand{\comment}[1]{}

\newtheorem{thm}{Theorem}[section]
\newtheorem{cor}[thm]{Corollary}
\newtheorem{lemma}[thm]{Lemma}
\newtheorem{prop}[thm]{Proposition}
\newtheorem{conj}[thm]{Conjecture}

\newtheorem{ques/conj}[thm]{Question/Conjecture}

\newtheorem{defn}[thm]{Definition}
\newtheorem{prob}[thm]{Problem}
\newtheorem{rmk}[thm]{Remark}

\newtheorem{exmp}[thm]{Example}
\newtheorem{algo}[thm]{Algorithm}

\DeclareMathOperator{\supp}{supp}

\numberwithin{equation}{section}

\begin{document}

\title{Beyond the Adams Conjecture}
\author{Alexander Hazeltine, Aarya Kumar, Andrew Tung}
\date{\today}
\thanks{The research of the first named author was supported by the AMS-Simons Travel Grant program. The research of the second and third named authors was supported by Department of Mathematics at the University of Michigan and NSF Grant DMS-2301507.}

\begin{abstract}
    For symplectic and even orthogonal groups over a $p$-adic field, we determine the number of local Arthur packets containing the local theta lift of a tempered representation at the first occurrence in the going-up tower. These counts show that the local theta lifts may lie in many more local Arthur packets than those predicted by the Adams conjecture.
\end{abstract}

\maketitle

\tableofcontents

\section{Introduction}\label{sec intro}

Recently, Ben-Zvi, Sakellaridis, and Venkatesh have proposed a series of conjectures which are known as the relative Langlands program (\cite{BZSV24}). One of the key ideas in their theory is the notion of duality. The goal of this article is to study a potential manifestation of this duality which we outline below.

The influential Gan-Gross-Prasad (GGP) conjectures seek to study certain branching problems for classical groups (\cite{GGP12, GGP20}). For example, let $n\in\mathbb{Z}_{\geq 1}$, $F$ be a non-Archimedean local field of characteristic 0, $G=\SO_{2n+1}(F)$ and $H=\SO_{2n}(F).$ Given an irreducible $\pi$ representation of $G$, the GGP conjectures aim to study the restriction of $\pi$ to $H$, i.e., to determine the irreducible admissible representations $\pi'$ of $H$ for which
\[
\Hom_H(\pi\otimes \pi',\mathbb{C})\neq 0.
\]
The GGP conjectures provide a partial answer in terms of \emph{relevant} pairs of local Arthur packets (\cite[Conjectures 6.1 and 7.1]{GGP20}). However, it is known that this relevance condition is not necessary. Specifically, there exists $\pi$ and $\pi'$ as above such that $\Hom_H(\pi\otimes \pi',\mathbb{C})\neq 0$, but $\pi$ and $\pi'$ lie in a pair of non-relevant local Arthur packets (\cite[\S7]{GGP20}).  

The GGP conjectures fit into the broader theory of the relative Langlands program (\cite[\S1.5]{BZSV24}). In this framework, the GGP conjectures can be considered as the ``dual'' problem to the Adams conjecture (see Conjecture \ref{conj Adams intro} below). We defer to \cite[Remark 7.12]{GW25} for the details of this duality.
We are particularly interested in studying the analogues of non-relevant pairs (called non-$\theta$-relevant below) for the Adams conjecture.

Hereinafter, let $m,n\in\mathbb{Z}_{\geq 0}$ and $G_n=\Sp_{2n}(F)$ denote a split symplectic group and $H_m^\pm=\OO_{2m}^\pm(F)$ denote an even orthogonal group (see \S\ref{sec local theta correspondence} for an explanation of the notation). We assume that $m>n$ and let $\alpha=2m-2n-1$ denote the difference in the ranks of the dual groups. Let $\Pi(G_n),$ respectively $\Pi(H_m^\pm)$, denote the set of equivalence classes of complex irreducible admissible representations of $G_n$, respectively $H_m^\pm.$
We consider the local theta correspondence defined by Howe (\cite{How79}) which we denote as a map $\theta_{-\alpha}^\pm:\Pi(G_n)\rightarrow\Pi(H_m^\pm)\cup\{0\}$. The local theta correspondence has proven to be an incredibly useful tool within the Langlands program; however, it does not necessarily preserve $L$-packets. As an attempt at remedying this non-preservation, Adams proposed that instead of $L$-packets, the local theta correspondence should preserve local Arthur packets (\cite{Ada89}). This prediction is known as the Adams conjecture and we recall its precise formulation in Conjecture \ref{conj Adams intro} below.

We briefly recall several notions related to local Arthur packets and defer to \S\ref{sec local Arthur packets} for details.
For brevity, let $G\in\{G_n,H_m^\pm\}.$ A local Arthur parameter of $G$ may be roughly thought of as a homomorphism $\psi:W_F\times\SL_2(\BC)\times\SL_2(\BC)\rightarrow{}^LG,$ where $W_F$ is the Weil group of $F$ and ${}^LG$ denotes the $L$-group of $G$. We note that we may regard ${}^LG_n=\SO_{2n+1}(\BC)$ and ${}^L H_m^\pm=\OO_{2m}(\BC).$ Furthermore, we have a natural inclusion of ${}^LG_n\hookrightarrow{}^LH_m^\pm$ since $m>n.$ In Arthur's seminal work, to each local Arthur parameter $\psi$ of $G$, Arthur attaches a local Arthur packet $\Pi_\psi$ which is a finite subset of $\Pi(G)$ which satisfies the twisted endoscopic character identities (\cite[Theorem 1.5.1]{Art13}).

Let $\chi_W$ denote the trivial representation of $W_F$ and $\chi_V$ denote the character of $W_F$ associated to a certain quadratic character related to $H_m^\pm$ (see \S\ref{sec the Adams conjecture}) via local class field theory. Given a local Arthur parameter $\psi$ of $G_n$, we set
\[
\psi_\alpha:=(\chi_W\chi_V^{-1}\otimes\psi)\oplus \chi_W\otimes S_1\otimes S_\alpha.
\]
The Adams conjecture is the following.

\begin{conj}[The Adams conjecture {(\cite{Ada89})}]\label{conj Adams intro}
    Suppose that $\pi\in\Pi_\psi$ for some local Arthur parameter $\psi$ of $G_n.$ If $\theta_{-\alpha}^\pm(\pi)\neq 0,$ then $\theta_{-\alpha}^\pm(\pi)\in\Pi_{\psi_\alpha}.$
\end{conj}

We remark that the Adams conjecture in this setting is fully understood via \cite{BH22, Moe11c}. In particular, M{\oe}glin showed that the Adams conjecture is true when $\alpha\gg0$ (\cite[Theorem 6.1]{Moe11c}) but it can and does fail otherwise. The exact extent of this failure is determined by Baki{\'c} and Hanzer algorithmically (\cite[Theorem A]{BH22}).
Moreover, if we fix $\pi$ and allow $\psi$ to vary, then we understand the extent to which the Adams conjecture holds via \cite[Theorem 1.3]{Haz24}.

Note that it is a pair of local Arthur parameters $(\psi,\psi_\alpha)$ which are involved in the Adams conjecture. In analogy with the GGP conjectures, we say that a pair of local Arthur parameters $(\psi,\psi')$ of $G_n\times H_m^\pm$ are \emph{$\theta$-relevant} if $\psi'=\psi_\alpha.$ Note that representations often lie in many local Arthur packets (\cite{Ato23, HLL22, HLL25}). Consequently, we expect to have many non-$\theta$-relevant pairs $(\psi,\psi')$ of local Arthur parameters for which $(\pi,\theta_{-\alpha}^\pm(\pi))\in\Pi_\psi\times\Pi_{\psi'}$. Our overarching goal is a systematic investigation of these non-$\theta$-relevant pairs given a fixed $\pi\in\Pi(G_n)$.

\begin{prob}\label{prob non theta relevant}
    Determine all non-$\theta$-relevant pairs given a fixed $\pi\in\Pi(G_n)$.
\end{prob}

For a representation $\pi'\in\Pi(G),$ let $\Psi(\pi')$ denote the set of local Arthur parameters $\psi$ of $G$ such that $\pi'\in\Pi_\psi.$ Then, for a fixed $\pi\in\Pi(G_n),$ we have that $(\pi,\theta_{-\alpha}^\pm(\pi))\in\Pi_\psi\times\Pi_{\psi'}$ for any $(\psi,\psi')\in\Psi(\pi)\times\Psi(\theta_{-\alpha}^\pm(\pi)).$  Consequently, determining the non-$\theta$-relevant pairs reduces to the computation of $\Psi(\theta_{-\alpha}^\pm(\pi)).$ For $\alpha\gg0,$ it is a straightforward consequence of the theory of intersections of local Arthur packets developed by Liu, Lo, and the first author (\cite{HLL22, HLL25}, see Theorem \ref{thm intersections of local Arthur packets}) that $\Psi(\theta_{-\alpha}^\pm(\pi))=\{\psi_\alpha \ | \ \psi\in\Psi(\pi)\}$ and so determining the non-$\theta$-relevant pairs is simple for $\alpha\gg0$. Specifically, for $\alpha\gg0,$ we have $|\Psi(\pi)|=|\Psi(\theta_{-\alpha}^\pm(\pi))|$. When $\alpha$ is small, the problem becomes significantly more complicated as we discuss in the following example.

Let $\up\in\{+,-\}$ be the ``going-up'' tower for $\pi\in\Pi(G_n)$ and $m^{\up, \alpha}(\pi)$ denote a certain odd integer determined by the first occurrence of the local theta lift $\pi$ in the going-up tower (see \S\ref{sec local theta correspondence} for terminology). Baki{\'c} and Hanzer showed that the Adams conjecture holds at (and above) the first occurrence in the going-up tower (\cite[Theorem 2]{BH22}, see Theorem \ref{thm Adams going up tower}). Suppose that $\pi$ is supercuspidal. Then $\theta_{-m^{\up, \alpha}(\pi)}^\up(\pi)$ is also supercuspidal (\cite{Kud86}). Results of M{\oe}glin imply that $|\Psi(\theta_{-m^{\up, \alpha}(\pi)}^\up(\pi))|=3|\Psi(\pi)|$ (\cite{Moe06b,Moe09a}, see \cite[Theorem 7.5]{GGP20}) and so we obtain many non-$\theta$-relevant pairs than in the case for $\alpha\gg0$. In particular, $\theta_{-m^{\up, \alpha}(\pi)}^\up(\pi)$ lies in many local Arthur packets besides those predicted by the Adams conjecture.

In this article, we provide a conjecture which completely determines the count $|\Psi(\theta_{-m^{\up, \alpha}(\pi)}^\up(\pi))|$ when $\pi\in\Pi(G_n)$ is tempered. 

\begin{conj}[Theorem \ref{thm-count-theta-temp} and Conjecture \ref{conj-count-theta-temp}]\label{conj theta up first occur count intro}
   Suppose that $\pi\in\Pi(G_n)$ is tempered. Then there is a recursive formula for computing $|\Psi(\theta_{-m^{\up, \alpha}(\pi)}^\up(\pi))|$.
\end{conj}

The precise form of Conjecture \ref{conj theta up first occur count intro} is given in seven cases based on a reduction to a certain combinatorial problem which we discuss below. The main result of this article is that we verify six of these cases.

\begin{thm}[Theorem \ref{thm-count-theta-temp}]\label{thm conjecture verification intro}
In six cases, we determine a recursive formula for $|\Psi(\theta_{-m^{\up, \alpha}(\pi)}^\up(\pi))|$.
\end{thm}

The remaining case for Conjecture \ref{conj theta up first occur count intro} is stated in Conjecture \ref{conj-count-theta-temp}.

Before remarking on the ideas behind the proof, we note that one could compute $\Psi(\theta_{-m^{\up, \alpha}(\pi)}^\up(\pi))$ via brute force by algorithmically using algorithms for the local theta correspondence (\cite{AG17a,BH22}) and algorithms for determining the local Arthur packets to which the theta correspondence belongs (\cite{HLL25}), but the relation to $\Psi(\pi)$ becomes unclear. Furthermore, this computation by brute force is impractical for theoretical applications. In contrast, the main benefit of
Conjecture \ref{conj theta up first occur count intro} is its explicit use of the structure of $\Psi(\pi)$. For example, when $\pi$ is supercuspidal, Theorem \ref{thm conjecture verification intro} immediately implies that $|\Psi(\theta_{-m^{\up, \alpha}(\pi)}^\up(\pi))|=3|\Psi(\pi)|$. In contrast, the brute force approach would require computing both $\Psi(\theta_{-m^{\up, \alpha}(\pi)}^\up(\pi))$ and $\Psi(\pi)$ explicitly.

Of course, as outlined earlier, the supercuspidal case was already known. As further applications, we determine  $|\Psi(\theta_{-m^{\up, \alpha}(\pi)}^\up(\pi))|$ in several new important cases.

\begin{thm}
    Let $\pi\in\Pi(G_n)$ be of Arthur type.
    \begin{enumerate}
        \item If $\pi$ is anti-tempered, then $|\Psi(\theta_{-m^{\up, \alpha}(\pi)}^\up(\pi))|$ is determined via a recursive formula (Theorem \ref{theta-antitempered-count}).
        \item If $\pi$ is generic, then $|\Psi(\theta_{-m^{\up, \alpha}(\pi)}^\up(\pi))|\in\{1,2,3\}$ (Theorem \ref{thm-generic-theta-count}).
        \item If $\pi$ is unramified, then $|\Psi(\theta_{-m^{\up, \alpha}(\pi)}^\up(\pi))|\in\{2,3\}$ (Theorem \ref{thm-unram-theta-count}).
    \end{enumerate}
\end{thm}

We note that when $\pi$ is generic or unramified, then $|\Psi(\pi)|=1$ (see \cite[Theorem 1.2]{HLL24} for the generic case and \cite[Proposition 6.4]{Moe09b} for the unramified case; see also \cite[Proposition 5.5]{HLL24}). In particular, there are often non-$\theta$-relevant pairs in these cases.

We now remark on the strategy of the proof of Theorem \ref{thm conjecture verification intro}. The main idea is to reduce to a combinatorial problem using several deep results: the parameterization of local Arthur packets by extended multi-segments and their theory of intersections developed in \cite{Ato20b, HLL22, HLL25}, the validity of the Adams conjecture for the going-up tower (\cite{BH22, Moe11c}), the calculation of the first occurrence on the going-up tower for tempered representations (\cite{AG17b}), and the translation of these results into the language of extended multi-segments (\cite{HLL25}). The resulting combinatorial problem then largely follows from the following counting problem. 

\begin{prob}\label{prob temp count intro}
    Given a tempered representation $\pi\in\Pi(G_n)$, determine is the size of $\Psi(\pi)$.
\end{prob}
Indeed, in many cases if $\pi$ is tempered, then $\theta_{-m^{\up, \alpha}(\pi)}^\up(\pi)$ is also tempered (\cite[Theorem 4.1]{AG17a}). In fact, the computation of $\theta_{-m^{\up, \alpha}(\pi)}^\up(\pi)$ only affects a certain part of the Langlands classification of $\pi$ and so it largely suffices to understand Problem \ref{prob temp count intro} for a certain class of tempered representations. 
In an earlier paper, we developed a partial result towards Problem \ref{prob temp count intro} which covers this class of tempered representations (\cite[Theorem 3.7]{HKT26}; see Theorem \ref{thm-count-temp}). 

\begin{thm}\label{thm tempered count intro}
For certain tempered representations $\pi\in\Pi(G_n)$, there is a recursive formula which determines $|\Psi(\pi)|$.
\end{thm}

The recursive formula for the above theorem is determined by decomposing extended multi-segments into blocks (see Definition \ref{def-block}) and then determining the number of local Arthur packets ``containing'' each block (Theorem \ref{thm-count-block-temp}).
The precise form of Conjecture \ref{conj theta up first occur count intro} predicts that $|\Psi(\theta_{-m^{\up, \alpha}(\pi)}^\up(\pi))|$ is tempered (in which case we use Theorem \ref{thm-count-temp}) or is determined recursively based on these blocks, or almost blocks, in which case we use Theorem \ref{thm-count-block-temp}, respectively Theorem \ref{thm-count-block-theta-temp}. Based on the block decomposition,  Conjecture \ref{conj theta up first occur count intro} falls into the seven cases mentioned above. 
The main difficulty between the various cases is determining how different blocks may or may not interact with each other.

Here is the organization of this paper.  
In \S\ref{sec background} and \S\ref{sec extended multi-segments}, we recall various results on the local theta correspondence and parameterization of local Arthur packets using extended multi-segments. 
We then state the main results precisely in \S\ref{sec-results}. 
Before proving the main results, we collect some technical preliminary definitions and statements in \S\ref{sec basic notions and lemmas}. 
We proceed to prove the main results in \S\ref{sec theta results}. 
We further provide some motivation behind the conjectural formula for the remaining case of Conjecture \ref{conj theta up first occur count intro}, i.e., Conjecture \ref{conj-count-theta-temp}, in Appendix \ref{sec motivation}. We also discuss some results related to determining $\Psi(\theta_{-m^{\up, \alpha}(\pi)}^\up(\pi)),$ as opposed to just its size, in Appendix \ref{sec commutativity}.

\subsection*{Acknowledgments}
The first named author thanks Yiannis Sakellaridis for inspiring conversations which started this research. The authors also thank Chi-Heng Lo for enlightening discussions and remarks. Finally, the authors thank Wee Teck Gan for helpful comments.

\section{Background}\label{sec background}

Recall that $F$ is a non-Archimedean local field of characteristic 0. For the moment, let $\mathrm{G}$ be a reductive group defined over $F$ and $G=\mathrm{G}(F).$ We are primarily concerned with the set of equivalence classes of complex irreducible admissible representations of $G$, which we denote by $\Pi(G).$ 

\subsection{Local theta correspondence}\label{sec local theta correspondence}

We let $n\in\mathbb{Z}_{\geq 0}$, fix an additive character $\psi_F$ of $F,$ and 
let $W_{2n}$ denote the unique (up to isomorphism) symplectic vector space over $F$ of dimension $2n.$ 
Let $m\in\mathbb{Z}_{\geq 1}$, fix $d\in F^\times/(F^\times)^2$, and consider the quadratic spaces over $F$ of dimension $2m$ and discriminant $d.$ Note that up to isomorphism, there are two such quadratic spaces which are distinguished by their Hasse-Witt invariant. We let $V_{2m}^\pm$ be the unique quadratic space over $F$ of dimension $2m$, discriminant $d,$ and Hasse-Witt invariant $\pm1.$ The isometry groups of $W_{2n}$ and $V_{2m}^\pm$, denoted $G_n=G(W_{2n})$ and $H_m^\pm=H(V_{2m}^\pm)$, respectively, are isomorphic to the symplectic group $\Sp_{2n}(F)$ and an even orthogonal group $\OO_{2m}^\pm(F)$ (which depends on $d$, but we have omitted it in the notation). Here the superscript $\pm$ in $\OO_{2m}^\pm(F)$ is used to keep track of the underlying quadratic space as it plays an important role in the local theta correspondence. We have that $H_m^+$ is quasi-split while $H_m^-$ denotes its (possibly non-quasi-split) nontrivial pure inner form.
These symplectic and quadratic spaces are naturally arranged into towers:
\begin{align*}
    \mathcal{W}&=\{W_{2n} \ | \ n\in\mathbb{Z}_{\geq 0}\}, \\
    \mathcal{V}^+&=\{V_{2m}^+ \ | \ m\in\mathbb{Z}_{\geq 1}\}, \\
    \mathcal{V}^-&=\{V_{2m}^- \ | \ m\in\mathbb{Z}_{\geq 1}\}.
\end{align*}

The pair $(G_n,H_m^\pm)$ forms a reductive dual pair of a certain metaplectic group. Consequently, given $\pi\in\Pi(G_n),$ we may consider its local theta lift which we denote by $\theta_{W_{2n},V_{2m}^\pm,\psi_F}(\pi)$ which either vanishes or is an element of $\Pi(H_m^\pm)$. The local theta lift was originally defined by Howe (\cite{How89}) and has found many uses within the Langlands program. For our purposes, it is sufficient to recall several properties of the local theta lift.
First we recall a property called Howe duality. It was originally conjectured by Howe (\cite{How79}) and later proved by Waldspurger, Gan and Takeda, and Gan and Sun (\cite{GS17, GT16, Wal90}).

\begin{thm}[Howe Duality]\label{thm Howe duality} 
Let $\pi_1,\pi_2\in\Pi(G_n).$
\begin{enumerate}
    \item If $\theta_{W_{2n},V_{2m}^\pm,\psi_F}(\pi_1)\neq 0$, then $\theta_{W_{2n},V_{2m}^\pm,\psi_F}(\pi_1)$ is irreducible.
    \item If $\pi_1\not\cong\pi_2$ and both $\theta_{W_{2n},V_{2m}^\pm,\psi_F}(\pi_1)$ and $\theta_{W_{2n},V_{2m}^\pm,\psi_F}(\pi_2)$ are nonzero, then \[\theta_{W_{2n},V_{2m}^\pm,\psi_F}(\pi_1)\not\cong \theta_{W_{2n},V_{2m}^\pm,\psi_F}(\pi_2).\]
\end{enumerate}
\end{thm}
The next theorem is known as the persistence principle (or the tower property) for the local theta correspondence.
\begin{thm}[{\cite{Kud86}}]\label{thm persistence principle}
Let $\pi\in\Pi(G_n).$ If $\theta_{W_{2n},V_{2m}^\pm,\psi_F}(\pi)\neq 0$, then $\theta_{W_n,V_{2m'}^\pm,\psi_F}(\pi)\neq 0$ for any $m'\geq m.$
\end{thm}

The persistence principle allows us to consider the first occurrence which is defined as follows. 

\begin{defn}\label{def first occurence}
    Let $\pi\in\Pi(G_n).$ The \emph{first occurrence} of $\pi$ (in $\mathcal{V}^\pm$) is
    \[
    m^{\pm}(\pi):=\min\{2m \ | \ \theta_{W_{2n},V_{2m}^\pm,\psi_F}(\pi)\neq 0\}.
    \]
\end{defn}

Note that since we have two ``target'' towers $\mathcal{V}^\pm$, we have two first occurrences. They are related by the following theorem which is known as the conservation relation.

\begin{thm}[{\cite{SZ15}}]\label{thm conservation relation}
    Let $\pi\in\Pi(G_n).$ Then
    $$
    m^{+}(\pi)+m^{-}(\pi)=4n+4.
    $$
\end{thm}

As a consequence, we may choose $\up, \down\in\{\pm\}$ such that $m^{\up}(\pi)\geq 2n+2\geq m^{\down}(\pi).$ Also, if one inequality is strict, then both inequalities are strict. In this situation, we call the tower whose first occurrence is $m^{\up}(\pi)$ the ``going-up'' tower for $\pi$ and denote it by $\mathcal{V}^{\up}.$ Similarly, we call the tower whose first occurrence is $m^{\down}(\pi)$ the ``going-down'' tower for $\pi$ and denote it by $\mathcal{V}^{\down}.$ 
When $m^{\up}(\pi)=m^{\down}(\pi),$ there is an ambiguity in the choice of $\up,\down\in\{\pm\}$, but it will not matter for the Adams conjecture (see Remark \ref{rmk equal first occurences}).

It is convenient for us to repurpose our notation. Fix $m$ and $n$ and let $\alpha=2m-2n-1$ (which is the difference in the ranks of the dual groups). For $\pi\in\Pi(G_n),$ we let $\theta_{-\alpha}^{\pm}(\pi)=\theta_{W_{2n},V_{2m}^\pm,\psi_F}(\pi).$ We let $\up, \down\in\{\pm\}$ be such that $m^{\up}(\pi)\geq m^{\down}(\pi).$ This allows us to consider the local theta lifts $\theta_{-\alpha}^{\up}(\pi)$ or $\theta_{-\alpha}^{\down}(\pi)$ for the going-up and going-down towers for $\pi$, respectively. We also set $m^{\up, \alpha}(\pi)=m^{\up}(\pi)-2n-1.$

\subsection{Local Arthur packets}\label{sec local Arthur packets}

In this subsection, recall some results concerning local Arthur packets. We let $G$ denote one of $G_n=\Sp_{2n}(F)$ or $H_m^\pm=\OO^\pm_{2m}(F)$ for brevity. Note that $G_n$ is connected, but $H_m^\pm$ is disconnected. Consequently, we will consider the complex dual group and L-group for the identity component  $(H_m^\pm)^\circ=\SO_{2m}^\pm(F).$ We have that
the complex dual group $\widehat{G}(\BC)$ is given by $\SO_{2n+1}(\BC)$ or $\SO_{2m}(\BC)$, respectively. Recall that $G_n$ is split while $H_m^\pm$ is  possibly non-quasi-split. However $H_m^\pm$ splits over a quadratic extension of $F$ and hence, the L-group ${}^LG$ may be identified with $\SO_{2n+1}(\BC)$ (when $G=G_n$), $\SO_{2m}(\BC)$ (when $H_m^+$ is split) which we regard as a subgroup of $\OO_{2m}(\BC)$, or $\OO_{2m}(\BC)$ (when $H_m^+$ is quasi-split, not not split). We set ${}^LG'$ to be $\SO_{2n+1}(\BC)$ when $G=G_n$ and $\OO_{2m}(\BC)$ when $G=H_m^\pm$.

Using the standard embedding of ${}^LG\rightarrow{}^LG'\rightarrow\GL_{N}(\BC)$, where $N\in\{2n+1,2m\}$ according to whether $G=G_n$ or $H_m^\pm$, we may view
a local Arthur parameter of $G$ as
a direct sum of irreducible representations of $W_F\times\SL_2(\mathbb{C})\times\SL_2(\mathbb{C})$
\begin{equation}\label{eq decomp psi +}
  \psi = \bigoplus_{i=1}^r \phi_i|\cdot|^{x_i} \otimes S_{a_i} \otimes S_{b_i},  
\end{equation}
satisfying the following conditions:
\begin{enumerate}
    \item [(1)]$\phi_i(W_F)$ is bounded and consists of semi-simple elements, and $\dim(\phi_i)=d_i$;
    \item [(2)] $\psi$ is self-dual;
    \item [(3)] $x_i \in \R$ and $|x_i|<\frac{1}{2}$;
    \item [(4)]the restrictions of $\psi$ to the two copies of $\SL_2(\mathbb{C})$ are analytic, $S_k$ is the $k$-dimensional irreducible representation of $\SL_2(\mathbb{C})$, and 
    $$\sum_{i=1}^r d_ia_ib_i = N:= 
\begin{cases}
2n+1 & \text{ when } G=\Sp_{2n}(F),\\
2m & \text{ when } G=\OO^\pm_{2m}(F).
\end{cases}
$$ 
\end{enumerate}
We remark that the bound $|x_i|<\frac{1}{2}$ is used as it is a bound for the Ramanujan Conjecture. Furthermore, when $N=2m$, we have that $\det\psi$ determines the discriminant character of $V_{2m}^\pm$ and hence determines $H_m^\pm$.
We also note that we may equivalently view $\psi$ as a suitable homomorphism
$$\psi: W_F \times \SL_2(\mathbb{C}) \times \SL_2(\mathbb{C}) \rightarrow {}^LG.$$

With this point of view, $\psi$ is called \emph{relevant} for $G$ if any parabolic subgroup ${}^LP$ of ${}^LG$ through which $\psi$ factors is relevant for $G$ (in the sense of \cite[\S3.3]{Bor79}). We note that this condition only matters for us when $G$ is the nontrivial pure inner form of the split group $\OO_{2m}(F)$.

Two local Arthur parameters of $G$ are said to be equivalent if they are conjugate under ${}^LG'$. When  $G=G_n$, this is equivalent to the usual notion of equivalence using conjugation by the complex dual group. However, if $G=H_m^\pm,$ defining the equivalence using conjugation by the complex dual group would result in equivalence classes of local Arthur parameters for $(H_m^\pm)^\circ,$ not $H_m^\pm.$

We let $\Psi^+(G)$ denote the set of equivalence classes of relevant local Arthur parameters. We will not distinguish a local Arthur parameter $\psi$ and its equivalence class. We let $\Psi(G)$ denote the subset of equivalence classes of bounded relevant local Arthur parameters, i.e., those $\psi$ for which $x_i=0$ for any $i=1,\dots,r$ in the decomposition \eqref{eq decomp psi +}.

By the Local Langlands Correspondence for $\GL_{d}(F)$, any bounded representation $\phi$ of $W_F$ corresponds to an irreducible unitary supercuspidal representation $\rho$ of $\GL_{d}(F)$ (\cite{HT01, Hen00, Sch13}). Consequently, we may identify \eqref{eq decomp psi +} as
\begin{equation}\label{A-param decomp}
  \psi = \bigoplus_{\rho}\left(\bigoplus_{i\in I_\rho} \rho|\cdot|^{x_i} \otimes S_{a_i} \otimes S_{b_i}\right),  
\end{equation}
where the first sum runs over a finite set of
irreducible unitary supercuspidal representations $\rho$ of $\GL_d(F)$ where $d \in \mathbb{Z}_{\geq 1}$.

When $G$ is quasi-split,
for a local Arthur parameter $\psi \in \Psi(G)$, Arthur constructed a finite multi-set $\Pi_\psi$ consisting of irreducible unitary representations of $G$ that satisfy certain twisted endoscopic character identities (\cite{Art13}). We call $\Pi_\psi$ the \emph{local Arthur packet} of $\psi.$ When $G$ is a pure inner form of a quasi-split group, e.g. $G=H_m^-$, we define $\Pi_\psi$ via quasi-split transfer (as in \cite{Moe11a}).
M{\oe}glin showed that $\Pi_\psi$ is multiplicity-free (\cite{Moe11a}). We do not recall the precise definition of $\Pi_\psi$. Instead, it suffices for our purposes to recall a parameterization of $\Pi_\psi$ using extended multi-segments (see Theorem \ref{thm Arthur packets extended multi-segment parameterization}).

M{\oe}glin showed that the computation of $\Pi_\psi$ can be reduced to the ``good parity'' case (see Theorem \ref{thm red to gp} below). We recall this reduction.
\begin{defn}\label{def Arthur parameter good parity}
    Let $\psi$ be a local Arthur parameter as in \eqref{A-param decomp}. We say that $\psi$ is of \emph{good parity} if $\psi \in \Psi(G)$, i.e., $x_i=0$ for all $i$, and every summand $\rho \otimes S_{a_i} \otimes S_{b_i}$ is self-dual and orthogonal.
We let $\Psi_{gp}(G)$ denote the subset of $\Psi^+(G)$ consisting of local Arthur parameters of good parity.
\end{defn}
We explicate this condition further. 
Consider a summand $\rho|\cdot|^x\otimes S_a \otimes S_b$ of $\psi$ as in \eqref{A-param decomp}. This summand is self-dual and orthogonal if and only if $x=0,$ $\rho$ is orthogonal (resp. symplectic), and $a_i+b_i$ is even (resp. odd).

Let $\psi\in\Psi^+(G).$ Since $\psi$ is self-dual, we have a decomposition 
\[
\psi=\psi_{ngp} + \psi_{gp}+\psi_{ngp}^\vee,
\]
where $\psi_{ngp}^\vee$ denotes the dual of $\psi_{ngp}$, $\psi_{gp}\in\Psi_{gp}(G),$ and $\psi_{gp}$ is maximal for this decomposition, i.e., if we decompose $\psi_{ngp}+\psi_{ngp}^\vee$ as in \eqref{A-param decomp}, then any irreducible summand $\rho|\cdot|^x\otimes S_a\otimes S_b$ is not of good parity. Note that $\psi_{gp}$ is uniquely determined by this decomposition, but $\psi_{ngp}$ is not necessarily unique. M{\oe}glin showed that the local Arthur packet $\Pi_\psi$ can be constructed from $\Pi_{\psi_{gp}}.$

\begin{thm}[{\cite[Proposition 5.1]{Moe11b}}]\label{thm red to gp}
Let $\psi\in\Psi^+(G)$ with decomposition $\psi=\psi_{ngp}+\psi_{gp}+\psi_{ngp}^\vee$ as above. Then, there exists $\tau\in\Pi(\GL_d(F))$ (determined by $\psi_{ngp}$) such that for any $\pi_{gp}\in\Pi_{\psi_{gp}},$ the normalized parabolic induction $\tau\rtimes\pi_{gp}$ is irreducible and \begin{equation}\label{non-unitary A-packet}
    \Pi_\psi=\{\tau\rtimes\pi_{gp} \ | \ \pi_{gp}\in\Pi_{\psi_{gp}}\}.
\end{equation}
\end{thm}

\subsection{The Adams conjecture}\label{sec the Adams conjecture}

In this subsection, we recall the Adams conjecture (see Conjecture \ref{conj Adams} below) along with some relevant results. 

Following \cite[\S3.2]{GI14}, we fix a pair of characters $\chi_W,\chi_V$ associated to $W_n$ and $V_m^\pm$ respectively. More specifically, we have that $\chi_W$ is the trivial character of $F^\times$ and $\chi_V$ is the quadratic character associated to $F(\sqrt{d})/F$, where $d$ is the discriminant of $V_m^\pm$. Recall also that $\alpha=2m-2n-1.$ We recall the Adams conjecture below.

\begin{conj}[{\cite{Ada89}}]\label{conj Adams}
    Assume that $m>n.$
    Suppose that $\pi\in\Pi(G_n)$ lies in a local Arthur packet $\Pi_\psi$ for some $\psi\in\Psi^+(G_n).$ If $\theta_{-\alpha}^\pm(\pi)\neq 0,$ then $\theta_{-\alpha}^\pm(\pi)\in\Pi_{\psi_\alpha}$ where
    \begin{equation}\label{eqn psi_alpha}
    \psi_\alpha=
       (\chi_W\chi_V^{-1}\otimes\psi)\oplus \chi_W\otimes S_1\otimes S_\alpha.
    \end{equation}
\end{conj}

M{\oe}glin verified that when $\alpha$ is large, the Adams conjecture is true.
\begin{thm}[{\cite[Theorem 6.1]{Moe11c}}]\label{thm Moeglin Adams}
Suppose that $\pi\in\Pi(G_n)$ lies in a local Arthur packet $\Pi_\psi$ for some $\psi\in\Psi^+(G_n).$
    For $\alpha\gg 0,$ we have $\theta_{-\alpha}^\pm(\pi)\in\Pi_{\psi_\alpha}$.
\end{thm}
However, M{\oe}glin also showed that the Adams conjecture can and does fail in general. The failure of the Adams conjecture in this case is well-understood through the works of \cite{BH22, Haz24}. We are particularly concerned with the case of the going-up tower for $\pi$ in which case Baki{\' c} and Hanzer showed that the Adams conjecture always holds.
\begin{thm}[{\cite[Theorem 2]{BH22}}]\label{thm Adams going up tower}
Suppose that $\pi\in\Pi(G_n)$ lies in a local Arthur packet $\Pi_\psi$ for some $\psi\in\Psi^+(G_n).$
    If $\theta_{-\alpha}^{\up}(\pi)\neq 0,$ then $\theta_{-\alpha}^{\up}(\pi)\in\Pi_{\psi_\alpha}$.
\end{thm}
We remark that while \cite[Theorem 2]{BH22} is stated for $\psi\in\Psi(G_n)$; however, the extension to $\Psi^+(G_n)$ follows directly from \cite[Lemma 2.33]{Haz24}.

\begin{rmk}\label{rmk equal first occurences}
    We remark that if $m^{up}(\pi)=m^{down}(\pi)$ then the Adams conjecture is always true. That is, for any odd positive integer $\alpha$ and $\pi\in\Pi_\psi,$ we have $\theta_{-\alpha}^\pm(\pi)\in\Pi_{\psi_\alpha}$ (see \cite[p. 15]{BH22}).
\end{rmk}

\section{Extended multi-segments}\label{sec extended multi-segments}

Let $G\in\{G_n, H_m^\pm\}.$
From Theorem \ref{thm red to gp}, it is desirable to parameterize $\Pi_\psi$ for $\psi\in\Psi_{gp}(G).$ In this section, we recall such a parameterization using extended multi-segments (Theorem \ref{thm Arthur packets extended multi-segment parameterization}). Furthermore, we also recall some results on the theory of intersections of local Arthur packets from \cite{HLL22, HLL25}. We begin by recalling some notions related for extended multi-segments.

We fix the following notation throughout this subsection. Let $\psi\in\Psi_{gp}(G)$ with decomposition 
\[ \psi= \bigoplus_{\rho} \bigoplus_{i \in I_{\rho}} \rho \otimes S_{a_i} \otimes S_{b_i}. \]
We set $A_i=\frac{a_i+b_i}{2}-1$ and $B_i=\frac{a_i-b_i}{2}$ for $i\in I_\rho.$ 
 
We say that a total order $>_\psi$ on $I_\rho$ is \emph{admissible} if it satisfies:
\[
\tag{$P$}
\text{
For $i,j \in I_\rho$, 
if $A_i > A_j$ and $B_i > B_j$, 
then $i >_\psi j$.
}
\]
We often consider an order $>_\psi$ on $I_\rho$ satisfying:
\[
\tag{$P'$}
\text{
For $i,j \in I_\rho$, 
if $B_i > B_j$, 
then $i >_\psi j$.
}
\]
Note that ($P'$) implies ($P$). For brevity, we often write $>$ instead of $>_\psi$ when it is clear that we are working with a fixed admissible order. 

Suppose now that we have fixed an admissible order for $\psi.$ The \emph{support} of $\psi$ is the collection of ordered multi-sets 
$$\supp(\psi) := \cup_{\rho}\{ [A_i,B_i]_{\rho} \}_{i \in (I_\rho,>)}.
$$
Note that $\supp(\psi)$ depends implicitly on the fixed admissible order.

We recall the definition of extended multi-segments.

\begin{defn}(Extended multi-segments)\label{def multi-segment}
\begin{enumerate}
\item
An \emph{extended segment} is a triple $([A,B]_\rho, l, \eta)$,
where
\begin{itemize}
\item
$[A,B]_\rho = \{\rho|\cdot|^A, \rho|\cdot|^{A-1}, \dots, \rho|\cdot|^B \}$ is a segment 
for an irreducible unitary supercuspidal representation $\rho$ of some $\GL_d(F)$; 
\item
$l \in \Z$ with $0 \leq l \leq \frac{b}{2}$, where $b = \#[A,B]_\rho = A-B+1$; 
\item
$\eta \in \{\pm1\}$ unless $2l=b$ in which case $\eta=1$. 
\end{itemize}

\item
An \emph{extended multi-segment} for $G$ is 
an equivalence class (via the equivalence defined below) of multi-sets of extended segments 
\[
\EE = \cup_{\rho}\{ ([A_i,B_i]_{\rho}, l_i, \eta_i) \}_{i \in (I_\rho,>)}
\]
such that the following holds.
\begin{itemize}
\item
$I_\rho$ is a totally ordered finite set with a fixed admissible total order $>$.

\item
$A_i + B_i \geq 0$ for all $\rho$ and $i \in I_\rho$.

\item
We have
\[
\psi_{\EE} = \bigoplus_\rho \bigoplus_{i \in I_\rho} \rho \otimes S_{a_i} \otimes S_{b_i}. 
\]
where $(a_i, b_i) = (A_i+B_i+1, A_i-B_i+1)$,
is a local Arthur parameter for $G$ of good parity.
\item The sign condition
\begin{align}\label{eq sign condition}
\prod_{\rho} \prod_{i \in I_\rho} (-1)^{[\frac{b_i}{2}]+l_i} \eta_i^{b_i} = \epsilon_{G}
\end{align}
holds. Here $\epsilon_{G}=1$ if $G=\Sp_{2n}(F)$ or $G=\OO_{2m}^+(F)$ and $\epsilon_{G}=-1$ otherwise.
\end{itemize}

\item
We define the \emph{support} of $\EE$ to be the collection of ordered multi-sets 
\[
\supp(\EE) = \cup_{\rho}\{ [A_i,B_i]_{\rho} \}_{i \in (I_\rho,>)}.
\]
We implicitly include the admissible order $>$ in $\supp(\EE).$
\item We let $\Eseg(G)$ denote the set of all extended multi-segments of $G$ up to weak equivalence.
\end{enumerate}
\end{defn}
If the admissible order $>$ is clear in the context, for $k \in I_{\rho}$, we often let $k+1 \in I_{\rho}$ be the unique element adjacent with $k$ and $k+1>k$.

The data in an extended multi-segment can be cumbersome to list out in detail. Instead, 
we attach a symbol to each extended multi-segment by the same way in \cite[Section 3]{Ato20b}. We give an example to explain this.
\begin{exmp}\label{exmp symbol}
 Let $\rho$ be the trivial representation of $\GL_1(F)$. The symbol
\[\EE=\bordermatrix{
& -1 & 0 &1 & 2 &3 & 4\cr
& \lhd & \lhd & \oplus & \ominus & \rhd & \rhd \cr
&  &  &  & \lhd &\rhd  &  \cr
&  &  &  &  &  & \ominus \cr
}_{\rho}\]
corresponds to $\EE= \{ ([A_i,B_i]_{\rho},l_i,\eta_i)\}_{i \in (1<2<3)}$ of $\Sp_{45}(F)$ where the data is given as follows.
\begin{itemize}
    \item  $([A_1,B_1]_{\rho},[A_2,B_2]_{\rho},[A_3,B_3]_{\rho})=([4,-1]_{\rho},[3,2]_{\rho},[4,4]_{\rho})$ specifies the support of each row.
    \item  $(l_1,l_2,l_3)=(2,1,0)$ counts the number of pairs of triangles in each row. 
    \item  $(\eta_1,\eta_2,\eta_3)=(1, 1, -1)$ records the sign of the first circle in each row. Note that when there are no circles, we have $2l=b$ and so $\eta=1$ by definition.
\end{itemize}
The associated local Arthur parameter is recovered from the supports
    \[ \psi_{\EE}= \rho \otimes S_{4}\otimes S_{6} + \rho \otimes S_{6}\otimes S_{2} + \rho \otimes S_9 \otimes S_1.  \]
\end{exmp}

With the above symbol in mind, we often say that an extended segment $r=([A,B]_\rho,l,\eta)$ is a \emph{row} of $\EE$ if $r\in \EE.$ Furthermore, we define the support of $r$ to be $\supp(r)=[A,B]_\rho$ and let $A(r)=A$, $B(r)=B,$ $a(r)=A(r)+B(r)+1,$ $b(r)=A(r)-B(r)+1,$ $l(r)=l$, and $\eta(r)=\eta.$

Let $\EE\in\Eseg(G).$ If $G=\Sp_{2n}(F),$ then we attach a representation $\pi(\EE)$ of $G$ as in \cite[\S3.2]{Ato20b}. This is done using the theory of derivatives developed by Atobe and M{\'i}nguez in \cite{AM23}. We have that either $\pi(\EE)$ vanishes or $\pi(\EE)\in\Pi(G).$ If $G=\OO_{2m}^\pm(F),$ then we attach a representation $\pi(\EE)$ of $G$ as in \cite[Definition 6.3]{HLL25}. This definition avoids the use of derivatives by using the Adams conjecture as established by M{\oe}glin (see Theorem \ref{thm Moeglin Adams}).
Again, we have that either $\pi(\EE)$ vanishes or $\pi(\EE)\in\Pi(G).$

In either case, we have that extended multi-segments parameterize local Arthur packets.
\begin{thm}\label{thm Arthur packets extended multi-segment parameterization}
    Suppose $\psi= \bigoplus_{\rho} \bigoplus_{i \in I_{\rho}} \rho \otimes S_{a_i} \otimes S_{b_i}$ is a relevant local Arthur parameter of good parity of $G$. Choose an admissible order $>$ on $I_{\rho}$ for each $\rho$ that satisfies ($P'$) if $\frac{a_i-b_i}{2}<0$ for some $i \in I_{\rho}$. Then
\[ \bigoplus_{\pi \in \Pi_{\psi}} \pi= \bigoplus_{\EE} \pi(\EE),\]
where $\EE$ runs over all extended multi-segments with $\supp(\EE)= \supp(\psi)$ and $\pi(\EE) \neq 0$. 
\end{thm}
When $G=\Sp_{2n}(F),$ the above theorem was proven by Atobe (\cite[Theorem 3.3]{Ato20b}). When $G=\OO_{2m}^\pm(F)$, the above theorem was proven in \cite[Theorem 6.10]{HLL25}.
As a direct consequence of the above theorem and the fact that local Arthur packets are multiplicity-free (\cite{Moe11a}), we obtain the following corollary.
\begin{cor}\label{cor same support}
    Let $\EE_1, \EE_2\in\Eseg(G)$ and suppose that $\supp(\EE_1)=\supp(\EE_2).$ If $\frac{a_i-b_i}{2}<0$ for some $i \in I_{\rho}$, then we also require that the order on $I_\rho$ is $(P').$ If $\pi(\EE_1)=\pi(\EE_2)\neq 0,$ then $\EE_1=\EE_2.$
\end{cor}
Recall that $\supp(\EE)$ implicitly records the admissible order on $\EE\in\Eseg(G)$. Thus, the hypothesis $\supp(\EE_1)=\supp(\EE_2)$ in the above corollary also asserts that the admissible orders on $\EE_1$ and $\EE_2$ agree.

\subsection{Intersections}

In this subsection, we recall the theory of intersections of local Arthur packets as developed in \cite{HLL22, HLL25}. We begin by recalling various operators which are used in the classification of these intersections (see Theorem \ref{thm intersections of local Arthur packets}), along with some other useful operators and their properties.

We note that the effect of an operator often only depends on a fixed $\rho.$ To simplify the definitions of the operators, we introduce the following notation.

\begin{defn}
Let $\EE= \cup_{\rho} \{([A_i,B_i]_{\rho}, l_i, \eta_i)\}_{i \in (I_{\rho},>)}\in\Eseg(G).$
We set
\[ \EE_{\rho}=\{([A_i,B_i]_{\rho}, l_i, \eta_i)\}_{i \in (I_{\rho},>)},\  \EE^{\rho}=\cup_{\rho' \not\cong \rho}\{([A_i,B_i]_{\rho'}, l_i, \eta_i)\}_{i \in (I_{\rho'},>)}. \]
We let $\Vseg(G)$ denote the set of elements of the form \[\FF=\cup_{\rho'} \EE_{\rho'}\]
where the union is over a finite set of irreducible self-dual supercuspidal representations $\rho'$ of $\GL_d(F),$ $d\geq 1,$ and $\EE\in\Eseg(G).$ Essentially, $\mathcal{F}$ is an extended multi-segment (for some group) except that we not enforce the sign condition \eqref{eq sign condition}. We say that $\mathcal{F}$ is a \emph{virtual extended multi-segment}.
\end{defn}

Next, we recall the shift and add operators. 

\begin{defn}\label{defn shift and add}
Let $\EE = \cup_{\rho}\{ ([A_i,B_i]_{\rho}, l_i, \eta_i) \}_{i \in (I_\rho,>)}$ be an extended multi-segment. For $j \in I_{\rho'}$ and $d \in \Z$, we define the following operators: 

\begin{enumerate}
    \item [1.] $sh^{d}(\EE)= \cup_{\rho}\{ ([A_i+d,B_i+d]_{\rho}, l_i, \eta_i) \}_{i \in (I_\rho,>)}$, and
     \item [2.] $add_j^{d}(\EE)= \cup_{\rho}\{ ([A_i',B_i']_{\rho}, l_i', \eta_i) \}_{i \in (I_\rho,>)}$ with 
    \[ ([A_i',B_i']_{\rho},l_i')= \begin{cases}
    ([A_i+d,B_i-d]_{\rho},l_i+d) & \text{ if }\rho=\rho' \text{ and } i = j,\\
     ([A_i,B_i]_{\rho},l_i) & \text{ otherwise. }\end{cases} 
    \]
\end{enumerate}
We use these notations in the case that the resulting object is still an extended multi-segment.
\end{defn}

The next operator we recall is the row exchange operator. Its effect is to change the admissible order on an extended multi-segment. We first recall the definition of a symbol.
\begin{defn}\label{def symbol}
  A \emph{symbol} is a multi-set of extended segments 
\[
\EE = \cup_{\rho}\{ ([A_i,B_i]_{\rho}, l_i, \eta_i) \}_{i \in (I_\rho>)},
\] which satisfies the same conditions in Definition \ref{def multi-segment}(2) except we drop the condition $0 \leq l_i \leq \frac{b_i}{2}$, for each $i \in I_{\rho}$.
\end{defn}

Any change of admissible orders can be derived from a composition of the row exchange operators $R_k$ which we recall from \cite[Definition 3.15]{HLL22}.

\begin{defn}[Row exchange]\label{def row exchange} 
Suppose $\EE$ is a symbol where
$$\EE_{\rho}=\{([A_i,B_i]_{\rho},l_i,\eta_i)\}_{i \in (I_{\rho},>)}.$$
For $k<k+1 \in I_{\rho}$, let $\gg$ be the total order on $I_\rho$ defined by $k\gg k+1$ and if $(i,j)\neq (k,k+1)$, then $ i \gg j$ if and only if $
i >j .$ 

Suppose $\gg$ is not an admissible order on $I_{\rho}$, then we define $R_k(\EE)=\EE$. Otherwise, we define 
\[R_{k}(\EE_{\rho})=\{([A_i,B_i]_{\rho},l_i',\eta_i')\}_{i \in (I_{\rho},\gg)},\]
where $( l_i',\eta_i')=(l_i,\eta_i)$ for $i \neq k,k+1$, and $(l_k',\eta_k')$ and $(l_{k+1}', \eta_{k+1}')$ are given as follows: Denote $\epsilon=(-1)^{A_k-B_k}\eta_k\eta_{k+1}$.
\begin{enumerate}
    \item [Case 1.] $ [A_k,B_k]_{\rho} \supset [A_{k+1},B_{k+1}]_{\rho}$:
    
    In this case, we set $(l_{k+1}',\eta_{k+1}')=(l_{k+1}, (-1)^{A_k-B_k}\eta_{k+1})$, and
    \begin{enumerate}
    \item [(a)] If $\epsilon=1$ and $b_k- 2l_k < 2(b_{k+1}-2l_{k+1})$, then
    \[ (l_k', \eta_{k}')= (b_k-(l_k+ (b_{k+1}-2l_{k+1})), (-1)^{A_{k+1}-B_{k+1}} \eta_k).  \]
    \item [(b)] If $\epsilon=1$ and $b_k- 2l_k \geq  2(b_{k+1}-2l_{k+1})$, then
    \[ (l_{k}', \eta_{k}')= (l_k+ (b_{k+1}-2l_{k+1}), (-1)^{A_{k+1}-B_{k+1}+1} \eta_k).  \]
    \item [(c)] If $\epsilon=-1$, then
    \[ (l_{k}', \eta_{k}')= (l_k- (b_{k+1}-2l_{k+1}), (-1)^{A_{k+1}-B_{k+1}+1} \eta_k).  \]
\end{enumerate}
    \item [Case 2.] $ [A_k,B_k]_{\rho} \subset [A_{k+1},B_{k+1}]_{\rho}$:
    
    In this case, we set $(l_{k}',\eta_{k}')=(l_{k}, (-1)^{A_{k+1}-B_{k+1}}\eta_{k})$, and
    \begin{enumerate}
   \item [(a)] If $\epsilon=1$ and $b_{k+1}- 2l_{k+1} < 2(b_{k}-2l_{k})$, then
    \[ (l_{k+1}', \eta_{k+1}')= (b_{k+1}-(l_{k+1}+ (b_{k}-2l_{k})), (-1)^{A_{k}-B_{k}} \eta_{k+1}).  \]
    \item [(b)] If $\epsilon=1$ and $b_{k+1}- 2l_{k+1} \geq  2(b_{k}-2l_{k})$,
    then
    \[ (l_{k+1}', \eta_{k+1}')= (l_{k+1}+ (b_{k}-2l_{k}), (-1)^{A_{k}-B_{k}+1} \eta_{k+1}).  \]
    \item [(c)] If $\epsilon=-1$, then
    \[ (l_{k+1}', \eta_{k+1}')= (l_{k+1}- (b_{k}-2l_{k}), (-1)^{A_{k}-B_{k}+1} \eta_{k+1}).  \]
\end{enumerate}
\end{enumerate}
Finally, we define $R_{k}(\EE)= \EE^{\rho} \cup R_{k}(\EE_{\rho})$.
\end{defn}

We remark that there is another definition of row exchange is given in \cite[Section 4.2]{Ato20b}; however, these definitions agree when $\pi(\EE)\neq 0.$ 

The next operator we recall is known as union-intersection.
\begin{defn}[union-intersection]\label{ui def}
 Let $\EE\in\Eseg(G)$. For $k< k+1 \in I_{\rho}$, we define an operator $ui_k$, called union-intersection, on $\EE$ as follows. Write 
 \[ \EE_{\rho}= \{([A_i,B_i]_\rho, l_i,\eta_i)\}_{i \in (I_{\rho},>)}.\]
  Denote $\epsilon=(-1)^{A_k-B_k}\eta_k \eta_{k+1}.$ If $A_{k+1}>A_k$, $B_{k+1}>B_k$ and any of the following cases holds:
\begin{enumerate}
    \item [{Case 1}.] $ \epsilon=1$ and $A_{k+1}-l_{k+1}=A_k-l_k,$
    \item [{Case 2}.] $ \epsilon=1$ and $B_{k+1}+l_{k+1}=B_k+l_k,$
    \item [{Case 3}.] $ \epsilon=-1$ and $B_{k+1}+l_{k+1}=A_k-l_k+1,$
\end{enumerate}
we define
\begin{align*}
     ui_{k}(\EE_{\rho})=\{ ([A_i',B_i']_{\rho},l_i',\eta_i')\}_{i \in (I_{\rho}, >)},
\end{align*} 
where $ ([A_i',B_i']_{\rho},l_i',\eta_i')=([A_i,B_i]_{\rho},l_i,\eta_i)$ for $i \neq k,k+1$, and $[A_k',B_k']_{\rho}=[A_{k+1},B_k]_{\rho}$, $[A_{k+1}',B_{k+1}']_{\rho}=[A_k,B_{k+1}]_{\rho}$, and $( l_k', \eta_k', l_{k+1}',\eta_{k+1}' )$ are given case by case as follows:
\begin{enumerate}
    \item[$(1)$] in Case 1, $( l_k', \eta_k', l_{k+1}',\eta_{k+1}' )= (l_k,\eta_k, l_{k+1}-(A_{k+1}-A_k), (-1)^{A_{k+1}-A_k}\eta_{k+1})$;
    \item [$(2)$] in Case 2, if $b_k-2l_k \geq A_{k+1}-A_k$, then
    \[( l_k', \eta_k', l_{k+1}',\eta_{k+1}' )= (l_k+(A_{k+1}-A_k),\eta_k, l_{k+1}, (-1)^{A_{k+1}-A_k}\eta_{k+1}),\]
    if $b_k-2l_k < A_{k+1}-A_k$, then
    \[( l_k', \eta_k', l_{k+1}',\eta_{k+1}' )= (b_k-l_k,-\eta_k, l_{k+1}, (-1)^{A_{k+1}-A_k}\eta_{k+1});\]
    \item [$(3)$] in Case 3, if $l_{k+1} \leq  l_k$, then
    \[( l_k', \eta_k', l_{k+1}',\eta_{k+1}' )= (l_k,\eta_k, l_{k+1}, (-1)^{A_{k+1}-A_k}\eta_{k+1}),\]
    if $l_{k+1}> l_{k}$, then
    \[( l_k', \eta_k', l_{k+1}',\eta_{k+1}' )= (l_k,\eta_k, l_{k}, (-1)^{A_{k+1}-A_k+1}\eta_{k+1});\]
    \item [$(3')$] if we are in Case 3 and $l_k=l_{k+1}=0$, then we delete $ ([A_{k+1}',B_{k+1}']_{\rho},l_{k+1}',\eta_{k+1}')$ from $ui_k(\EE_{\rho})$.
\end{enumerate}
Otherwise, we define $ui_k(\EE_{\rho})=\EE_{\rho}$. In any case, we define $ui_k(\EE)= \EE^{\rho} \cup ui_k(\EE_{\rho})$.

We say $ui_k$ is applicable on $\EE$ or $\EE_{\rho}$ if $ui_k(\EE)\neq \EE$. We say this $ui_k$ is of type 1 (resp. 2, 3, 3') if $\EE_{\rho}$ is in Case 1 (resp. 2, 3, 3'). 
\end{defn}

We also consider the composition of the row exchange and union-intersection operators as follows.

\begin{defn} \label{def ui}
Suppose $\EE\in\Eseg(G)$ and write
\[\EE_{\rho}=\{ ([A_i,B_i]_{\rho},l_i,\eta_i)\}_{i\in (I_{\rho,>})}.\] 
Given $i,j \in I_{\rho}$, we define $ui_{i,j}(\EE_{\rho})=\EE_{\rho}$ unless
\begin{enumerate}
    \item [1.] We have $ A_i< A_j$, $B_i <B_j$ and $(j,i,>')$ is an adjacent pair for some admissible order $>'$ on $I_{\rho}$. 
    \item [2.] $ui_i$ is applicable on $\EE_{\rho,>'}$
\end{enumerate}
In this case, we define $ui_{i,j}(\EE_{\rho}):=(ui_{i}(\EE_{\rho,>'}))_{>}$, so that the admissible order of $ui_{i,j}(\EE_{\rho})$ and $\EE_{\rho}$ are the same. (If the $ui_i$ is of type 3', then we delete the $j$-th row.) Finally, we define $ui_{i,j}(\EE)= \EE^{\rho} \cup ui_{i,j}(\EE_{\rho})$.

We say $ui_{i,j}$ is applicable on $\EE$ if $ui_{i,j}(\EE) \neq \EE$. Furthermore, we say that $ui_{i,j}$ is of type 1, 2, 3, or 3' if the operation $ui_i$ is of type 1, 2, 3, or 3', respectively, in Definition \ref{ui def}.
\end{defn}

The next operator is called the dual operator.
  
\begin{defn}[dual]\label{def dual}
Let $\EE= \cup_\rho \{([A_i,B_i]_{\rho},l_i,\eta_i)\}_{i\in (I_\rho, >)}$ be an extended multi-segment such that the admissible order $>$ on $I_{\rho}$ satisfies (P') for all $\rho$. We define 
$$dual(\EE)=\cup_{\rho}\{([A_i,-B_i]_{\rho},l_i',\eta_i')\}_{i\in (I_\rho, >')}$$ as follows:
\begin{enumerate}
    \item The order $>'$ is defined by $i>'j$ if and only if $j>i.$ 
    \item We set \begin{align*}
l_i'=\begin{cases}
l_i+B_i  & \mathrm{if} \, B_i\in\mathbb{Z},\\
 l_i+B_i+\frac{1}{2}(-1)^{\alpha_{i}}\eta_i  & \mathrm{if} \, B_i\not\in\mathbb{Z},
\end{cases}
\end{align*}
and
\begin{align*}
\eta_i'=\begin{cases}
(-1)^{\alpha_i+\beta_i}\eta_i  & \mathrm{if} \, B_i\in\mathbb{Z},\\
 (-1)^{\alpha_i+\beta_i+1}\eta_i  & \mathrm{if} \, B_i\not\in\mathbb{Z},
\end{cases}
\end{align*}
where $\alpha_{i}=\sum_{j\in I_\rho, j<i}a_j,$ and $\beta_{i}=\sum_{j\in I_\rho, j>i}b_j,$ $a_j=A_j+B_j+1$, $b_j=A_j-B_j+1$.
\item When $B_i\not\in\mathbb{Z}$ and $l_i=\frac{b_i}{2}$, we set $\eta_i=(-1)^{\alpha_i+1}.$
\end{enumerate}
If $\FF= \EE_{\rho}$, we define $dual(\FF):= (dual(\EE))_{\rho}$.

As a shorthand, for each $i\in I_{\rho},$ let $r_i=([A_i,B_i]_{\rho},l_i,\eta_i)$. Then we let $\widehat{r}_i$ denote the effect of the dual operator on $\EE$ on this row, i.e., 
\[
\widehat{r}_i=([A_i,-B_i]_{\rho},l_i',\eta_i').\]
\end{defn} 

We note that if $\pi(\EE)\neq 0,$ then $\pi(dual(\EE))$ is the Aubert-Zelevinsky dual of $\pi(\EE)$ (\cite[Theorem 6.2]{Ato20b} if $G=\Sp_{2n}(F)$ and \cite[Proposition 6.11]{HLL25} if $G=\OO_{2m}^\pm(F)$). For our purposes, it is sufficient to use the following direct implication.

\begin{lemma}\label{lemma dual}
    Let $\EE\in\Eseg(G)$ be such that $\pi(\EE)\neq 0.$ Then $\pi(dual(\EE))\neq 0.$ Moreover, if $\EE'\in\Eseg(G)$ is such that $\pi(\EE')=\pi(\EE)$, then $\pi(dual(\EE'))=\pi(dual(\EE)).$
\end{lemma}

From \cite{HLL22, HLL25}, we understand the inverse of union-intersections not of type 3'.

\begin{lemma}\label{lemma ui inv = dud}
   If $ui$ is not of type 3', then its inverse is of the form $dual\circ ui\circ dual$ of the same type. 
\end{lemma}

In the case that some $B_i\in\frac{1}{2}+\mathbb{Z}$, another operator, called the partial dual operator and denoted by $dual_k$ (see \cite[Definition 6.5]{HLL22}), is also needed. However, for our applications we are only concerned with the cases where $B_i\in\mathbb{Z}$ and so this operator does not play a significant role for us and we omit its definition. We let $dual_k^-$ denote the inverse of $dual_k$.

We distinguish certain operators that have a key role in the intersection of local Arthur packets (see Theorem \ref{thm intersections of local Arthur packets} below).

\begin{defn}\label{def basic operators}
    The operators $R_k$, $ui_{i,j},$ $dual\circ ui_{j,i}\circ dual$, or $dual_k$ are known as \emph{basic operators}.
\end{defn}

Basic operators fully determine intersections of local Arthur packets as follows.

\begin{thm}\label{thm intersections of local Arthur packets}
    Let $\EE\in\Eseg(G)$ be such that $\pi(\EE)\neq 0.$ Then the following hold.
    \begin{enumerate}
        \item If $T$ is a basic operator, then $\pi(\EE)=\pi(T(\EE)).$
        \item If $\EE'$ is another extended multi-segment for which $\pi(\EE)=\pi(\EE'),$ then $\EE$ and $\EE'$ are related by a finite composition of basic operators and their inverses.
    \end{enumerate}
\end{thm}

We remark the above theorem was proven when $G=\Sp_{2n}(F)$ in \cite[Theorem 1.4]{HLL22} and when $G=\OO_{2m}^\pm(F)$ in \cite[Theorems 6.13 and 6.15]{HLL25}. With the above theorem in mind, for $\EE,\EE'\in\Eseg(G),$ we write $\EE\sim\EE'$, and say that they are (strongly) equivalent, if $\EE$ and $\EE'$ are related by a finite composition of basic operators and their inverses.

A special class of the basic operators are called raising operators. They elucidate a structure on the set
\[
\Psi(\pi)=\{\psi\in\Psi(G) \ | \ \pi\in\Pi_\psi\}.
\]

\begin{defn}\label{defn raising operators}
    The operators $dual\circ ui \circ dual,$ $ui^{-1}$, and $dual_k^-$ are called \emph{raising operators}. Given local Arthur parameters $\psi_1, \psi_2\in\Psi(\pi)$, we write $\psi_1\geq_O\psi_2$ if there exists a sequence of raising operators $(T_i)_{i=1}^l$ such that \[\EE_1=(T_l\circ T_{l-1}\circ\cdots\circ T_1)(\EE_2),\] where $\EE_j\in\Eseg(G)$ is such that $\psi_{\EE_j}=\psi_j$ and $\pi(\EE_j)=\pi$ for $j=1,2.$ 
\end{defn}

Remarkably, the partial order $\geq_O$ on $\Psi(\pi)$ has unique maximal and minimal elements.

\begin{thm}\label{thm psi max min}
    Let $\pi\in\Pi(G)$ be of Arthur type. Then there exists unique maximal and minimal elements of $\Psi(\pi)$ with respect to $\geq_O.$ We denote these elements by $\psi^{max}(\pi)$ and $\psi^{min}(\pi)$, respectively.
\end{thm}

The above theorem was proven when $G=\Sp_{2n}(F)$ in \cite[Theorem 1.7]{HLL22} and when $G=\OO_{2m}^\pm(F)$ in \cite[Theorem 7.3]{HLL25}.

Recall that it is possible that $\pi(\EE)=0.$ However, there is a combinatorial description of when $\pi(\EE)\neq 0$. See \cite[Theorems 3.6, 4.4]{Ato20b} for $G=\Sp_{2n}(F)$ and \cite[Theorem 6.17]{HLL25} for $G=\OO_{2m}^\pm(F)$.
We say that $\EE\in\Rep(G)$ if $\pi(\EE)\neq 0$. Similarly, we let $\VRep(G)$ denote the set of $\mathcal{F}=\cup_{\rho'}\EE_{\rho'}\in \Vseg(G)$ such that $\EE\in\Rep(G).$ Applying the combinatorial nonvanishing conditions (\cite[Theorems 3.6, 4.4]{Ato20b} or \cite[Theorem 6.17]{HLL25}) to $\EE$, we have that $\mathcal{F}\in \VRep(G)$ if and only if $\mathcal{F}$ formally satisfies the combinatorial nonvanishing conditions. Furthermore, we set $\Vseg_\rho(G)=\{\EE_\rho \ | \ \EE\in \Vseg(G)\}$ and $\VRep_\rho(G)=\{\EE_\rho \ | \ \EE\in \VRep(G)\}$. Again, we have that $\mathcal{E}\in \VRep_\rho(G)$ if and only if $\mathcal{E}$ formally satisfies the combinatorial nonvanishing conditions.

\subsection{The Adams conjecture revisited} 
In this subsection, we reconsider the Adams conjecture (Conjecture \ref{conj Adams}) using extended multi-segments. We remark that the Adams conjecture was first understood using M{\oe}glin's parameterization of local Arthur packets (\cite{BH22, Haz24,  Moe11c}); however, these results are reinterpreted using extended multi-segments in \cite{HLL25}. We recall this reformulation here.

\begin{defn}\label{defn EE_alpha}
    Let $\EE\in\Eseg(G_n)$ and write
    \[
    \EE = \cup_{\rho}\{ ([A_i,B_i]_{\rho}, l_i, \eta_i) \}_{i \in (I_\rho>)}.
    \]
    Let $\EE'$ be the extended multisegment defined by replacing each $\rho$ with $\chi_{V^\pm}\inv\rho$.
    Set $\frac{\alpha_0-1}{2}=\max\{A_i \ | \ i \in I_{\chi_{V^\pm}} \}+1$ and let $\alpha\geq \alpha_0.$ In this case, we define
    \[
    \EE_\alpha^\pm=\EE'\cup\left\{\left(\left[\frac{\alpha-1}{2},-\frac{\alpha-1}{2}\right]_{\chi_W}, \frac{\alpha-1}{2},\pm1 \right)\right\}.
    \]
    We remark that the added extended segment should be inserted such that it is the first extended segment in the admissible order. The extended segments in $\EE'$ should be ordered similarly to $\EE.$ Note that $\EE_\alpha^\pm\in\Eseg(H_m^\pm).$
\end{defn} 

The following is a reformulation of M{\oe}glin's result on the Adams conjecture (Theorem \ref{thm Moeglin Adams}).

\begin{thm}[{\cite[Proposition 6.2]{HLL25}}]\label{thm theta lift large}
    Let $\EE\in\Eseg(G_n).$
    If $\pi=\pi(\EE)\neq0$ and $\alpha\gg0$ ($\alpha\geq\alpha_0$ is sufficient), then
    \[
    \theta_{-\alpha}^{\pm}(\pi)=\pi(\EE_\alpha^\pm).
    \]
\end{thm}

This provides a simple way to compute the theta lift when $\alpha$ is large. For $\alpha< \alpha_0$, the calculation is done algorithmically. This is a reformulation of a construction of \cite{BH22}.
\begin{algo}\label{algo compute theta lift}
    Let $\EE\in\Rep(G_n)$ with $\pi=\pi(\EE).$ 
    \begin{enumerate}
        \item If $\alpha\geq\alpha_0,$ then, by Theorem \ref{thm theta lift large}, we have $\theta_{-\alpha}^{\pm}(\pi)=\pi(\EE_\alpha^\pm)$. 
        \item  If $\alpha<\alpha_0$, then we proceed recursively. From $\EE^\pm_{\alpha_0}$, we construct $\EE^\pm_{\alpha_0-2}$ as follows (taking $\beta=\alpha_0$). By definition \ref{defn EE_alpha}, we have  
        \[
    \EE_\beta^\pm=\EE'\cup\left\{\left(\left[\frac{\beta-1}{2},-\frac{\beta-1}{2}\right]_{\chi_W}, \frac{\beta-1}{2},\pm1 \right)\right\}.
    \]
    We define $\EE^\pm_{\beta-2}$ by replacing the added extended segment 
    \[\left(\left[\frac{\beta-1}{2},-\frac{\beta-1}{2}\right]_{\chi_W}, \frac{\beta-1}{2},\pm1 \right)\] 
    with 
    \[\left(\left[\frac{\beta-3}{2},-\frac{\beta-3}{2}\right]_{\chi_W}, \frac{\beta-3}{2},\pm1 \right)\]
    and then row exchanging this extended segment with any extended segment $([A_i,B_i]_{\chi_W},l_i,\eta_i)$ ($i\in I_{\chi_W}$) where $B_i=-\frac{\beta-3}{2}.$ We proceed into Step 3 for the recursion.
    \item We define $\EE^\pm_{\alpha}$ by repeatedly applying Step 2. That is, we construct a sequence  $\EE_{\alpha_0}^\pm,$ $\EE_{\alpha_0-2}^\pm,$ \dots, $\EE_{\alpha}^\pm,$ by repeatedly performing an $add_j^{-1}$ operator (the inverse of the operator $add_j^1$ in Definition \ref{defn shift and add}) on the added extended segment and row exchanging so that it is as maximal as possible in the order.
    \end{enumerate}
\end{algo}

This algorithm computes the theta lift if it is nonzero. From \cite[Theorem 7.5]{HLL25}, we directly obtain the following reformulation of \cite[Theorem A]{BH22}.

\begin{thm}\label{thm compute theta lift}
    Let $\EE\in\Rep(G_n)$, $\alpha$ be any positive odd integer, and $\pi=\pi(\EE).$ If $\pi(\EE_\alpha^\pm)\neq0$, then
    \[
    \theta_{-\alpha}^{\pm}(\pi)=\pi(\EE_\alpha^\pm).
    \]
    In particular, if $\pi(\EE_\alpha^\pm)\neq0$, then $\theta_{-\alpha}^{\pm}(\pi)\in\Pi_{\psi_\alpha}.$
\end{thm}

\subsection{Tempered representations}

In this subsection, we recall some results related to tempered representations. Let $\Pi_{temp}(G)$ denote the subset of $\Pi(G)$ consisting of tempered representations. We begin by recalling the parameterization of $\Pi_{temp}(G)$ via the local Langlands correspondence due to Arthur.

A local Arthur parameter $\psi\in\Psi^+(G)$ is called tempered if $\psi\in\Psi(G)$ and $\psi$ is trivial on the second $\SL_2(\BC),$ i.e., if $\psi=\bigoplus_{i=1}^r \rho_i\otimes S_{a_i}\otimes S_{b_i},$ then $b_i=1$ for any $i=1,\dots,r.$ We let $\Psi_{temp}(G)$ denote the subset of $\Psi^+(G)$ consisting of tempered local Arthur parameters. 

\begin{thm}[{\cite[Theorem 1.5.1]{Art13}}]\label{thm Arthur tempered}
    We have that
    \[\Pi_{temp}(G)=\bigcup_{\psi\in\Psi_{temp}(G)} \Pi_\psi.\]
    Moreover, if $\psi_1,\psi_2\in\Psi_{temp}(G)$ and $\psi_1\neq\psi_2$, then $\Pi_{\psi_1}\cap\Pi_{\psi_2}=\emptyset.$
\end{thm}
In other words, local Arthur packets associated to tempered local Arthur packets partition  $\Pi_{temp}(G)$. Note that from Theorem \ref{thm red to gp}, to construct $\Pi_\psi$ for tempered $\psi$, it suffices to assume that $\psi$ is also of good parity.
With Theorem \ref{thm Arthur packets extended multi-segment parameterization} in mind, we say that $\EE\in\Rep(G)$ is tempered if $\psi_\EE$ is tempered. We remark that $\EE\in\Eseg(G)$ is tempered if and only if
\[
\EE=\cup_\rho\{([A_i,A_i]_\rho,0,\eta_i)\}_{(i\in I_\rho, >)}\in\Eseg(G),
\]
where for any $i,j\in I_\rho$ with $A_i=A_j$, we have $\eta_i=\eta_j$.  In this setting, for $i\in I_\rho$ we write $\eta_\rho(A_i):=\eta_i$ as a shorthand. Note that $\pi(\EE)$ can be tempered even if $\EE$ is not tempered.

We now recall how to compute the first occurrence and the going-up tower for tempered representations using extended multi-segments. The following result is a straightforward adaptation of \cite[Theorem 4.1]{AG17a} to our setting. 

\begin{thm}[{\cite[Theorem 4.1]{AG17a}}]\label{thm going up first occurrence}
    Let $\EE\in\Rep(G_n)$ be tempered and write
    \[
    \EE_{\chi_V}=\cup_{i=1}^r\{([A_i,A_i]_{\chi_V},0,\eta_i)\}.
    \]
    We let $\mathcal{T}$ denote the set containing $-1$ and all odd positive integers $l$ satisfying the following conditions:
    \begin{itemize}
        \item (chain condition) The multi-set $\mathcal{A}=\cup_{i=1}^r\{A_i\}$ contains $\{0,1,\dots,\frac{l-1}{2}\}$;
        \item (oddness condition) the multiplicity of $\frac{i-1}{2}$ in $\mathcal{A}$ is odd for $i=1,3,\dots,l$;
        \item (alternating condition) $\eta_{\chi_V}(A)=-\eta_{\chi_V}(A+1)$ for $A\in\{0,1,\dots,\frac{l-3}{2}\}$.
    \end{itemize}
    Let $l(\pi(\EE))=\max\mathcal{T}.$
    Then $m^{\up}(\pi(\EE))=2n+3+l(\pi).$ Moreover, if $l(\pi)=-1,$ then $m^{\up}(\pi(\EE))=m^{\down}(\pi(\EE))=2n+2.$ Otherwise, $\up=-\eta_{\chi_V}(0).$
\end{thm}

Note that $\chi_V\otimes S_a\otimes S_b$ (and also $\chi_W\otimes S_a\otimes S_b$) is of good parity if and only if $A=\frac{a+b}{2}-1\in\Z$ and $B=\frac{b-a}{2}\in\Z.$ Consequently, with Definition \ref{defn EE_alpha} and the computation of the theta lift (Theorem \ref{thm Adams going up tower}) in mind, we often focus on the set $\Vseg^\Z(G)$ which is defined to be the set of elements $\EE=\cup_\rho\{([A_i,B_i]_\rho,l_i,\eta_i)\}_{(i\in I_\rho, >)}\in\Vseg(G)$ such that $A_i,B_i\in\Z$ for any $i\in I_\rho$ and any $\rho.$ We also set $\VRep^\Z(G)=\VRep(G)\cap\Vseg^\Z(G)$. 

\subsection{Further notation}\label{sec combinatorial nonsense}

Throughout this subsection, we let $\rho$ be an orthogonal supercuspidal representation of some $\GL_d(F)$. In this subsection, we provide a combinatorial extension of the results concerning the Adams conjecture, but replacing $\chi_V$ with $\rho.$

We have $\Vseg_\rho(G_n)=\{\EE_\rho \ | \ \EE\in\Vseg^\Z(G)\}$ and $\VRep_\rho(G_n)=\Vseg_\rho(G_n)\cap\VRep(G).$ To $\EE\in\VRep_\rho(G_n)$, we consider a formal local Arthur parameter $\psi_\EE$ defined analogously to Definition \ref{def multi-segment}(2).

As a direct consequence of \cite[Theorems 3.6, 4.4]{Ato20b}, we have a map between $\VRep_\rho(G_n)$ and $\VRep_{\chi_V}(G_{n'})$ for a suitable $n'$ given by simply replacing $\rho$ by $\chi_V.$ For  $\EE\in\VRep_\rho(G_n),$ let $\EE_{\rho\rightarrow\chi_V}$ denote its image in $\VRep_{\chi_V}(G_n)$ via this bijection. 

Consider $\EE\in\Rep(G_n)$ and recall the definition of $\EE_\alpha^\pm$ from Definition \ref{defn EE_alpha}. Formally, Definition \ref{defn EE_alpha} extends to $\EE\in\VRep_{\chi_V}(G_n).$ Given $\FF\in\VRep_\rho(G_n)$, we define $\FF_\alpha^\pm=(\FF_{\rho\rightarrow\chi_V})_\alpha^\pm.$ Similarly, we extend Algorithm \ref{algo compute theta lift} to define $\FF_\alpha^\pm$ generally.

Formally, Theorem \ref{thm going up first occurrence} applies to any tempered $\EE\in\VRep_{\chi_V}(G_n),$ i.e., the formal local Arthur parameter $\psi_\EE$ is trivial on the second $\SL_2(\BC).$
Suppose now that $\FF\in\VRep_\rho(G_n)$ is tempered. Then we can define the going-up tower (which we record its sign by $\up$) and first occurrence of $\FF$ by defining them to be those of $\FF_{\rho\rightarrow\chi_V}$. We let $\up$ and $m^{\up,\alpha}$ denote the corresponding sign and first occurrence. We further set
\[
\Theta_1(\FF):=\FF^\up_{m^{\up,\alpha}}.
\]
Note that when $\rho=\chi_V,$ $\Theta_1(\FF)$ is simply the $\chi_W$-part of the extended multi-segment associated to the first occurrence on the going-up tower of some extended multi-segment in $\Rep(G_n)$. Moreover, in this case, $\Theta_1(\FF)\in\VRep_{\chi_W}(H_{m^\up}^\up)$ by the Adams conjecture for the going-up tower (Theorems \ref{thm Adams going up tower} and \ref{thm compute theta lift}).

Later, we shall see that there may be further operators applicable on $\Theta_1(\FF)$ which result in virtual extended multi-segments which are not predicted by the Adams conjecture. These arise in a systematic manner and will be denoted by $\Theta_i(\FF)$ for $i=2,3,4$ (e.g., see \S\ref{sec theta results}).

\section{Statement of results}
\label{sec-results}

Throughout this section, we let $\rho$ be an orthogonal supercuspidal representation of some $\GL_d(F)$. 
Given $\EE\in \Vseg_\rho(G_n),$ we write $\EE= \{([A_i, B_i], \ell_i, \eta_i)\}_{i \in I}$ where we implicitly identify $I=\{1,\dots,k\}$ with the order $1<2<\cdots<k.$ That is, we suppress $\rho$ in the notation and identify the admissible order with the standard order.

We recall some notation and definitions from \cite[\S3]{HKT26}. First, we record the first and last columns of the symbols (see Example \ref{exmp symbol}) attached to extended multi-segments as follows. 

\begin{defn}
\label{def-starts-ends}
    Let $\EE\in \VRep_\rho^\mathbb{Z}(G_n)$. We set $c_{\min} := \min_{r \in \EE} B(r)$ and $c_{\max} := \max_{r \in \EE} A(r)$. We say that $\EE$ \emph{starts at column} $c_{\min}$ and that $c_{\min}$ is the \emph{first column}, and similarly that $\EE$ \emph{ends at column} $c_{\max}$ and that $c_{\max}$ is the \emph{last column}. 
\end{defn}

Recall that we are mainly interested in the case that $\EE\in \VRep_\rho^\mathbb{Z}(G_n)$ is tempered, i.e., $\supp(r)$ is a singleton for any row $r\in\EE$. In this case, we keep track of the multiplicities of the columns by adopting the following notation.

\begin{defn}
\label{def-multiplicity}
    Let $\EE\in \VRep_\rho^\mathbb{Z}(G_n)$ and fix $c\in\Z$. We denote by $m_c$ the \emph{multiplicity of $c$ in $\EE$}, which is defined to be the number of times $([c,c], 0, \eta)$ appears in $\EE$ for any $\eta$.
\end{defn}

Our main results and conjectures (and those of \cite{HKT26}) are phrased in terms of blocks and almost-blocks which we recall below. We remark that the definition of a block is largely inspired by Theorem \ref{thm going up first occurrence}.

\begin{defn}
\label{def-block}
    Let $\EE = \{([A_i, B_i], \ell_i, \eta_i)\}_{i \in I}\in \VRep_\rho^\mathbb{Z}(G_n)$ be  tempered (so that $A_i = B_i$ and $\ell_i = 0$). A \emph{block} is a multi-set $\mathcal{B} = \{([A_i, B_i], \ell_i, \eta_i)\}_{i \in J}$ with $J \subset I$ such that
    \begin{itemize}
        \item if $([A, A], 0, \eta_1)\in\mathcal{B}$ and $([A+1, A+1], 0, \eta_2)\in\mathcal{B}$ then $\eta_1 = - \eta_2$,
        \item if for some $\eta\in\{\pm1\}$, we have $([A, A], 0, \eta)\in\mathcal{B}$ and $([A+2, A+2], 0, \eta)\in\mathcal{B}$, then $([A+1, A+1], 0, -\eta)\in\mathcal{B}$,
        \item if $([A, A], 0, \eta)\in\mathcal{B}$ then it appears with an odd multiplicity,
    \end{itemize}
    and $\mathcal{B}$ is maximal for these properties (i.e. if $\mathcal{B}' = \{([A_i, B_i], \ell_i, \eta_i)\}_{i \in J'}$ with $J' \supset J$ also satisfies these conditions, then $J = J'$).

    We say that $\mathcal{B}$ is an \emph{almost-block} if it satisfies the first two conditions and satisfies the third condition for all $([A, A], 0, \eta)$ except possibly when $A$ is maximal among all $([A, A], 0, \eta)$ appearing in $\mathcal{B}$.

    More generally, we say that $\mathcal{B}$ is a block (or almost-block) if there exists some $\EE\in  \VRep_\rho^\mathbb{Z}(G_n)$ for which $\mathcal{B}$ is a block. We let $\Block_\rho(G_n)$, resp. $\ABlock_\rho(G_n)$ denote the set of all blocks, resp. almost-blocks.
\end{defn}

Tempered extended multi-segments decompose uniquely into blocks.

\begin{lemma}[{\cite[Lemma 3.4]{HKT26}}]\label{lem-unique-block-decomp}
    Any tempered  $\EE\in \VRep_\rho(G)$ has a unique decomposition into disjoint blocks.
\end{lemma}

We denote the unique block decomposition by $\BB_1 \cup \dots \cup \BB_k$, where each of the $\BB_i$ lie in $\VRep_\rho^\mathbb{Z}(G_n)$. Then $\EE = \BB_1 \cup \dots \cup \BB_k$ consists of the rows of $\BB_1$, followed by the rows of $\BB_2$, and so on. Note that the decomposition is ordered, in the sense that $\BB_1 \cup \BB_2 \neq \BB_2 \cup \BB_1$.

Next we recall the remove column operator.

\begin{defn}[{\cite[Definition 3.5]{HKT26}}]
    Let $\EE\in  \VRep_\rho(G)$ be tempered. The \emph{remove column} operator $\rc_k$ is defined to be the operator which removes the $k$th column of $\EE$ (counting empty columns). More precisely, $\rc_k(\EE)$ consists of all the extended segments of $\EE$ except those of the form $([k, k], 0, \eta)$ for any $\eta$. If no such extended segments exist, then $\rc_k(\EE) = \EE$.
\end{defn}

Let $\mathcal{B}\in\ABlock(G)$. We let $\Psi(\pi(\mathcal{B}))$ denote the set of formal local Arthur parameters $\psi_{\mathcal{F}}$ where $\mathcal{F}\in\VRep_\rho(G)$ is such that $\mathcal{F}$ and $\mathcal{B}$ are related by a sequence of basic operators. We recall how to determine the cardinality of this set for blocks.

\begin{thm}[{\cite[Theorem 3.6]{HKT26}}]
\label{thm-count-block-temp}
    Let $\mathcal{B}$ be a block starting at $c_{min}$. Let $c_{max}$ be the last column of $\BB$, and let $m_{c_{max}-1}$ be the multiplicity of $c_{max}-1$ in $\mathcal{B}$. Let
    \begin{align*}
        \BB' &:= \rc_{c_{max}}(\BB) \\
        \BB'' &:= \rc_{c_{max}-1}(\BB')
    \end{align*}
    
    First suppose that $c_{min} = 0$. Then 
    \begin{equation} 
    \label{eq first recursion}
    |\Psi(\pi(\mathcal{B}))| = \begin{cases} 3 |\Psi(\pi(\mathcal{B}'))| & \text{if } m_{c_{max}-1} = 1 \\  4 |\Psi(\pi(\mathcal{B}'))| - |\Psi(\pi(\mathcal{B}''))| & \text{if } m_{c_{max}-1} = 3, 5, \dots. \end{cases}
    \end{equation}
    On the other hand, if $c_{min} > 0$, then 
    \begin{equation}
    \label{eq second recursion}
    |\Psi(\pi(\mathcal{B}))| = \begin{cases} 2 |\Psi(\pi(\mathcal{B}'))| & \text{if } m_{c_{max}-1} = 1 \\ 3 |\Psi(\pi(\mathcal{B}'))| - |\Psi(\pi(\mathcal{B}''))| & \text{if } m_{c_{max}-1} = 3, 5, \dots. \end{cases}
    \end{equation}
\end{thm}

Our first main result is a count for the ``theta lift'' of an almost-block.

\begin{thm}[count for theta lifts of almost-blocks]
\label{thm-count-block-theta-temp}
    Let $\BB$ be an almost-block starting at zero. Let $c_{max}$ be the last column of $\BB$, and let $m_{c_{max}}$ its multiplicity in $\mathcal{B}$. Let \[\BB' := \rc_{c_{max}}(\BB).\] Then 
    \[|\Psi(\pi(\Theta_1(\mathcal{B})))| = \begin{cases} 3 |\Psi(\pi(\mathcal{B}))| & \text{if } m_{c_{max}} = 1 \\ 3 |\Psi(\pi(\mathcal{B}))| - |\Psi(\pi(\mathcal{B}'))| & \text{if } m_{c_{max}} = 2, 4, \dots \\ 4 |\Psi(\pi(\mathcal{B}))| - |\Psi(\pi(\mathcal{B}'))| & \text{if } m_{c_{max}} = 3, 5, \dots. \end{cases}\]
\end{thm}

We will prove the above theorem in \S\ref{sec Theta Correspondence for Almost-Blocks}. Note that the case that $\pi$ is supercuspidal (mentioned in \S\ref{sec intro}) is already covered by Theorem \ref{thm-count-block-theta-temp}. Indeed, in this case we have that $\EE_{\chi_V}$ is a block and $m_{c_{\max}}=1$ (for $\EE_{\chi_V}$). From Theorem \ref{thm-count-block-theta-temp} (and Theorems \ref{thm compute theta lift} and \ref{thm going up first occurrence}), we obtain that $|\Psi(\theta_{-m^{\up,\alpha}(\pi)}^\up(\pi))|=3|\Psi(\pi)|$ as stated in \S\ref{sec intro}, but with a different justification.

Next we recall how the numbered of local Arthur parameters containing a fixed tempered representation is determined by its block decomposition.

\begin{thm}[{\cite[Theorem 3.7]{HKT26}}]
\label{thm-count-temp}
    Let $\EE\in\VRep_\rho(G)$ be a tempered extended multi-segment. Suppose that $\EE$ decomposes into blocks $\mathcal{B}_1, \dots, \mathcal{B}_k$ in that order. Then \[|\Psi(\pi(\EE))| = |\Psi(\pi(\BB_1))| \cdot \prod_{i=2}^k |\Psi(\pi(sh^1(\BB_i)))|.\]
\end{thm}

\begin{rmk}
    Note that by combining Equations \eqref{eq first recursion} (to compute $|\Psi(\pi(\BB_1))|$) and \eqref{eq second recursion}  (to compute $|\Psi(\pi(sh^1(\BB_i)))|$), the above theorems provide a complete computation of $|\Psi(\pi(\EE))|$ for any tempered $\EE\in\VRep_\rho(G)$.
\end{rmk}

\begin{rmk}
    In general, given $\EE\in\Rep^\Z(G)=\Rep(G)\cap\Vseg^\Z(G),$ there is a unique decomposition $\EE=\EE_{\rho_1}\cup\cdots\cup\EE_{\rho_k}$ where $\rho_i\neq\rho_j$ for any $i\neq j.$ By Theorem \ref{thm intersections of local Arthur packets}, we obtain $\Psi(\pi(\EE))=\prod_{i=1}^k\Psi(\pi(\EE_{\rho_i})).$ Applying Theorem \ref{thm-count-temp} to determine each $|\Psi(\pi(\EE_{\rho_i}))|$ gives the formula for computing $|\Psi(\pi(\EE))|$.
\end{rmk}

Finally, we give our main result on counting the number of local Arthur packets containing a given theta lift of a tempered representation to its first occurrence in the going-up tower (in the sense of Remark \ref{rmk theta count}).

\begin{thm}
\label{thm-count-theta-temp}
    Let $\EE = ([A_i, B_i], l_i, \eta_i)\in\VRep_\rho(G_n)$ be a tempered extended multi-segment. 
    Let $\mathcal{B}_1, \dots, \mathcal{B}_k$ denote the block decomposition of $\EE$. We consider various cases determined by the columns near the end of the first block.
    Let $H_{col}$ be the last column of $\mathcal{B}_1$ and let $N_{col}$ be the first column of $\mathcal{B}_2$.
    \begin{enumerate}

    \item[Case 1.] Assume that $\EE$ does not start at zero. Then $\Theta_1(\EE)$ is tempered and so, by Theorem \ref{thm-count-temp}, we can compute $|\Psi(\pi(\Theta_1(\EE)))|$.

    \item [Case 2.] \underline{$N_{col} > H_{col}+1$ and $\EE$ starts at zero:} 

         Then $\Theta_1(\EE)$ is tempered and $|\Psi(\pi(\Theta_1(\EE)))|$ is computed by Theorem \ref{thm-count-temp}.
    \end{enumerate}
    Now suppose $N_{col} \leq H_{col}+1$. Let $\eta_H$ be such that $([H_{col}, H_{col}], 0, \eta_H)$ appears in $\EE$, and let $\eta_N$ be such that $([H_{col}+1, H_{col}+1], 0, \eta_N)$ appears in $\EE$. (These are well-defined since $\EE$ is tempered and $N_{col} \leq H_{col}+1$, so that an extended segment in column $H_{col}+1$ appears in $\EE$.) Also let $m_H$ be the multiplicity of $H_{col}$ in $\EE$ and let $m_N$ be the multiplicity of $H_{col}+1$ in $\EE$.

    \begin{enumerate}
        \item [Case 3.] \underline{$m_H \equiv 0 \bmod{2}$, $\eta_N = \eta_H$, $m_N \equiv 1 \bmod{2}$, and $\EE$ starts at zero:} 
        
        We have that
        \[|\Psi(\pi(\Theta_1(\EE)))| = |\Psi(\pi(\Theta_1(\BB_1 \cup \BB_2)))| \cdot \prod_{i=3}^k |\Psi(\pi(\BB_i))|.\]
        
        Suppose $\EE$ has at most two blocks in its block decomposition. Then we have the following recursive formulas:

        \begin{enumerate}
            \item [Case 3.1.] \underline{$\EE$ ends at $H_{col}+1$:}

            Let $\EE' = \rc_{H_{col}+1}(\EE)$. Then \[|\Psi(\pi(\Theta_1(\EE)))| = |\Psi(\pi(\Theta_1(\EE')))|.\]

            \item [Case 3.2.] \uline{$\EE$ ends at $H_{col}+2$:}

            Let $\EE' = \rc_{H_{col}+2}(\EE)$. Then \[|\Psi(\pi(\Theta_1(\EE)))| = 4 |\Psi(\pi(\Theta_1(\EE')))|.\]

            \item [Case 3.3.] \underline{otherwise:}

            We apply similar recursive formulas as Equation \ref{eq first recursion}. More precisely, let $c_{max}$ be the last column of $\EE$, and let $m_{c_{max}-1}$ be the multiplicity of $c_{max}-1$. Let $\EE' = \rc_{c_{max}}(\EE)$ and $\EE'' = \rc_{c_{max}-1}(\EE')$. Then
            \[|\Psi(\pi(\Theta_1(\EE)))| = \begin{cases} 3 |\Psi(\pi(\Theta_1(\EE')))| & \text{if } m_{c_{max}-1} = 1 \\ 4 |\Psi(\pi(\Theta_1(\EE')))| - |\Psi(\pi(\Theta_1(\EE'')))| & \text{if } m_{c_{max}-1} = 3, 5, \ldots. \end{cases}\]
        \end{enumerate}

        \item [Case 4.] \underline{$m_H \equiv 1 \bmod 2$, $m_N \equiv 0 \bmod{2}$, $\eta_N = \eta_H$, and $\EE$ starts at zero:} 

        Then $\Theta_1(\EE)$ is tempered and $|\Psi(\pi(\Theta_1(\EE)))|$ is computed by Theorem \ref{thm-count-temp}.

        \item [Case 5.] \underline{$m_H \equiv 0 \bmod{2}$, $m_N \equiv 0 \bmod{2}$, $\eta_N = \eta_H$, and $\EE$ starts at zero:}
        \[|\Psi(\pi(\Theta_1(\EE)))| = |\Psi(\theta_{-m^{\up, \alpha}}^\up(\pi(\mathcal{B}_1 \cup \mathcal{B}_2)))| \cdot \prod_{i=3}^k |\Psi(\pi(\mathcal{B}_i))|.\]

        \item [Case 6.] \underline{$m_H \equiv 1 \bmod{2}$, $m_N \equiv 1 \bmod{2}$, $\eta_N = \eta_H$, and $\EE$ starts at zero:} Then $\Theta_1(\EE)$ is tempered, and moreover 
        \[|\Psi(\pi(\Theta_1(\EE)))| = |\Psi(\pi(\Theta_1(\mathcal{B}_1)))| \cdot \prod_{i=2}^k |\Psi(\pi(\mathcal{B}_i))|.\]
    \end{enumerate}
\end{thm}

Note that the case that $\eta_N\neq\eta_H$ is excluded from the above theorem. This case is covered by the following conjecture.

\begin{conj}
\label{conj-count-theta-temp}
    Let $\EE = ([A_i, B_i], l_i, \eta_i)\in\VRep_\rho(G_n)$ be a tempered extended multi-segment. Suppose further that $\EE$ starts at zero and consider the (ordered) decomposition of $\EE$ into blocks $\mathcal{B}_1, \dots, \mathcal{B}_k$. 

    Let $H_{col}$ be the last column of $\mathcal{B}_1$ and let $N_{col}$ be the first column of $\mathcal{B}_2$.
    Now suppose $N_{col} \leq H_{col}+1$. Let $\eta_H$ be such that $([H_{col}, H_{col}], 0, \eta_H)$ appears in $\EE$, and let $\eta_N$ be such that $([H_{col}+1, H_{col}+1], 0, \eta_N)$ appears in $\EE$. (These are well-defined since $\EE$ is tempered and $N_{col} \leq H_{col}+1$, so that an extended segment in column $H_{col}+1$ appears in $\EE$.) Also let $m_H$ be the multiplicity of $H_{col}$ in $\EE$. 
    \begin{enumerate}
        \item [Case 7.] \underline{$\eta_N \neq \eta_H$:} 

        Let $\BB_1' := \rc_{H_{col}}(\BB_1)$, $\BB_2' := \rc_{H_{col}}(\BB_2)$, and $\BB_2'' := \rc_{H_{col}+1}(\BB_2')$. Let
        \[D = \begin{cases}
            \parbox[t]{3.4in}{%
            $|\Psi(\pi(\mathcal{B}_1))| \cdot |\Psi(\pi(\mathcal{B}_2))| + |\Psi(\pi(\mathcal{B}_1))| \cdot |\Psi(\pi(\mathcal{B}_2'))|$ \\
            \null\hfill${}- |\Psi(\pi(\mathcal{B}_1'))| \cdot |\Psi(\pi(\mathcal{B}_2'))|$} & \text{if } m_H - 1 = 1 \\
            2 |\Psi(\pi(\mathcal{B}_1))| \cdot |\Psi(\pi(\mathcal{B}_2))| - |\Psi(\pi(\mathcal{B}_1'))| \cdot |\Psi(\pi(\mathcal{B}_2))| & \text{if } m_H-1 > 1.
        \end{cases}\]
        If $\mathcal{B}_2''$ is empty then let $U=0$, and otherwise let \[U = |\Psi(\pi(\mathcal{B}_1))| \cdot (|\Psi(\pi(\mathcal{B}_2'))| - |\Psi(\pi(\mathcal{B}_2''))|).\]
        Then \[|\Psi(\pi(\Theta_1(\EE)))| = \big( |\Psi(\pi(\mathcal{B}_1))| \cdot |\Psi(\pi(\mathcal{B}_2))| + D + U \big) \cdot \prod_{i=3}^k |\Psi(\pi(\mathcal{B}_i))|.\]
    \end{enumerate}
\end{conj}

We prove Case 1 of Theorem \ref{thm-count-theta-temp} in Lemma \ref{lemma-EE-does-not-start-at-zero}, Cases 2, 4, and 6 in \S\ref{sec Cases 1, 2, 4}, Case 3 in \S\ref{sec case 3}, and Case 5 in \S\ref{sec case 5}. Most of the cases of Theorem \ref{thm-count-theta-temp} follow from the theory of blocks and almost blocks developed in \cite{HKT26} in the course of the proof of Theorems \ref{thm-count-block-temp}, along with the theory developed here in proving Theorem \ref{thm-count-block-theta-temp} (except Case 5 which requires further work, see \S\ref{sec case 5}).

We do not prove Conjecture \ref{conj-count-theta-temp}; however, we comment on its motivation in Appendix \ref{sec motivation}.
Collectively, the previous theorem and conjecture allow us to compute $|\Psi(\theta_{-m^{\up, \alpha}}^{\up}(\pi(\EE)))|$ from $|\Psi(\pi(\EE))|$ in the following sense.

\begin{rmk}\label{rmk theta count}
    Let $\EE\in\Rep(G_n).$
    Note that $\Psi(\pi(\EE))=\prod_{\rho'}\Psi(\pi(\EE_{\rho'}))$ and by Theorems \ref{thm Adams going up tower} and \ref{thm compute theta lift}, we have \[\Psi(\theta_{-m^{\up, \alpha}}^{\up}(\pi(\EE)))=\Psi(\pi(\EE_{m^{\up,\alpha}}^\up))=\prod_{\rho''}\Psi(\pi((\EE_{m^{\up,\alpha}}^\up)_{\rho''}).\] Moreover, if $\rho\neq\chi_V,$ then $|\Psi(\pi(\EE_{\rho'}))|=|\Psi(\pi((\EE_{m^{\up,\alpha}}^\up)_{\chi_W\chi_V\inv\rho'})|.$ Thus, we understand that $|\Psi(\theta_{-m^{\up, \alpha}}^{\up}(\pi(\EE)))|$ can be computed from $|\Psi(\pi(\EE))|$ provided that we understand $|\Psi(\pi((\EE_{m^{\up,\alpha}}^\up)_{\chi_W})|=|\Psi(\pi(\Theta_1(\EE_{\chi_V})))|$ which is described by Theorem \ref{thm-count-theta-temp} or Conjecture \ref{conj-count-theta-temp}.
\end{rmk}

We give several remarks on the above theorem and conjecture.

\begin{rmk}
    The fives cases for $\EE$ starting at zero given  in Theorem \ref{thm-count-theta-temp}, plus the additional case given in Conjecture \ref{conj-count-theta-temp} correspond to the various ways the first block could end, illustrated below. Case 2 of Theorem \ref{thm-count-theta-temp} corresponds to the situation where there is a gap between two columns. Conjecture \ref{conj-count-theta-temp}  corresponds to the situation where there is a column with an even number of circles, and the signs of the circles in the next column are opposite. Cases 3, 4, 5, and 6 of Theorem \ref{thm-count-theta-temp} correspond to the situation where the circles in two consecutive columns have the same sign. 

    \begin{center}
    \begin{tabular}{cc} 
    Case 2 & Case 3 \\
    $
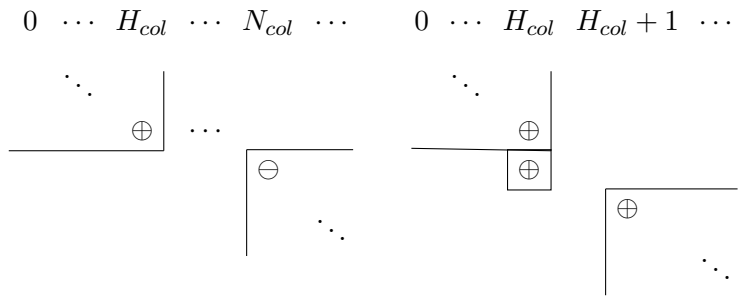
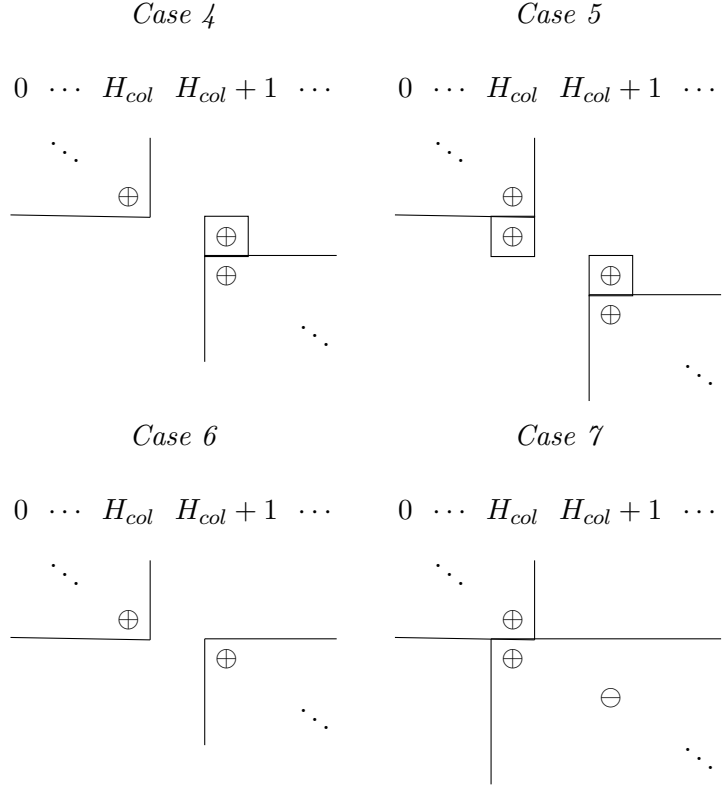
\begin{tikzpicture}[baseline=(current bounding box.north)]
      \matrix (m) [matrix of math nodes, nodes in empty cells] {
        0 & \cdots & H_{col} & \cdots & N_{col} & \cdots \\
        & \ddots & \phantom{\oplus} & & & \\
        \phantom{\oplus} & & \oplus & \cdots & & \\
        & & & & \ominus & \phantom{\oplus} \\
        & & & & \phantom{\oplus} & \ddots \\
      } ;
      \draw (m-3-3.south east) -- (m-3-1.south west);
      \draw (m-3-3.south east) -- (m-2-3.north east);
      \draw (m-4-5.north west) -- (m-4-6.north east);
      \draw (m-4-5.north west) -- (m-5-5.south west);
    \end{tikzpicture}$ & 
    $\begin{tikzpicture}[baseline=(current bounding box.north)]
      \matrix (m) [matrix of math nodes, nodes in empty cells] {
        0 & \cdots & H_{col} & H_{col}+1 & \cdots \\
        & \ddots & \phantom{\oplus} & & \\
        & \phantom{\oplus} & \oplus & & \\
        & & \oplus & & \\
        & & & \oplus & \phantom{\oplus}\\
        & & & \phantom{\oplus} & \ddots \\
      } ;
      \draw (m-3-3.south east) -- (m-3-1.south west);
      \draw (m-3-3.south east) -- (m-2-3.north east);
      \draw (m-4-3.north west) rectangle (m-4-3.south east);
      \draw (m-5-4.north west) -- (m-5-5.north east);
      \draw (m-5-4.north west) -- (m-6-4.south west);
    \end{tikzpicture}$ 
    \end{tabular}
    \end{center}
    \begin{center}
    \begin{tabular}{cc}
    Case 4 & Case 5 \\
    $\begin{tikzpicture}[baseline=(current bounding box.north)]
      \matrix (m) [matrix of math nodes, nodes in empty cells] {
        0 & \cdots & H_{col} & H_{col}+1 & \cdots \\
        & \ddots & \phantom{\oplus} & & \\
        & \phantom{\oplus} & \oplus & & \\
        & & & \oplus & \\
        & & & \oplus & \phantom{\oplus} \\
        & & & \phantom{\oplus} & \ddots \\
      } ;
      \draw (m-3-3.south east) -- (m-3-1.south west);
      \draw (m-3-3.south east) -- (m-2-3.north east);
      \draw (m-4-4.north west) rectangle (m-4-4.south east);
      \draw (m-5-4.north west) -- (m-5-5.north east);
      \draw (m-5-4.north west) -- (m-6-4.south west);
    \end{tikzpicture}$ & 
    $\begin{tikzpicture}[baseline=(current bounding box.north)]
      \matrix (m) [matrix of math nodes, nodes in empty cells] {
        0 & \cdots & H_{col} & H_{col}+1 & \cdots \\
        & \ddots & \phantom{\oplus} & & \\
        & \phantom{\oplus} & \oplus & & \\
        & & \oplus & & \\
        & & & \oplus & \phantom{\oplus}\\
        & & & \oplus & \phantom{\oplus}\\
        & & & \phantom{\oplus} & \ddots \\
      } ;
      \draw (m-3-3.south east) -- (m-3-1.south west);
      \draw (m-3-3.south east) -- (m-2-3.north east);
      \draw (m-4-3.north west) rectangle (m-4-3.south east);
      \draw (m-5-4.north west) rectangle (m-5-4.south east);
      \draw (m-6-4.north west) -- (m-6-5.north east);
      \draw (m-6-4.north west) -- (m-7-4.south west);
    \end{tikzpicture}$ 
    \end{tabular}
    \end{center}
    \begin{center}
    \begin{tabular}{cc}
    Case 6 & Case 7 \\
    $\begin{tikzpicture}[baseline=(current bounding box.north)]
      \matrix (m) [matrix of math nodes, nodes in empty cells] {
        0 & \cdots & H_{col} & H_{col}+1 & \cdots \\
        & \ddots & \phantom{\oplus} & & \\
        & \phantom{\oplus} & \oplus & & \\
        & & & \oplus & \phantom{\oplus} \\
        & & & \phantom{\oplus} & \ddots \\
      } ;
      \draw (m-3-3.south east) -- (m-3-1.south west);
      \draw (m-3-3.south east) -- (m-2-3.north east);
      \draw (m-4-4.north west) -- (m-4-5.north east);
      \draw (m-4-4.north west) -- (m-5-4.south west);
    \end{tikzpicture}$ & 
    $\begin{tikzpicture}[baseline=(current bounding box.north)]
      \matrix (m) [matrix of math nodes, nodes in empty cells] {
        0 & \cdots & H_{col} & H_{col}+1 & \cdots \\
        & \ddots & \phantom{\oplus} & & \\
        & \phantom{\oplus} & \oplus & & \\
        & & \oplus & & \phantom{\oplus}\\
        & & & \ominus & \\
        & & \phantom{\oplus} & & \ddots \\
      } ;
      \draw (m-3-3.south east) -- (m-3-1.south west);
      \draw (m-3-3.south east) -- (m-2-3.north east);
      \draw (m-4-3.north west) -- (m-4-5.north east);
      \draw (m-4-3.north west) -- (m-6-3.south west);
    \end{tikzpicture}$
    \end{tabular}
    \end{center}
\end{rmk}

\begin{rmk}\label{rmk going up packets calculation}
    In each of the six cases for $\EE$ in Theorem \ref{thm-count-theta-temp}, plus the additional case given in Conjecture \ref{conj-count-theta-temp}, each of the parts of the formulas can be computed using the previously stated results. We explicate this as follows.
    \begin{itemize}
    \item In Case 1, since $\Theta_1(\EE)$ is tempered, we can compute $|\Psi(\pi(\Theta_1(\EE)))|$ using Theorem \ref{thm-count-temp}.
        \item In Case 2, again $\Theta_1(\EE)$ is tempered and so we can compute $|\Psi(\pi(\Theta_1(\EE)))|$ by Theorem \ref{thm-count-temp}.
        \item In Case 7, $\BB_1'$, $\BB_2'$, and $\BB_2''$ are all blocks so we can compute $|\Psi(\pi(\BB_1'))|$, $|\Psi(\pi(\BB_2))|$, and $|\Psi(\pi(\BB_2''))|$ using Theorem \ref{thm-count-block-temp}.
        \item Suppose that we are in Case 3.
        \begin{itemize}
            \item In Case 3.1, since $\EE$ contains circles in only one extra column after the first block, namely $H_{col}+1$, $\EE'$ is almost-block, since it is $\BB_1$ together with an extra circle in column $H_{col}$. So we can compute $|\Psi(\pi(\Theta_1(\EE')))|$ using Theorem \ref{thm-count-block-theta-temp}.
            \item In Case 3.2, we note that the going-up tower and $m^{\up, \alpha}$ are the same for $\EE$ and $\EE'$, since they only differ in columns after the first block, so in fact $\EE'$ falls directly into Case 3.1.
            \item In Case 3.3, we similarly note that $\EE$, $\EE'$, and $\EE''$ all have the same up tower and value of $m^{up, \alpha}$, so again the recursive formulas eventually reduce to Case 3.1 or Case 3.2.
        \end{itemize}
        \item In Case 4, as in Case 1, $\Theta_1(\EE)$ is tempered, so we can apply Theorem \ref{thm-count-temp}.
        \item In Case 5, note that since $m_H \equiv 0 \bmod{2}$, the second block $\BB_2$ consists of a single extended segment $([H_{col}, H_{col}], 0, \eta_H)$. So $\BB_1 \cup \BB_2$ is an almost-block, and hence $|\Psi(\pi(\Theta_1(\BB_1 \cup \BB_2)))|$ can be computed by Theorem \ref{thm-count-block-theta-temp}. Note that the going-up tower and $m^{\up, \alpha}$ are the same for $\EE$ as for $\BB_1 \cup \BB_2$, so we can apply Theorem \ref{thm-count-block-theta-temp} directly. The other terms are just blocks and can be computed by Theorem \ref{thm-count-block-temp}.
        \item In Case 6, by similar reasoning, the first term can be computed by Theorem \ref{thm-count-block-theta-temp} and the other terms by Theorem \ref{thm-count-block-temp}.
    \end{itemize}
\end{rmk}

\begin{rmk}\label{rmk cases 1 and 4}
    In Cases 2 and 4 of Theorem \ref{thm-count-theta-temp}, while it is possible to compute the number of packets the theta lift lies in, there is no easy formula like in Case 6. This is because it is possible for the theta lift to cause two blocks to interact. For example, suppose 
    \[\EE = \bordermatrix{
    & 0 & 1 & 2 & 3 & 4 \cr
    & \ominus & & & & \cr
    & & \oplus & & & \cr
    & & & & \ominus & \cr
    & & & & \ominus & \cr
    & & & & \ominus & \cr
    & & & & & \oplus \cr
    }_{\chi_V}.\] The extended multi-segment corresponding to $\theta_{-m^{\up, \alpha}}^\up(\pi(\EE))$ is
    \[\bordermatrix{
    & -2 & -1 & 0 & 1 & 2 & 3 & 4 \cr
    & \lhd & \lhd & \oplus & \rhd & \rhd & & \cr
    &&& \ominus & & & & \cr
    &&& & \oplus & & & \cr
    &&& & & & \ominus & \cr
    &&& & & & \ominus & \cr
    &&& & & & \ominus & \cr
    &&& & & & & \oplus \cr
    }_{\chi_W}.\] This is equivalent to (via a $dual \circ ui \circ dual$ and a $ui\inv$) the extended multi-segment
    \[\bordermatrix{
    & 0 & 1 & 2 & 3 & 4 \cr
    & \oplus & & & & \cr
    & & \ominus & & & \cr
    & & & \oplus & & \cr
    & & & & \ominus & \cr
    & & & & \ominus & \cr
    & & & & \ominus & \cr
    & & & & & \oplus \cr
    }_{\chi_W}.\] Note that this extended multi-segment now consists of a single block, not two. In particular, since the blocks $\BB_1$ and $\BB_2$ making up $\EE$ have become one block, the number of local Arthur packets the theta lift lies in does not depend directly on $|\Psi(\pi(\BB_1))|$ or $|\Psi(\pi(\BB_2))|$ or other related quantities.
\end{rmk}

Finally, an analogous result, to be proven in \S\ref{sec Anti-Tempered Extended Multi-Segments}, follows from the above for the case where $\pi$ is an anti-tempered representation.

\begin{thm}
\label{theta-antitempered-count}
    Suppose $\mathcal{E}$ is an anti-tempered extended multi-segment. Denote its first row by $([n, -n], n, \eta(\mathcal{E}))$ and let $m_n$ be the multiplicity of the $n$th column of $dual(\mathcal{E}).$ Then $|\Psi(\theta_{-m^{\up, \alpha}}^\up (\pi(\mathcal{E})))|$ is equal to:
    $$\begin{cases}
        4|\Psi(\pi(\mathcal{E}))| - |\Psi(\pi(\EE'))| & m_n > 1, ~dual(\mathcal{E}) \text{ is a block starting at $0$},\\
        3 |\Psi(\pi(\mathcal{E}))| & m_n = 1, ~dual(\mathcal{E}) \text{ is a block starting at $0$},  \\
        3|\Psi(\pi(\mathcal{E}))| - |\Psi(\pi(\EE'))| & m_n > 1 \text{ odd}, ~dual(\mathcal{E}) \text{ not a block starting at $0$},\\
        2 |\Psi(\pi(\mathcal{E}))| & \text{otherwise},
    \end{cases}$$
    where $\EE'=(dual \circ \rc_n\circ dual)(\EE).$
\end{thm}

\subsection{Applications}
In this section, we apply the previously stated results to several families of representations. In particular, we obtain the counts for the number of local Arthur packets containing the first occurrence of the theta lift to the going-up tower for generic representations of Arthur type (Theorem \ref{thm-generic-theta-count}) and unramified representations of Arthur type (Theorem \ref{thm-unram-theta-count}).

We begin with the generic case (and explain why this case covers generic representations in Remark \ref{rmk generic}).

\begin{thm}\label{thm-generic-theta-count}
    Suppose that $\EE\in\VRep_\rho(G_n)$ with
    \[
    \EE=\{([A_i,A_i]_\rho,0,1)\}_{i\in I_\rho,>}.
    \]
    Then
    \[
    |\Psi(\pi(\Theta_1(\EE))|=\begin{cases}
        1 & \text{if $\min\{A_i\}\neq 1$ or $\min\{A_i\}=1$ and $\up=1$}, \\
        2 & \text{if $\EE$ is in Cases 3 or 5}, \\
        3 & \text{otherwise},
    \end{cases}
    \]
    where the cases refer to those of Theorem \ref{thm-count-theta-temp}. Moreover, in each case, $\Psi(\pi(\Theta_1(\EE))$ can be determined concretely (e.g., see the below proof).
\end{thm}

\begin{proof}
    The proof is a case-by-case application of Theorem \ref{thm-count-theta-temp} and we adopt its terminology. 
    
    Suppose first that $\EE$ is in Case 1. By Theorem \ref{thm going up first occurrence}, the first occurrence is at equal rank and so the choice of going-up tower is ambiguous (see also Remark \ref{rmk equal first occurences}). In either tower, the first occurrence is tempered, by Theorem \ref{thm-count-theta-temp}, and so we apply Theorem \ref{thm-count-temp}. If $\min\{A_i\}\neq 1,$ then in either tower, we obtain that $|\Psi(\pi(\Theta_1(\EE))|=1$ directly from Theorem \ref{thm-count-temp}. If $\min\{A_i\}=1,$ then the choice of tower does matter.
    \begin{enumerate}
        \item If we choose $\up=1$, then Theorem \ref{thm-count-temp} implies that $|\Psi(\pi(\Theta_1(\EE))|=1$. Indeed, the added extended segment $([0,0]_{\chi_W},0,1)$ forms its own block and the rest of block decomposition of $\Theta_1(\EE)$ is obtained from $\EE$.
        \item If we choose $\up=-1$, then Theorem \ref{thm-count-temp} implies that $|\Psi(\pi(\Theta_1(\EE))|=3$. Indeed, in this setting, the added extended segment $([0,0]_{\chi_W},0,-1)$ gets grouped with $([1,1]_{\chi_W},0,1)$ into a block $\BB_1$. It is straightforward to verify that $|\Psi(\pi(\BB_1))|=3$ (either by direct calculation or by applying Theorem \ref{thm-count-block-temp}) from which, we obtain that $|\Psi(\pi(\Theta_1(\EE))|=3$ by Theorem \ref{thm-count-temp}.
    \end{enumerate}
This completes the proof for Case 1.

Hereinafter, we assume that $\EE$ starts at zero. That is, $([0,0]_\rho,0,1)\in\EE$. By Theorem \ref{thm going up first occurrence}, we see that $\up=-1$ and $m^{\up,\alpha}=3$. Thus we have
\[
\Theta_1(\EE)=\{([1,-1]_{\chi_W},1,-1)\}\cup\{([A_i,A_i]_{\chi_W},0,1)\}_{i\in I_\rho,>}
\]
which is equipped with the unique $(P')$ ordering.

Suppose that $\EE$ is in Case 2. That is, $([0,0]_\rho,0,1)$ occurs with odd multiplicity in $\EE$ and $([1,1]_{\rho},0,1)\not\in\EE$.  Then $\Theta_1(\EE)$ is tempered by Theorem \ref{thm-count-theta-temp}. Moreover, we can perform a $dual\circ ui\circ dual$ on $\Theta_1(\EE)$ to combine $([1,-1]_{\chi_W},1,-1)$ and $([0,0]_{\chi_W},0,1)$ into $([1,0]_{\chi_W},0,\eta)$ for some $\eta\in\{\pm1\}$. Then we perform a $ui^{-1}$ to split $([1,0]_{\chi_W},0,\eta)$ into $([0,0]_{\chi_W},0,\eta)$ and $([1,1]_{\chi_W},0,-\eta)$ which results in a tempered extended multi-segment. We note that if $([2,2]_{\chi_W},0,1)\in\Theta_1(\EE)$, then after performing the same $dual\circ ui\circ dual$ and $ui^{-1}$, it becomes $([2,2]_{\chi_W},0,-\eta)$. Indeed this follows directly from the definition of $dual$ (Definition \ref{def dual}) and the hypothesis that $([0,0]_\rho,0,1)$ occurs with odd multiplicity in $\EE$.
We complete this case by applying Theorem \ref{thm-count-temp}. This shows that we have performed all possible operators on $\Theta_1(\EE)$, i.e., $|\Psi(\pi(\Theta_1(\EE))|=3$.

Suppose next that $\EE$ is in Case 4. In this case, we have that $([0,0]_{\rho},0,1)\in\EE$ with odd multiplicity and $([1,1]_{\rho},0,1)\in\EE$ with even multiplicity. As in Case 2, $\Theta_1(\EE)$ is tempered and we can apply $dual\circ ui \circ dual$ and then a $ui^{-1}.$ Indeed, we first apply $dual\circ ui \circ dual$ on $\Theta_1(\EE)$ to combine $([1,-1]_{\chi_W},1,-1)$ and $([0,0]_{\chi_W},0,1)$ into $([1,0]_{\chi_W},0,\eta)$ for some $\eta\in\{\pm1\}$. Then we perform a $ui^{-1}$ to split $([1,0]_{\chi_W},0,\eta)$ into $([0,0]_{\chi_W},0,\eta)$ and $([1,1]_{\chi_W},0,-\eta)$ which results in a tempered extended multi-segment. We note that this $ui^{-1}$ is indeed possible since if $([1,1]_{\chi_W},0,\eta')\in dual\circ ui \circ dual(\EE)$ for some $\eta'$, then it follows from our hypothesis and the definition of $dual$ (Definition \ref{def dual}) that $\eta'=-\eta$. Next, we observe that if $([2,2]_{\chi_W},0,1)\in\Theta_1(\EE)$, then after performing the same $dual\circ ui\circ dual$ and $ui^{-1}$, it becomes $([2,2]_{\chi_W},0,-\eta)$. Indeed this follows directly from the definition of $dual$ (Definition \ref{def dual}) and our hypothesis.
We complete this case by applying Theorem \ref{thm-count-temp}. This shows that we have performed all possible operators on $\Theta_1(\EE)$, i.e., $|\Psi(\pi(\Theta_1(\EE))|=3$.

If $\EE$ is in Cases 3, 5 or 6, then the result follows directly from applying Theorem \ref{thm-count-theta-temp}. Similarly to the previous cases, it is straightforward to determine $\Psi(\pi(\Theta_1(\EE))$ explicitly using $dual\circ ui \circ dual$ and $ui^{-1}$ operators, but we omit this.
Finally, we note that we are never in the situation of Case 7 of Conjecture \ref{conj-count-theta-temp}. This completes the proof of the theorem.
\end{proof}

\begin{rmk}\label{rmk generic}
    Let $\EE\in\Rep(G_n)$ be tempered with $B_i=A_i$ and $\eta_i=1$ for any $i\in I_\rho$ and any $\rho.$ Then $\pi(\EE)$ is a generic representation and $|\Psi(\pi)|=1$ (\cite[Theorem 1.2]{HLL24}). In this case, $\EE_{\chi_V}$ is in the setting of Theorem \ref{thm-generic-theta-count}. From Theorem \ref{thm-generic-theta-count}, we see that there often exists a non-$\theta$-relevant pair for $(\pi,\theta_{-m^{\up,\alpha}(\pi)}^\up(\pi)).$
\end{rmk}

Next, we consider the opposite extremal case: $\EE$ is anti-tempered with each $\eta_i=1$.

\begin{thm}\label{thm-unram-theta-count}
    Suppose that $\EE\in\VRep_\rho(G_n)$ with
    \[
    \EE=\{([A_i,-A_i]_\rho,A_i,1)\}_{i\in I_\rho,>}.
    \]
    Then
    \[
    |\Psi(\pi(\Theta_1(\EE))|=\begin{cases}
        3 & \mathrm{if} \ dual(\EE) \ \text{is a block starting at zero and} \ m_{A_1} \ \text{is odd,}  \\
        2 & \mathrm{otherwise.} 
    \end{cases}
    \]
\end{thm}

\begin{proof}
    This is a straightforward application of Theorem \ref{theta-antitempered-count}. We note that $dual(\EE)=\{([A_i,A_i]_\rho,0,\eta)\}_{i\in I_\rho,>}$ for some $\eta\in\{\pm1\}.$ In particular, $dual(\EE)$ is a block if and only if $A_1=A_2=\cdots=A_{|I_\rho|}$ and $m_{A_1}$ is odd. 
\end{proof}

\begin{rmk}
Let $\pi$ be an irreducible unramified representation of Arthur type of $G_n$. Then $\pi=\pi(\EE)$ for some extended multi-segment $\EE$. By \cite[Proposition 6.4]{Moe09b} (see also \cite[Theorem 5.4 and Proposition 5.5]{HLL24}), we have that $|\Psi(\pi)|=1$,  $\psi_\EE$ is trivial on the Arthur-$\SL_2(\mathbb{C})$, and $\EE_{\chi_V}$ is in the setting of Theorem \ref{thm-unram-theta-count}. Applying Theorem \ref{thm-unram-theta-count}, we see that there always exists a non-$\theta$-relevant pair for $(\pi,\theta_{-m^{\up,\alpha}(\pi)}^\up(\pi)).$
\end{rmk}

\section{Basic notions and lemmas}\label{sec basic notions and lemmas}

In this section, we recall from \cite[\S 4]{HKT26} some basic notions and lemmas which will be helpful in the proofs of our results.

\subsection{Relations between operators}

First, we recall two facts. The first says that certain operators are ``local'' operators.

\begin{lemma}[{\cite[Lemma 4.3]{HKT26}}]
\label{lem-local}
    Suppose that $\EE\in\VRep_\rho(G_n)$. 
    Then the operations $ui$ and $dual \circ ui \circ dual$ are ``local'' operations. More specifically, if a certain row is not involved in the union-intersection or any row exchanges, then it is fixed by these operations.
\end{lemma}

The second fact is about row swaps commuting with certain operations, and will be used for certain reductions in several technical arguments later.
\begin{lemma}[{\cite[Corollary 4.5]{HKT26}}]
\label{cor Alex}
    Suppose $r$ is a row in $\mathcal{E}\in\VRep_\rho^\mathbb{Z}(G_n)$ immediately followed by some consecutive rows consisting a sub-multi-segment $\mathcal{E}_1 \subset \mathcal{E},$ $\mathcal{E}_2$ is an extended multi-segment equivalent to $\EE_1$. 
    We suppose further that
    $\supp(r)\supseteq\supp(s)$ for any $s\in\EE_1\cup\EE_2$.    
    Then, the result after $r$ is exchanged with $\EE_1$ is the same as the result after $r$ is exchanged with $\EE_2$. 
\end{lemma}

We briefly recall definitions and lemmas relating to the alternating sign condition.

If $r = ([A, B]_\rho, l, \eta)$ is a row, the \emph{number of circles in $r$} is $C(r) := b - 2l,$ and for $\EE$ a virtual extended multi-segment, the \emph{total number of circles in $\EE$} is $C(\EE) := \sum_{r \in \EE} C(r)$. The \emph{sign} $\eta(\EE)$ of an extended multi-segment $\EE$ is $\eta(\EE) := \eta(r_1),$ where $r_1$ is the first row in $\EE$ in the admissible order.

\begin{defn}[{\cite[Definition 4.8]{HKT26}}]
    If $r_1$, $r_2$ are rows with $r_1 < r_2$ in some admissible order, we say that $r_1$ and $r_2$ (in that order) satisfy the \emph{alternating sign condition} if \[\eta(r_2) = (-1)^{C(r_1)} \eta(r_1).\]
    We say that an extended multi-segment $\EE$ is \emph{alternating} if any two consecutive rows of $\EE$ satisfy the alternating sign condition.
\end{defn}

If $r$ is a row, we call $(-1)^{C(r)-1} \eta(r)$ the \emph{sign of the last circle of $r.$} With this definition, note that two rows $r_1 < r_2$ satisfy the alternating sign condition if and only if the sign of the last circle of $r_1$ is opposite the sign of $r_2$.

Now, much of the substance of the ensuing proofs will depend on understanding how a row changes under certain row exchanges. In particular, we are interested in analyzing what happens when we swap a row with a series of consecutive rows satisfying the alternating sign condition. The following lemmas answer this question.

\begin{lemma}[{\cite[Lemma 4.15]{HKT26}}]
\label{big swap down}
    Suppose that $r_0 < r_1 < \cdots < r_k$ are consecutive rows satisfying the alternating sign condition with $\supp(r_0) \supset \supp(r_i)$ for $i > 0$. Let $r_1' < \dots < r_k' < r_0^{(k)}$ be their images after $r_0$ is row exchanged $k$ times. Then we have
    \begin{align*}
        (l(r_0^{(k)}), \eta(r_0^{(k)})) &= \left(l(r_0) - \sum_{i=1}^k C(r_i), (-1)^{\sum_{i=1}^k C(r_i)} \eta(r_0) \right), \\
        (l(r_i'), \eta(r_i')) &= (l(r_i'), (-1)^{C(r_0)+1} \eta(r_i)) \text{ for } i >0.
    \end{align*}
    Moreover, after the row exchanges every pair of adjacent rows satisfies the alternating sign condition except $r_k'$ and $r_0^{(k)}$.
\end{lemma}

\begin{lemma}[{\cite[Lemma 4.16]{HKT26}}]
\label{big swap up}
    Suppose that $r_1 < \cdots < r_k < r_0$ are consecutive rows such that every pair of consecutive rows satisfies the alternating sign condition except $r_k$ and $r_0$. Suppose further that $\supp(r_0) \supset \supp(r_i)$ for $i > 0$, and that $C(r_0) \geq 2 \sum_{i = 1}^k C(r_i)$. Let $r_0^{(k)} < r_1' < \cdots < r_k'$ be their images after $r_0$ is exchanged up $k$ times. Then:
    \begin{align*}
        (l(r_0^{(k)}), \eta(r_0^{(k)})) &= \left(l(r_0) + \sum_{i=1}^k C(r_i), (-1)^{\sum_{i=1}^k C(r_i)} \eta(r_0) \right) \\
        (l(r_i'), \eta(r_i')) &= (l(r_i), (-1)^{C(r_0) + 1} \eta(r_i)) \text{ for } i > 0.
    \end{align*}
    Moreover, after the row exchanges every pair of adjacent rows satisfies the alternating sign condition.
\end{lemma}

We also recall the following two lemmas, which demonstrate cases in which swapping a row $h$ with two rows $h_1$ and $h_2$ leaves $h$ unchanged. These results will allow us to simplify future proofs by removing such pairs $h_1$ and $h_2$ from a virtual extended multi-segment.

\begin{lemma}[{\cite[Lemma 4.17]{HKT26}}]
\label{lem-multiplicity-cancel}
    Suppose $h < h_1 < h_2$ are consecutive rows and that $\supp(h) \supset \supp(h_1)$, $\supp(h_2)$. Further suppose $C(h) \neq 0$, and that $C(h_1) = C(h_2) = 1$, $l(h_1) = l(h_2),$ and $\eta(h_1) = \eta(h_2).$ Then $h$ is unchanged after being exchanged with $h_1$ and $h_2$. Moreover, $h_1$ and $h_2$ are unchanged except that their sign is multiplied by $(-1)^{C(h) + 1}$.
\end{lemma}

\begin{lemma}[{\cite[Lemma 4.18]{HKT26}}]
\label{lem-multiplicity-cancel-up}
    Suppose $h_1 < h_2 < h$ are consecutive rows and that $\supp(h) \subset \supp(h_1)$, $\supp(h_2)$. Further suppose $C(h) \neq 0$, and that $C(h_1) = C(h_2) = 1$, $l(h_1) = l(h_2)$ and $\eta(h_1) = \eta(h_2).$ Then $h$ is unchanged after being exchanged with $h_1$ and $h_2$. Moreover, $h_1$ and $h_2$ are unchanged except that their sign is multiplied by $(-1)^{C(h)}$.
\end{lemma}

Next, we recall a specific kind of extended segment that often appears in extended multi-segments corresponding to tempered representations.

\begin{defn}[{\cite[Definition 4.12]{HKT26}}]
\label{def-hat}
    A \emph{hat} is an extended segment of the form $([A, B]_\rho, l, \eta)$ where $B = -l$.
\end{defn}

The alternating sign condition helps provide an important criterion for whether consecutive hats can interact with each other.

\begin{lemma}[{\cite[Lemma 4.19]{HKT26}}]
\label{merge}
    Given two consecutive hats $h_1 < h_2$ in $(P')$ order written as $h_1 = ([A_1, -B_1], B_1, \eta_1)$ and $h_2 = ([A_2, -B_2], B_2, \eta_2),$
    it is possible to apply a nontrivial $dual \circ ui \circ dual$ to $h_1$ and $h_2$ if and only if $A_2 = B_1 - 1$ and $h_1$ and $h_2$ satisfy the alternating sign condition.
    The result has only one row and it is of the form \[h = ([A_1, -B_2], B_2, \eta_1).\]
\end{lemma}

\begin{defn}[{\cite[Definition 4.20]{HKT26}}]
    We refer to the act of performing a $dual \circ ui \circ dual$ of type 3' on two consecutive hats $h_1$, $h_2$ as described above as \emph{merging} the hats. We say two hats are \emph{mergeable} if they satisfy the conditions in Lemma \ref{merge}. We denote the merged hat by $h_1 * h_2.$
\end{defn}

We observe from Lemma \ref{merge} that $C$ is additive under mergings. Formally, if $h_1$ and $h_2$ are two hats, then $C(h_1 * h_2) = C(h_1) + C(h_2).$ It follows immediately from the formulas in Lemmas \ref{big swap down} and \ref{big swap up} that row exchanges commute with the action of merging hats. We state this explicitly in the following corollary.

\begin{cor}[{\cite[Lemma 4.21]{HKT26}}]
\label{swap commutes with M}
    Let $r_1 < r_2 < r_3$ be three consecutive rows in a virtual extended multi-segment.
    \begin{itemize}
        \item Suppose $r_2$ and $r_3$ are mergeable hats, and  $\supp(r_2), \supp(r_3) \subset \supp(r_1)$. Then the image of $r_1$ after being exchanged with $r_2$ and $r_3$ is the same as the image of $r_1$ after being exchanged with $r_2 * r_3$.
        \item Suppose $r_1$ and $r_2$ are mergeable hats, and  $\supp(r_1), \supp(r_2) \subset \supp(r_3)$. Then the image of $r_3$ after being exchanged with $r_2$ and $r_1$ is the same as the image of $r_3$ after being exchanged with $r_1 * r_2$.
    \end{itemize}
\end{cor}

\subsection{\texorpdfstring{Type $Y_\mathcal{M}$}{}} In this section, we recall the notion of type $Y_\mathcal{M}$ (Definition \ref{defn type Y M}) and various definitions and results associated to it from \cite{HKT26}. This notion is used in adapting Theorem \ref{thm-count-temp} to the setting of the local theta correspondence.

Let $\BB\in \Block_\rho(G_n)$.
Recall from Definition \ref{def-starts-ends} that $c_{\min} = \min_{r \in \mathcal{B}} B(r)$ and $c_{\max} = \max_{r \in \mathcal{B}} A(r)$. Also recall from Definition \ref{def-multiplicity} that $m_c$ denotes the number of rows in $\mathcal{B}$ with support $[c, c]$, which we refer to as the multiplicity of column $c$. We let 
$\mathcal{M}_\B = (m_{c_{\min}}, m_{c_{\min} + 1}, \dots, m_{c_{\max}})$ denote the tuple of multiplicities of all rows in $\mathcal{B}.$ When $\B$ is fixed, we often write $\mathcal{M}=\mathcal{M}_\B$ for brevity. Note that each entry in $\mathcal{M}$ is a positive odd integer by Definition \ref{def-block}. More generally, we consider two fixed integers $c_{\min},c_{\max}\in\Z_{\geq 0}$ with $c_{\max}\geq c_{\min}$. A tuple \[\mathcal{M}(c_{\min}, c_{\max})=(m_{c_{\min}}, m_{c_{\min}+1},\dots,m_{c_{\max}})\in\Z^{c_{\max} - c_{\min} +1}\] is called a \emph{block-tuple} if each $m_{i}$ is a positive odd integer. We note that if $\B$ is a block, then $\mathcal{M}_\B$ is a block-tuple.

\begin{defn}[{\cite[Definition 5.1]{HKT26}}]
\label{valid S}
    A \emph{valid} tuple $\mathcal{S}$ for a block-tuple $\mathcal{M}(c_{\min},c_{\max})$ is a tuple $(\mathcal{S}_1, \dots, \mathcal{S}_k)$ of subsets of $\{c_{\min}, \dots, c_{\max}\}$ satisfying the following conditions.
    \begin{enumerate} 
        \item Each $\mathcal{S}_i \subset \{c_{\min}, \dots, c_{\max} \}$ is a nonempty set of consecutive integers.
        \item $\bigcup_i \mathcal{S}_i = \{c_{\min}, \dots, c_{\max}\}.$
        \item If $i < j$ and $s_i \in \mathcal{S}_i, s_j \in \mathcal{S}_j,$ then $s_i \leq s_j$.
        \item If $i < j$ and $c \in \mathcal{S}_i \cap \mathcal{S}_j,$ then $j - i = 1,$ $|\mathcal{S}_j| \geq 2,$ and $m_c > 1.$
        \item If $i\geq 2$, $|\mathcal{S}_i| \geq 2$, and $c = \min \mathcal{S}_i$ has $m_c > 1,$ then $c \in \mathcal{S}_{i - 1}.$
\end{enumerate}
\end{defn}

Fix a block $\B\in\Block_\rho(G_n)$ and let $\mathcal{M}=\mathcal{M}_\B.$
To any valid tuple $\mathcal{S}$ for $\mathcal{M}$, a virtual extended multi-segment $\mathcal{E}(\mathcal{M}, \mathcal{S}, \eta)$ is associated as follows.

\begin{defn}[{\cite[Definition 5.2]{HKT26}}]\label{defn E(M,S)}
    Let $\mathcal{M}= (m_{c_{\min}}, m_{c_{\min} + 1}, \dots, m_{c_{\max}})$ be a block-tuple. Given a sign $\eta \in \{\pm 1\}$ and a valid tuple $\mathcal{S}=(\mathcal{S}_1,\dots,\mathcal{S}_k)$ for $\mathcal{M},$ we define a virtual extended multi-segment $\EE(\mathcal{M}, \mathcal{S}, \eta)$ consisting of the following rows.
    \begin{enumerate}
    \item For each $\mathcal{S}_i,$ we include a row of circles (i.e. a row with $l=0$) with support 
    $[\max \mathcal{S}_i, \min \mathcal{S}_i].$
    We refer to these rows as \emph{chains}.
    \item For each $c \in \{c_{\min}, \dots, c_{\max}\},$ we add $m_c - |\{i \mid c \in \mathcal{S}_i\}|$ copies of a row of circles with support
    $[c, c].$
    We refer to these rows as \emph{multiples}.
\end{enumerate}

The rows are ordered in a $(P')$ ordering such that if $A(r) < A(r')$ then $r < r'$. In the case where $r$ and $r'$ both have $supp(r) = supp(r') = [c, c]$ but $r$ is a chain and $r'$ is a multiple, we choose the order so that $r < r'.$

The signs $\eta$ are chosen to be ``odd-alternating'': i.e., the following conditions hold.
    \begin{enumerate}
        \item $\eta(\mathcal{E}(\mathcal{M}, \mathcal{S}, \eta)) = \eta.$
        \item If $r_i$ and $r_{i + 1}$ are consecutive rows and neither is a multiple, then $r_i$ and $r_{i + 1})$ satisfy the alternating sign condition.
        \item If $r_i$ and $r_{i + 1}$ are consecutive rows and $r_{i + 1}$ is a multiple, then $r_i$ and $r_{i + 1}$ fail the alternating sign condition.
        \item If $r_i$ and $r_{i + 1}$ are consecutive rows and only $r_{i}$ is a multiple, then $r_i$ and $r_{i+1}$ satisfy the alternating sign condition if and only if $B(r_{i + 1}) > B(r_i).$
    \end{enumerate}
\end{defn}

\begin{rmk}[{\cite[Remark 5.3]{HKT26}}]
    Note that the number of multiples $m_c - |\{\mathcal{S}_i \mid c \in \mathcal{S}_i\}|$ is chosen such that the number of circles in column $c$ is always equal to $m_c.$
\end{rmk}

We recall some terminology related to the study of the virtual extended multi-segments $\EE(\mathcal{M},\mathcal{S},\eta).$

\begin{defn}[{\cite[Definition 5.4]{HKT26}}]
\label{def z-chain}
    We say that two chains $r_1$ and $r_2$ are \emph{consecutive chains} if they are associated to two consecutive sets $\mathcal{S}_i$ and $\mathcal{S}_{i + 1}$. If $S_1 \cap S_{i + 1} \neq \emptyset,$ then we say that $r'$ is a \emph{$z$-chain}. An example is given in Figure \ref{fig: z-chain}.
\end{defn}

\begin{figure}[h]
    \centering
    $$\bordermatrix{&2 &3 & 4 & 5 & 6 \cr & \oplus & \ominus & \oplus \cr & & & \oplus \cr && & \oplus & \ominus & \oplus & \ominus}$$
    \caption{The $\mathcal{S}$ data for the above multi-segment is $(\{2, 3, 4\}, \{4, 5, 6\}).$ Since $\mathcal{S}_1 \cap \mathcal{S}_2 = \{1\},$ the chain with support $[4, 2]$ is a $z$-chain.}
    \label{fig: z-chain}
\end{figure}
    
\begin{defn}[{\cite[Definition 5.6]{HKT26}}]
\label{def multiples belonging to a chain}
    Let $r_1$ and $r_2$ be consecutive chains. If $r_3$ is a multiple such that $r_1 < r_3 < r_2,$ then we say that $r_3$ \emph{belongs} to $r_1.$ Note that from the sign condition in Definition \ref{defn E(M,S)}, we have that all multiples belonging to some chain $r$ have the same sign $\eta(r) \cdot (-1)^{C(r) - 1}.$
\end{defn}

\begin{rmk}[{\cite[Remark 5.8]{HKT26}}]\label{rmk odd signs}
    Suppose $\mathcal{E} = \mathcal{E}(\mathcal{M}, \mathcal{S}, \eta)$ as in Definition \ref{defn E(M,S)}. Suppose the circles of $\mathcal{E}$ are read out in order, row by row, then the symbols $\oplus$ or $\ominus$ always appear an odd number of times consecutively. For example, in Figure \ref{fig: z-chain}, there is one $\oplus,$ then three $\ominus$s, then one $\oplus,$ and then one $\ominus.$
\end{rmk}

For blocks $\mathcal{B}$ starting at zero (i.e., $c_{\min} = 0$), more data is required to describe all the virtual extended multi-segments equivalent to $\mathcal{B}.$ 
The extra data is recovered by introducing a new parameter $\mathcal{T}.$

\begin{defn}[{\cite[Definition 5.9]{HKT26}}]
\label{valid T}
    Suppose $\BB$ starts at zero and let $\mathcal{M} := \mathcal{M}_\mathcal{B}$ be the corresponding block-tuple. Let $\mathcal{S} = (\mathcal{S}_1, \dots, \mathcal{S}_k)$ be a valid tuple for $\mathcal{M}$. Then we say a collection of partitions of each $\mathcal{S}_i$ of the form $\mathcal{T}=(\mathcal{T}_i^0, \mathcal{T}_i^1, \dots,  \mathcal{T}_i^{\ell_i})$ is \emph{valid} if
    \begin{enumerate}
        \item each $\mathcal{T}_i^j$ is a nonempty set of consecutive integers;
        \item if $j < j'$ and $t_j \in \mathcal{T}_i^j$, $t_{j'} \in \mathcal{T}_i^{j'}$, then $t_j \leq t_j'$;
        \item if $|\mathcal{S}_i \cap \mathcal{S}_{i + 1}| \geq 1,$ then $|\mathcal{T}_{i + 1}^0| \geq 2$.
    \end{enumerate}
A \emph{valid} tuple $(\mathcal{M}, \mathcal{S}, \mathcal{T})$ is one such that $\mathcal{S}$ and $\mathcal{T}$ are both valid.
\end{defn}

\begin{rmk}[{\cite[Remark 5.10]{HKT26}}]
    For brevity, we often represent the subsets $\mathcal{T}_i^j \subset \mathcal{S}_i$ by overlining each $\mathcal{T}_i^j$ for $j > 0$. As an example, suppose $\mathcal{S}_i = \{2, 3, 4, 5, 6, 7, 8, 9\},$ and $\mathcal{T}_i^0 = \{2, 3, 4\}, \mathcal{T}_i^1 = \{5, 6\}, \mathcal{T}_i^2 = \{7\}, \mathcal{T}_i^3 = \{8, 9\}.$ As shorthand, we write
    $$\mathcal{S}_i = \{2, 3, 4, \overline{5, 6}, \overline{7}, \overline{8, 9}\}.$$
\end{rmk}

We describe the virtual extended multi-segments obtained with this new parameter below.

\begin{defn}[{\cite[Definition 5.11]{HKT26}}]
\label{defn E(M,S, T)}
    Let $\mathcal{M} = (m_{c_{\min}}, \dots, m_{c_{\max}})$ be a block-tuple with $c_{\min} = 0$. Given valid $(\mathcal{M}, \mathcal{S}, \mathcal{T})$ along with a sign $\eta \in \{\pm 1\}$, we associate a virtual extended multi-segment $\mathcal{E}(\mathcal{M}, \mathcal{S}, \mathcal{T}, \eta)$ as follows.
\begin{enumerate}
    \item For each $\mathcal{T}_i^j$ with $j \geq 1,$ we include a hat with support
    $[\max \mathcal{T}_i^j, -\min \mathcal{T}_i^j].$
    \item For each $\mathcal{T}_i^0 \subset \mathcal{S}_i,$ we include a row of circles with support
    $[\max \mathcal{T}_i^0, \min \mathcal{T}_i^0].$
    \item For each $c \in \{c_{\min}, \dots, c_{\max}\},$ we add $m_c - |\{i \mid c \in \mathcal{S}_i\}|$ copies of a row of circles with support
    $[c, c]$
    for some $\eta_i \in\{\pm 1\}.$
\end{enumerate}
The order of the rows and the signs of each row follow exactly the same rules as in Definition \ref{defn E(M,S)}.
\end{defn}

As before,  rows falling into Case (2) are called \emph{chains} and rows falling into Case (3) are called \emph{multiples}. Similarly, we have exactly the same notions of $z$-chains (Definition \ref{def z-chain}) and a multiple belonging to a chain (Definition \ref{def multiples belonging to a chain}) as before.

\begin{defn}[{\cite[Definition 5.12]{HKT26}}]\label{defn type Y M}
    We say that $\EE$ is of \emph{type $Y_\mathcal{M}$} if it is of the form $\EE(\mathcal{M}, \mathcal{S}, \mathcal{T}, \eta)$ for valid $(\mathcal{M}, \mathcal{S}, \mathcal{T})$ or $\EE(\mathcal{M}, \mathcal{S}, \eta)$ for valid $(\mathcal{M}, \mathcal{S})$. We refer to $\mathcal{S}$ (or $\mathcal{S}$ and $\mathcal{T}$, if applicable) as the \emph{$\mathcal{S}$-data} of $\EE.$
\end{defn}

As it turns out, multi-segments of type $Y_\mathcal{M}$ completely classify multi-segments corresponding to tempered representations. This fact is formalized in the following theorem.

\begin{thm}[{\cite[Theorem 5.17]{HKT26}}]
\label{block classification}
    Let $\mathcal{B} \in\VRep_\rho^\mathbb{Z}(G_n)$ be a block. Then, the set $\Psi(\pi(\BB))$ is precisely the set of $\psi_\EE$ where $\EE\in\VRep_\rho^\mathbb{Z}(G_n)$ is of type $Y_\mathcal{M}$ with $\mathcal{M} = \mathcal{M}_\BB$ and $\eta(\EE)=\eta(\BB)$.
\end{thm}

Suppose that $\EE\in\VRep_\rho^\mathbb{Z}(G_n)$ is of type $Y_\mathcal{M},$ with $\mathcal{M}$ starting at zero.
    Theorem \ref{block classification} guarantees that there exists a tempered virtual extended multi-segment $\EE_{temp}$ equivalent to $\EE.$

    The next result describes several important operators on multi-segments of form $Y_\mathcal{M}$ that preserve equivalence, in the case where $\mathcal{M}$ begins at zero.

\begin{lemma} \label{lem SMUD operators}
    Suppose $\EE = \mathcal{E}(\mathcal{M}, \mathcal{S}, \mathcal{T},\eta)$ with $\mathcal{M}$ beginning at zero. Then the following operations $S, M, U, D$ changing the $\mathcal{S}$-data of $\mathcal{M}$ while creating an equivalent extended multi-segment.
    
    \begin{enumerate}
        \item Let $r = ([A, B], 0, \eta)$ be a chain and $k < C(r)$ be an integer. If $r$ is a $z$-chain and $c = A - B,$ then $S_{r, k}$ replaces 
        $$\{B, \dots, A, \dots\} \longrightarrow \{B + 1, \dots, A, \dots\}.$$ Otherwise, $S_{r, k}$ splits $\mathcal{S}_i$ into two sets \[\mathcal{S}_i^1 = \{B, \dots, A - c\} \text{ and } \mathcal{S}_i^2 = \{A - c + 1, \dots, A, \dots\}.\]

        \item If $h_1 = ([A, -B], B, \eta_1)$ and $h_2 = h_2 = ([-B + 1, C], -C, \eta_2)$ are consecutive hats, then $M_{h_1, h_2}$ replaces
        $$\{\dots \overline{C, \dots, B - 1}, \overline{B, \dots, A}, \dots \} \longrightarrow \{\dots \overline{C, \dots, B - 1, B, \dots, A}, \dots \}.$$

        \item If $h = ([A, -B], B, \eta)$ is a hat in $\EE$ with $C(h) > c$, then $U_{h, c}$ replaces
        $$\{\dots \overline{B, \dots, A}, \dots \} \longrightarrow \{\dots \overline{B, \dots, A - c} \}, \{A - c + 1, \dots, A, \dots \}.$$

        \item Let $h = ([A, -B], B, \eta)$ be a hat and $r_1, r_2, \dots, r_k$ be the rows of circles whose support contains $B - 1.$ Then
        \begin{enumerate}
            \item $D_{h, r}^1$ replaces
            $$\{C, \dots, B - 1, \overline{B, \dots, A}, \dots\} \longrightarrow \{C, \dots, B - 1, B, \dots, A, \dots\}.$$
            \item If $C(r_1) = 1,$ then $D_{h, r_1}^1$ and $D_{h, r_k}^2$ have the same effect on the $\mathcal{S}$-data. Otherwise, $D_{h, r_k}^2$ replaces
            $$\{\dots, B - 1, \overline{B, \dots, A}, \dots\} \longrightarrow \{\dots, B - 1\}, \{B - 1, \dots, A, \dots\}.$$
        \end{enumerate}
    \end{enumerate}
    Moreover, any multi-segment equivalent to $\mathcal{M}$ can be obtained through some series of these operations and their inverses.
\end{lemma}

The above lemma is an amalgamation of Remarks 5.30, 5.34, 5.40, and 5.42 of \cite{HKT26}, with the final assertion following from the above Theorem \ref{block classification}. Meanwhile, if $\EE$ is of type $Y_\mathcal{M}$ with $\mathcal{M}$ beginning after zero, the $S$ operator is applicable, but the others cannot be applied.

\begin{lemma}[{\cite[Lemma 5.28]{HKT26}}] \label{S for E starting after 0}
    Suppose $\EE = \mathcal{E}(\mathcal{M}, \mathcal{S}, \mathcal{T},\eta)$ with $\mathcal{M}$ beginning after zero. Then the operator $S$ defined in Lemma \ref{lem SMUD operators} is well defined. Moreover, any multi-segment equivalent to $\mathcal{M}$ can be obtained through some series of $S$-operators and their inverses.
\end{lemma}

The next definition will formally represent applying the theta correspondence on the level of extended multi-segments.

\begin{defn}[{\cite[Definition 6.2]{HKT26}}]\label{defn Theta_1(EE)}
    
    Let $\EE\in\VRep_\rho^\mathbb{Z}(G_n)$ be of type $Y_{\mathcal{M}}$ with $\mathcal{M}$ starting at zero. We define
    \[
    \Theta_1(\EE)=\left\{\left(\left[c_{\max}+1,-c_{\max}-1\right]_{\rho}, {c_{\max}+1}, -\eta(\EE) \right)\right\}\cup \EE.
    \]
    
    We remark that the added extended segment should be inserted such that it is the first extended segment in the admissible order. 
\end{defn}

Adding the new row to get from $\EE$ to $\Theta_1(\EE)$ enables new operations involving the new row. These are completely classified by the following two theorems. The first theorem deals with the case where $\mathcal{M} = (m_0, \dots, m_{c_{\max} - 1}, 1),$ while the second theorem deals with the case when the final multiplicity is strictly greater than $1.$ 

\begin{thm}[{\cite[Theorem 6.4]{HKT26}}]
\label{first block m_n = 1}
    Let $\mathcal{M} = (m_0, \dots, m_{c_{\max} - 1}, 1)$ begin at zero and $\EE$ be of type $Y_\mathcal{M}$. Then the following holds.
    \begin{enumerate}
        \item  We can perform a $dual \circ ui \circ dual$ involving the first row of $\Theta_1(\EE)$ to get a virtual extended multi-segment which we denote by $\Theta_2(\EE).$ We can further perform an additional $ui^{-1}$ involving the $(c_{\max} + 1)$-st column of $\Theta_2(\mathcal{E})$ to obtain another virtual extended multi-segment which we denote by $\Theta_3(\EE).$
        \item We have the relation $$\Psi(\Theta_1(\EE)) = \{\psi_{\Theta_i(\EE')} \mid \mathcal{E}' \sim \mathcal{E}; i = 1, 2, 3 \}.$$
    \end{enumerate}
\end{thm}

\begin{thm}[{\cite[Theorem 6.5]{HKT26}}]
\label{first block m_n > 1}
    Let $\EE$ be of type $Y_\mathcal{M}$ where $\mathcal{M} = (m_0, \dots, m_{c_{\max}})$ begins at zero and $m_{c_{\max}} > 1.$
    Then the following holds.
    \begin{enumerate}
        \item We can perform two $dual \circ ui \circ dual$s of type 3' involving the first row of $\Theta_1(\mathcal{E})$: one with the first row whose support ends at $c_{\max},$ and one with the last row whose support ends at $c_{\max}.$ We call the resulting segments $\Theta_2(\mathcal{E})$ and $\Theta_4(\mathcal{E}),$ and we have $\psi_{\Theta_2(\mathcal{E})} = \psi_{\Theta_4(\mathcal{E})}$ if and only if the $\mathcal{S}$-data of $\EE$ has some $\mathcal{S}_i = \{c_{\max}\}.$ We can perform an additional $ui^{-1}$ to both $\Theta_2$ and $\Theta_4$ to remove a row of one circle from the row with support ending at $c_{\max} + 1$ to obtain the same $\Theta_3(\EE).$
        \item We have the relation $$\Psi(\pi(\Theta_1(\EE))) = \{\psi_{\Theta_i(\EE')} \mid \mathcal{E}' \sim \mathcal{E}; i = 1, 2, 3, 4 \}.$$
    \end{enumerate}
\end{thm}

\section{Proof of the main results}\label{sec theta results}

Here, we use our previous results about tempered representations to provide proofs for Theorem \ref{thm-count-block-theta-temp}, Cases 1, 2, 3, 4, and 6 of Theorem \ref{thm-count-theta-temp}, and Theorem \ref{theta-antitempered-count}.

\subsection{Theta Correspondence for Almost-Blocks}\label{sec Theta Correspondence for Almost-Blocks}

Suppose that $\EE\in\Rep(G_n)$ is tempered and let $\pi=\pi(\EE)$. For brevity, we write $m^{\up, \alpha}=m^{\up, \alpha}(\pi).$ We are interested in determining $\Psi(\theta_{-m^{\up, \alpha}}^\up(\pi(\mathcal{E}))$ using $\Psi(\pi).$ We decompose $\EE=\cup_{i=1}^p\EE_{\rho_i}$ where $\rho_i\neq\rho_j$ for any $i\neq j.$  First, we note by Theorem \ref{thm intersections of local Arthur packets} that 
\[
\Psi(\pi(\EE))=\prod_{i=1}^p\Psi(\pi(\EE_{\rho_i})).
\]

We begin by remarking on a trivial case.
\begin{rmk}\label{rmk theta lift when no chi_v}
First suppose that $\rho_i\neq\chi_V$ for any $i=1,\dots,p.$ By Theorem \ref{thm going up first occurrence}, we have that $\up=\down$ and $m^{\up, \alpha}(\pi)=m^{\down, \alpha}(\pi)=1$ and so the distinction between the going-up and going-down towers does not matter. From Theorem \ref{thm compute theta lift}, we find that
\[
\Psi(\theta_{-1}^{\pm}(\pi(\mathcal{E}))=\{\psi_1 \ | \ \psi\in\Psi(\pi)\},
\]
where $\psi_1=\psi_\alpha$ where $\alpha=1$ (see Conjecture \ref{conj Adams}).     
\end{rmk}

Suppose that some $\rho_i=\chi_v,$ say $\rho_1=\chi_V.$ By Theorems \ref{thm intersections of local Arthur packets} and \ref{thm compute theta lift}, we have that
\[
|\Psi(\theta_{-m^{\up, \alpha}}^\up(\pi(\mathcal{E})))|=|\Psi(\pi((\EE_{\chi_V})_{m^{\up, \alpha}}^\up)))|\prod_{i=2}^p|\Psi(\pi(\EE_{\rho_i}))|,
    \]
where $\pi((\EE_{\chi_V})_{m^{\up, \alpha}}^\up)$ is defined by Algorithm \ref{algo compute theta lift}. Thus, it suffices to determine $|\Psi(\pi((\EE_{\chi_V})_{m^{\up, \alpha}}^\up)))|.$

We proceed in two cases based on whether $\EE_{\chi_V}$ starts at 0 or not. We begin with the latter case. Note that, by Theorem \ref{thm going up first occurrence}, $m^{\up,\alpha}(\pi)=m^{\down,\alpha}(\pi)=1$ in this case and so again the distinction between the going-up and going-down towers does not matter.

\begin{lemma}\label{lemma-EE-does-not-start-at-zero}
    If $\EE_{\chi_V}$ does not start at $0$, then $\theta_{-1}^\pm(\pi(\mathcal{E}))$ is tempered. In particular,  
     $(\EE_{\chi_V})_1^\pm$ is tempered and so $|\Psi(\pi((\EE_{\chi_V})_1^\pm))|$ can be computed by Theorem \ref{thm-count-block-temp}.
\end{lemma}

\begin{proof}
    First, we note that \cite[Theorem B]{BH22} implies that $(\EE_{\chi_V})_1^\pm\in\Rep(H_{n+1}^\pm).$ By Algorithm \ref{algo compute theta lift} and the fact that $\EE_{\chi_V}$ does not start at 0, it follows that $\EE_1^\pm$ is tempered which further implies that $\theta_{-1}^\pm(\pi(\mathcal{E}))$ is tempered.
\end{proof}

Hereinafter, we treat the case that $\EE_{\chi_V}$ starts at 0. By Theorem \ref{thm going up first occurrence}, we have that $\up=-\eta(\EE_{\chi_V})$ and $\up\neq\down.$ Furthermore, we have that $(\EE_{\chi_V})_{m^{\up, \alpha}}^\up=\Theta_1(\EE_{\chi_V})$ in this setting (see Definition \ref{defn Theta_1(EE)}). Thus, it suffices to study $|\Psi(\Theta_1(\EE_{\chi_V}))|.$

This case naturally breaks into two cases based on the block decomposition of $\EE_{\chi_V}$.

\begin{lemma}
Suppose that $\EE_{\chi_V}$ consists of a single block starting at $0$. Then $\Theta_1(\EE_{\chi_V})$ is equivalent to an almost block and hence $|\Psi(\Theta_1(\EE_{\chi_V}))|$ is determined by Theorem \ref{thm-count-block-temp}.
\end{lemma}
\begin{proof}
This is a direct consequence of Theorem \ref{block classification}.
\end{proof}

 In the case where $\mathcal{E}$ is a single block starting at zero, then it follows from Theorem \ref{block classification} that the theta lift $\theta_{-m^{\up, \alpha}}^\up(\pi(\mathcal{E}))$ is itself a tempered representation, since it is of type $Y_\mathcal{M}$ for some $\mathcal{M}$, so the count $|\Psi(\theta_{-m^{\up, \alpha}}^\up(\pi(\mathcal{E})))|$ follows immediately from Theorem \ref{thm-count-block-temp}. 

Next we consider consider the case where $\mathcal{E}_{\chi_V}=\mathcal{E}' \cup r,$ where $\mathcal{E}'$ is of type $Y_\mathcal{M}$ and $r = ([c_{\max}, c_{\max}]_{\chi_V}, 0, \eta)$ for $\eta$ equaling the sign of the last circle of $\mathcal{E}'.$

\begin{lemma}
\label{duds for almost block}
    Suppose $\mathcal{E}_{\chi_V} = \mathcal{E}' \cup ([c_{\max}, c_{\max}]_{\chi_V}, 0, \eta)$ is as described above. Let $\Theta_1(\mathcal{E}_{\chi_V}) = h \cup \mathcal{E}'',$ where $h = ([c_{\max} + 1, -c_{\max} - 1]_{\chi_W}, n + 1, -\eta(\mathcal{E}))$ and $\EE''$ is obtained from $\EE$ by replacing each $\chi_V$ by $\chi_W.$ Then we can perform precisely two $dual \circ ui \circ dual$s (up to row exchange) combining the first row of $\Theta_1(\mathcal{E})$: one with the first row whose support ends at $c_{\max},$ and one with the last row whose support ends at $c_{\max}.$ We call the resulting segments $\Theta_2(\mathcal{E})$ and $\Theta_4(\mathcal{E}),$ and we have $\psi_{\Theta_2(\mathcal{E}_{\chi_V})} = \psi_{\Theta_4(\mathcal{E}_{\chi_V})}$ if and only if $\EE'$ has $\{c_{\max}\}$ in its $\mathcal{S}$-data.
\end{lemma}

\begin{proof}
The proof is the same as the proof of \cite[Lemma 5.39]{HKT26}. Indeed, the proof there holds even when the multiplicities are not odd.   
\end{proof}

By an analogous proof as that of \cite[Lemma 6.7]{HKT26}, we have that all the elements of the set $\{\psi_{\Theta_i(\EE)} \mid  \psi_\EE \in \Psi(\pi), i \in \{1, 2, 4\}\}$ are distinct, except for the case where $\psi_{\Theta_2(\mathcal{E})} = \psi_{\Theta_4(\mathcal{E})},$ as specified in Lemma \ref{duds for almost block}. In order to prove Theorem \ref{thm-count-block-theta-temp}, it suffices to prove that the elements in this set are the only Arthur packets that $\pi(\Theta_1(\mathcal{E}))$ belongs to.

\begin{lemma}
\label{Almost block classification}
    If $\mathcal{E}$ is an almost block with $c_{\max}$ columns, then
    $$\Psi(\theta_{-m^{\up, \alpha}}^\up(\pi(\mathcal{E}))) = \{\psi_{\Theta_i(\EE')} \mid  \psi_{\EE'} \in \Psi(\pi(\mathcal{E})), i \in \{1, 2, 4\}\}.$$
\end{lemma}

\begin{proof}
    As shorthand, let $P$ denote the set $P = \{\psi_{\Theta_i(\EE')} \mid  \psi_{\EE'} \in \Psi(\pi(\mathcal{E})), i \in \{1, 2, 4\}\}.$ Lemma \ref{duds for almost block} gives us that $\Psi(\theta_{-m^{\up, \alpha}}^\up(\pi(\mathcal{E}))) \subset P,$ so we need only prove the reverse direction. It suffices to show that $P$ is closed under the basic operators, i.e., if $\psi_{\Theta_i(\mathcal{E}')} \in P,$ then we claim that there is no raising or lowering operator from $\Theta_i(\mathcal{E}')$ that does not induce another element of $P.$ To observe this, we split into several cases.

    If $i = 1,$ then $\Theta_i(\mathcal{E}')$ has the hat $h = ([c_{\max} + 1, -c_{\max} - 1], c_{\max} + 1, -\eta(\mathcal{E}))$ as its first row. It is clear that any operator not involving $h$ produces a multi-segment in $P$ with $i = 1.$ Meanwhile, any operator involving $h$ must be a $dual \circ ui \circ dual$ and end up producing either $\Theta_2(\mathcal{E}')$ or $\Theta_4(\mathcal{E}')$ by Lemma \ref{duds for almost block}.
    
    If $i = 2$ or $i = 4,$ it is impossible to execute a $ui^{-1}$ to $\Theta_3(\mathcal{E}')$ analogous to the one given in \cite[Theorem 6.5]{HKT26}. This is because exchanging the row whose support ends at $c_{\max} + 1$ to the bottom of the extended multi-segment produces a row with two triangles rather than a row of circles, due to the extra multiple with support $[c_{\max}, c_{\max}]$.
    
    To see that any operator not yet considered is closed under $P$, consider the extended multi-segment $\mathcal{E}' \cup r,$ where $r$ is an extra multiple of support $[c_{\max}, c_{\max}]$ (with $\eta(r)$ chosen so that $\mathcal{E}' \cup r$ is nonvanishing). This multi-segment must be of the form $\Theta_i(\mathcal{E}''),$ where $i \in \{2, 4\}$ and $\mathcal{E}''$ is of type $Y_{\mathcal{M}}$ for $\mathcal{M}$ ending at $c_{\max}$ with odd multiplicities. Any operator on the sub-virtual extended multi-segment $\mathcal{E}' \subset \Theta_i(\mathcal{E}'')$ must induce an operator on $\Theta_i(\mathcal{E}'').$ This is because the only operator that could interact with the added row $r$ is a $ui^{-1}$ separating one circle of support $[c_{\max} + 1, c_{\max} + 1],$ and we have seen that this operator does not exist. Since $i \in \{2, 4\}$ and any operator on $\mathcal{E}'$ cannot ascend to an operator from $\Theta_i$ to $\Theta_3,$ it is clear that any operator on $\mathcal{E}'$ must produce some other $\Theta_i$ with $i \neq 3$, which belongs to the set $P.$
\end{proof}

\subsection{Cases 2, 4, and 6 of Theorem \ref{thm-count-theta-temp}}\label{sec Cases 1, 2, 4}

First, suppose that $\mathcal{E}\in\VRep_{\chi_V}(G_n)$ is a multi-segment as described in Cases 2, 4, and 6 of Theorem \ref{thm-count-theta-temp}. We will prove that $\mathcal{E}_{m^{\up, \alpha}}^\up$ is tempered, and this along with the counts described in Theorem \ref{thm-count-temp} will imply the desired results.

If $\mathcal{B}_1, \dots, \mathcal{B}_k$ is the block decomposition of $\mathcal{E},$ we have that $\mathcal{B}_2$ begins at least one column after $\mathcal{B}_1$ ends. It follows then from Theorem \ref{thm going up first occurrence} that
$$(\mathcal{B}_1 \cup \mathcal{B}_2 \cup \cdots \cup \mathcal{B}_k)_{m^{\up, \alpha}}^\up = (\mathcal{B}_1)_{m^{\up, \alpha}}^\up \cup \mathcal{B}_2 \cup \cdots \cup \mathcal{B}_k.$$ Yet, according to Theorems \ref{first block m_n = 1} and \ref{first block m_n > 1}, $(\mathcal{B}_1)_{m^{\up, \alpha}}^\up$ can be changed into a new block $\mathcal{B}_1'$ with one more column than $\mathcal{B}$ via a $dual \circ ui \circ dual$ and then a $ui^{-1}.$ These operations can be implemented on the first component of $(\mathcal{B}_1)_{m^{\up, \alpha}}^\up \cup \mathcal{B}_2 \cup \cdots \cup \mathcal{B}_k,$ giving
$$(\mathcal{B}_1 \cup \mathcal{B}_2 \cup \cdots \cup \mathcal{B}_k)_{m^{\up, \alpha}}^\up = \mathcal{B}_1' \cup \mathcal{B}_2 \cup \cdots \cup \mathcal{B}_k.$$

In the event where $\mathcal{B}_1'$ and $\mathcal{B}_2$ share a column, we must have that the circles in this column have the same sign by the nonvanishing conditions given in \cite[Theorems 3.6, 4.4]{Ato20b} (which also holds for even orthogonal groups as a consequence of \cite[Theorem 6.8]{HLL25}). Since each component $\mathcal{B}_1'$ is tempered, we conclude that $\pi(\Theta_3(\mathcal{E}))$ is a tempered representation, from which we obtain Cases 2, 4, and 6 of Theorem \ref{thm-count-theta-temp}.

\begin{rmk}
    Note that $\mathcal{B}_1' \cup \mathcal{B}_2 \cup \cdots \cup \mathcal{B}_k$ is not necessarily itself the block decomposition of the theta lift. In the case where $\mathcal{B}_1'$ and $\mathcal{B}_2$ have an overlapping column, all but one of the circles in the first row of $\mathcal{B}_2$ need to be transferred to the end of $\mathcal{B}_1',$ thereby replacing them with two new blocks $\mathcal{B}_1''$ and $\mathcal{B}_2.$ Yet, \cite[Lemma 6.8]{HKT26} gives $\Psi(\pi(B_1'')) = \Psi(\pi(B_1'))$, while the formula in Theorem \ref{thm-count-block-temp} guarantees that $\Psi(\pi(B_2')) = \Psi(\pi(B_2))$. This, along with Proposition \ref{prop-independence-of-blocks}, implies the formula

    \[|\Psi(\theta_{-m^{\up, \alpha}}^\up(\pi(\EE)))| = |\Psi(\theta_{-m^{\up, \alpha}}^\up(\pi(\mathcal{B}_1)))| \cdot \prod_{i=2}^k |\Psi(\pi(\mathcal{B}_i))|.\]

    Such a formula is not generally possible for Cases 2 and 4, as described in Remark \ref{rmk cases 1 and 4}.
\end{rmk}

\subsection{Case 3 of Theorem \ref{thm-count-theta-temp}}\label{sec case 3}

Throughout this subsection, we often omit the orthogonal supercuspidal representations, e.g.,  $\chi_V$ and $\chi_W$. 

In order to prove Case 3 of Theorem \ref{thm-count-theta-temp}, we will first specify a modification of type $Y_{\mathcal{M}},$ which we will call type $Z_{\mathcal{M}}.$ Let $k_1$ and $k_2$ be integers where $0 < k_1 < k_2.$ Let $\mathcal{M} = (m_0, m_1, \dots, m_{k_2})$ be a tuple of positive integers. Here, we require that $m_i \in 2 \mathbb{Z}$ if and only if $i \in \{k_1, k_1 + 1\}.$ Again, we set a collection of sets $\mathcal{S} = \{\mathcal{S}_1, \dots, \mathcal{S}_\ell\}$ and subsets $\mathcal{T}$ with the same rules as before as in Definitions \ref{valid S} and \ref{valid T}, with the additional requirement that there exists no $i$ such that $\mathcal{T}_i^0 = \{k_1\}.$ We construct an associated $\mathcal{E}(\mathcal{M}, \mathcal{S}, \mathcal{T}, \eta)$ exactly as before as in Definition \ref{defn E(M,S, T)}, except with the following differences:

\begin{itemize}
    \item If $\min \mathcal{T}_i^0 = k_1,$ we associate to $\mathcal{T}_i^0$ the row $([\max \mathcal{T}_i^0, \min \mathcal{T}_i^0], 1, \eta)$ rather than $([\max \mathcal{T}_i^0, \min \mathcal{T}_i^0], 0, \eta).$ This chain should alternate in sign with the previous row.
    \item If $r = ([A, k_1 + 1], 0, \eta)$ is a chain beginning at $k_1 + 1,$ then $r$ should not alternate with the previous row.
\end{itemize}

\begin{rmk}
    Note that the first bullet point makes sense since there exists no $\mathcal{T}_i^0 = \{k_1\},$ so any row corresponding to a $\mathcal{T}_i^0$ with $\min \mathcal{T}_i^0 = k_1$ has a support of length at least two.
\end{rmk}

\begin{rmk}
    As a means of shorthand, we will refer to a multi-segment of Type $Z_\mathcal{M}$ as being of Type $Z_{k_1, k_2}.$
\end{rmk}

\begin{exmp}
    Let $k_1 = 1$ and $k_2 = 4.$ Then the virtual extended multi-segment associated to $\eta = 1, \mathcal{M} = (3, 2, 2, 1, 3, 1)$ and $\mathcal{S} = (\{0\}, \{1, 2, 3, \overline{4}\}, \{5\})$ is

    $$\bordermatrix{& -4 & -3 & -2 & -1 & 0 & 1 & 2 & 3 & 4 & 5 \cr & \lhd & \lhd & \lhd & \lhd & \oplus & \rhd & \rhd & \rhd & \rhd \cr & & & & & \ominus \cr & & & & & \ominus \cr & & & & & \ominus \cr & & & & & & \oplus \cr & & & & & & \lhd & \ominus & \rhd \cr & & & & & & & \ominus \cr & & & & & & & & \ominus \cr & & & & & & & & &  \ominus \cr & & & & & & & & &  \ominus \cr & & & & & & & & &  \ominus \cr & & & & & & & & & & \oplus}$$

    Here, the chain with support $[3, 1]$ has $l = 1$ and alternates with the row above it, in contrast to what would happen in a multi-segment of type $Y_\mathcal{M}.$
\end{exmp}

\begin{exmp}
    Suppose $\mathcal{E}$ is a tempered multi-segment satisfying $m_H \equiv 0 \bmod{2}$, $m_N \equiv 1 \bmod{2}$ as in Case 3 of Theorem \ref{thm-count-theta-temp}. Then $\mathcal{E}_{m^{\up, \alpha}}^\up$ is of type $Z_{k_1, k_2}$ with $\mathcal{S}$-data $(\{0\}, \{1\}, \dots, \{k - 1\}, \{k, \overline{k + 1}\}, \{k + 2\}, \dots, \{n\}).$ We provide an example below.
    \[\EE = \bordermatrix{
    & 0 & 1 & 2 & 3 \cr
    & \ominus & & & \cr
    & & \oplus & & \cr
    & & \oplus & & \cr
    & & & \oplus & \cr
    & & & & \ominus \cr
    } \longrightarrow \mathcal{E}_{m^{\up, \alpha}}^\up =  \bordermatrix{
    & -2 & -1 & 0 & 1 & 2 & 3 \cr
    & \lhd & \lhd & \oplus & \rhd & \rhd & \cr
    &&& \ominus & & & \cr
    &&& & \oplus & & \cr
    &&& & \oplus & & \cr
    &&& & & \oplus & \cr
    &&& & & & \ominus \cr
    }\]

    Note that in this example, the chain with support $[2, 2]$ satisfies the necessary sign condition by not alternating with the previous row.
\end{exmp}

We now briefly classify all operators on extended multi-segments $\mathcal{E}$ of type $Z_{k_1, k_2}.$ In short, we will verify that the operators $S, M, U, D$ (as defined for type $Y_{\mathcal{M}}$ in Lemma \ref{lem SMUD operators}) are still well defined, preserve type $Z_{k_1, k_2}$,  $\eta(\mathcal{E})$, and have the same effect on the $\mathcal{S}$-data of type $Z_{k_1, k_2}$ multi-segments. We will then verify that all other operators preserve type $Z_{k_1, k_2}$. Much of the proof will be exactly the same as the analogous proofs for type $Y_\mathcal{M}$ in \cite[\S 5]{HKT26}, so some details will be omitted. 

The proof that the operator $M$ behaves as prescribed for $Z_{k_1, k_2}$ extended multi-segments carries over immediately for the proof in the case of type $Y_\mathcal{M}$ (see \cite[Lemma 5.33]{HKT26}). Therefore, we will start with the operator $S.$

\begin{lemma}
\label{S existence Z}
    Suppose $\EE$ is of type $Z_{k_1, k_2}$, $r = ([A, B], 0, \eta)$ is a chain associated to the set $\mathcal{S}_i = \{B, \dots, A, \dots\}$ and $c < C(r)$ is a positive integer. Then there exists an operator, notated $S_{r, c},$ preserving type $Z_{k_1, k_2}$ and $\eta(\mathcal{E})$ with the following effect on the $\mathcal{S}$-data:
    \begin{enumerate}
        \item If $r$ is a $z$-chain (see Definition \ref{def z-chain}) with $A - B = c$, then $S_{r, c}$ replaces $\mathcal{S}_i$ by the set $S_i^2 = \{B + 1, \dots, A, \dots\}.$
        \item Otherwise, $S_{r, c}$ splits $\mathcal{S}_i = \{B, \dots, A, \dots\}$ into two sets \[\mathcal{S}_i^1 = \{B, \dots, A - c\} \text{ and } \mathcal{S}_i^2 = \{A - c + 1, \dots, A, \dots\}.\]
    \end{enumerate}
    Moreover, if $l(r) = 1,$ then $S_{r, c}$ is defined even when $c = C(r).$
\end{lemma}

\begin{proof}
    We prove the second case of the lemma. The first case follows in exactly the same way, except we remove the set $S_i^1 = \{B\}$ from the $\mathcal{S}$-data since this row would represent a multiple rather than a chain. 
    
    The proof that $S_{r, c}$ is possible carries over from \cite[Lemma 5.28]{HKT26} except when the row $r$ has to be exchanged with rows whose supports contain $k_1$ or $k_1 + 1$. Such a step is only a part of the $S$ operation whenever one of these multiples belongs to $r.$ We must fall into one of the following cases.
     \begin{enumerate}
         \item Both the multiples with support $[k_1, k_1]$ and $[k_1 + 1, k_1 + 1]$ belong to $r$ and the operation $S$ involves row exchanging $r$ with all of these multiples. There is an odd number of each of these, and they have the same sign, so exchanging $r$ with these multiples does not change $r$ by Lemma \ref{lem-multiplicity-cancel}. We must necessarily have that $l(r) = 0$ by Lemma \ref{big swap down} so a $ui^{-1}$ of type 3' may be applied after $r$ is exchanged down. Row exchanging up again has no effect by Lemma \ref{lem-multiplicity-cancel-up}. 
         \item Both the multiples with support $[k_1, k_1]$ and $[k_1 + 1, k_1 + 1]$ belong to $r$ but the operation $S$ involves row exchanging $r$ with only the multiples with support $[k_1, k_1].$ Here, we apply $ui^{-1}$ of type 3' to separate some circles from $r$ to obtain a row of circles $r'$ with support $[A, k_1].$ We then row exchange $r'$ until it is right below all the multiples with support $[k_1, k_1].$ Either $A = k_1$ and the resulting row is then another such multiple, or $A > k_1,$ in which case exchanging with an odd number of multiples produces a chain with $l = 1$ as required, by Lemma \ref{big swap down}.
         \item Only the multiples with support $k_1 + 1$ belong to $r,$ in which case $r$ must be a row with support $[A, k_1]$ for some $A \geq k_1 + 1.$ Such a row must have $l = 1$, so we must actually have $A \geq k_1 + 2$, lest $C(r) = 0$. Then, exchanging $r$ with the odd number of multiples of support $[k_1 + 1, k_1 + 1]$ yields a row with $l = 0$ by Lemma \ref{big swap down}, so a $ui^{-1}$ may be applied. 
     \end{enumerate}

     The proof that sign conditions are conserved proceeds analogously to the proof of \cite[Lemma 5.28]{HKT26}. Note that none of these separations can produce some set $\mathcal{T}_0^i = \{k + 1\},$ so the resulting virtual extended multi-segment is still of type $Z_{k_1, k_2}.$
\end{proof}

Next, we verify the existence of the $D$ operator.

\begin{lemma}
\label{D existence Z}
    Given $\mathcal{E}$ of type $Z_{k_1, k_2},$ the operator $D$ is well defined and affects the $\mathcal{S}$-data the same as in the $Y_\mathcal{M}$ case (see Lemma \ref{lem SMUD operators}).
\end{lemma}

\begin{proof}
    Suppose $h = ([A, -B], B, \eta)$ be the hat to be dualized by the $D$ operator. Firstly, if $A \leq k_1,$ then all possible $dual \circ ui \circ dual$ operations of type 3' are contained in a sub-virtual extended multi-segment of type $Y_\mathcal{M}$ with $c_{\max} = A$, from which the result follows immediately from Lemma \ref{lem SMUD operators}.
    
    If $A = k_1 + 1,$ then $h$ can either be dualized to a chain $r$ ending at $k_1$ or a multiple with support $[k_1, k_1]$ (these virtual extended multi-segments being equivalent by row-exchange if and only if $C(r) = 1$). By Lemma \ref{lem SMUD operators}, dualizing to the bottom multiple with support $[k_1, k_1]$ will produce a row $([k_1 + 1, k_1], 1, \eta)$ as desired.

    If $A \geq k_1 + 2,$ let $r_1$ be the chain ending at $A - 1$ (if such a chain exists) and let $r_2$ be the final multiple with support $[A - 1, A - 1].$ If we want to dualize $h$ to $r_1$, we may suppose by \cite[Corollary 5.32 and Lemma 6.8]{HKT26} that all rows between $h$ and $r_1$ are unmerged and none of them are hats. We have the following four cases.

    \begin{enumerate}
        \item Suppose  $r_1$ has support $[A - 1, B]$ with $B > k_1 + 1.$ If there exists a chain beginning at $k$ with $l = 1,$ then we row exchange the chain up so that it is the first row beginning at $k$ and it has $l = 0$ by Lemma \ref{big swap up}. Now, there are an odd number of rows of circles with support $[k_1, k_1]$ and an odd number with support $[k_1 + 1, k_1 + 1],$ all above $r_1.$ By Lemmas \ref{lem-multiplicity-cancel} and \ref{lem-multiplicity-cancel-up}, we can suppose that these evenly many multiples are not present since they do not affect row exchanges. This reduces to the $Y_\mathcal{M}$ case from which the claim follows by Lemma \ref{lem SMUD operators}.
        \item If $r_1$ has support $[A - 1, k_1 + 1],$ then suppose $\mathcal{S}_i = \{k_1 + 1, \dots, A - 1, \overline{A, \dots }, \dots\}$ be the associated set. But $\mathcal{S}_{i - 1}$ must also contain $k_1 + 1$ because $m_{k_1 + 1} > 1.$ Since all rows between $h$ and $r_1$ are as separated as possible, we must have $\mathcal{S}_{i - 1} = \{k_1, k_1 + 1\}$ (since we cannot have $\mathcal{T}_{i - 1}^0 = \{k_1 + 1\}.$ The associated chain is of the form $([k_1 + 1, k_1], 1, \eta').$ Removing this chain does not impact row exchanges and reduces to the $Y_\mathcal{M}$ case.
        \item If $r_1$ has support $[A - 1, k_1],$ then $r_1 = ([A - 1, k_1], 1, \eta)$ for some $\eta.$ In this case, row exchange $r_1$ up with one of the rows with support $[k_1, k_1]$ so it has $l = 0$ by Lemma \ref{big swap up}. Now the rows between $k_1$ and $r_1$, inclusive, constitute a sub-virtual extended multi-segment of type $Y_\mathcal{M}$. Perform the dualization and exchange the chain back down, restoring the value $l = 1$ by Lemma \ref{big swap down}.
        \item Otherwise, $r_1$ must be above all the multiples of support $[k_1, k_1]$ or $[k_1 + 1, k_1 + 1],$ in which case the multi-segment of rows from $h$ to $r_1$ is of type $Y_{\mathcal{M}}$ for $\mathcal{M}$ ending at $A,$ and the existence of the dualization operator follows immediately.
    \end{enumerate}

    If we want to dualize $h$ to $r_2,$ we again suppose that all rows in between are completely unmerged and all intermediate hats are dualized away. In this case, dualizing to $r_2$ is equivalent up to dualizing to the chain that $r_2$ belongs to, and we have already proved that this is possible.
\end{proof}

We show the existence of the $U$ operator.

\begin{lemma}
\label{U exstence Z}
    Given $\mathcal{E}$ of type $Z_{k_1, k_2},$ the operator $U$ is well-defined and affects the $\mathcal{S}$-data the same as in the $Y_\mathcal{M}$ case.
\end{lemma}

\begin{proof}
    Let $h = ([A, -B], B, \eta)$ for some $\eta$ and that we wish to perform $U_{h, c},$ the $ui^{-1}$ operator separating $c$ circles from the hat $h$. We use the following relation
    $$D_{h', r} \circ U_{h, c} =  S_{r', c}  \circ D_{h, r},$$
    where $r$ is the chain obtained from the same $\mathcal{S}_i$ as $h.$ Here, $r'$ is the image of $r$ under $D$ and $h'$ is the remaining hat after $U.$ This relation is stated as Part (9) of Theorem \ref{Commutativity of Operators} for the $Y_\mathcal{M}$ case, and it follows immediately from the definitions of the operators in terms of the $\mathcal{S}$-data. The existence of a $U$ operation that impacts the $\mathcal{S}$-data in the same way as the $U$-operation for type $Y_\mathcal{M}$ follows since the other operators and their inverses are already known to exist. Any $ui^{-1}$ applied to a hat will produce an extended multi-segment of the same support and order as $U$ and therefore must produce the same multi-segment by Corollary \ref{cor same support}.
\end{proof}

Now we aim to prove the following classification for extended multi-segments of type $Z_\mathcal{M}.$

\begin{thm}
\label{Z Classification}
    Suppose $\mathcal{E}$ is of type $Z_\mathcal{M}$. Then $\mathcal{E}' \sim \mathcal{E}$ if and only if $\mathcal{E}'$ is of type $Z_\mathcal{M}$ and $\eta(\mathcal{E}) = \eta(\mathcal{E}').$
\end{thm}

In order to do this, we will use the following description of $\mathcal{E}^{\min}$ for extended multi-segments $\mathcal{E}$ of type $Z_{\mathcal{M}}.$ The proof of this lemma is omitted as it is exactly the same as the proof of \cite[Lemma 5.46]{HKT26}.

\begin{lemma}
\label{E min classification Z}
    Suppose $\EE$ is of type $Z_\mathcal{M}$ with $\mathcal{S}$-data $(\{0, \overline{1}, \overline{2}, \dots, \overline{n}\}).$ Then $\EE = \EE^{min}.$
\end{lemma}

Any $\mathcal{E}$ of type $Z_{k_1, k_2}$ can clearly be turned into one of the form specified in Lemma \ref{E min classification Z} with the same $\mathcal{M}$ and $\eta$ through a series of $-S, -D$ and $-M$ operators. This suffices to show that two extended multi-segments of type $Z_\mathcal{M}$ with the same $\eta$ are equivalent. Now, to prove Theorem \ref{Z Classification}, it suffices to show that if $\mathcal{E}$ is of type $Z_\mathcal{M}$ then any equivalent extended multi-segment is of type $Z_\mathcal{M}$ with the same $\eta.$ Since we know that $\mathcal{E}^{\min}$ has the same $\mathcal{E}$ and $\eta$, it suffices to show that these properties are preserved under raising operators.

\begin{lemma}
    Let $\mathcal{E}$ be of type $Z_\mathcal{M}$ and $T$ be a raising operator. Then $T$ preserves $Z_\mathcal{M}$ form, including $\mathcal{M}$, as well as $\eta(\mathcal{E})$.
\end{lemma}

\begin{proof}
    It is apparent that the only raising operators of type 3' applicable to multi-segments of form $Z_{k_1, k_2}$ are $S, M, U,$ and $D,$ and these are known to conserve $\mathcal{M}$ and $\eta.$ Otherwise, $T$ must be a $dual \circ ui \circ dual$ not of type 3'. In this case, $T$ consists of some combination of the raising operators of type 3' and their inverses; the proof of this fact is exactly the same as the proofs of \cite[Lemmas 5.47 and 5.48]{HKT26}.
\end{proof}

Next, we need to ensure that two unequal but equivalent $Z_\mathcal{M}$ multi-segments are associated with different Arthur packets. The proof of the following result is exactly the same as the proof of \cite[Lemma 6.1]{HKT26}.

\begin{lemma}
\label{packets distinct Z}
    If $\EE_1$ and $\EE_2$ are both of type $Z_\mathcal{M}$ and $\eta(\EE_1) = \eta(\EE_2),$ then $\psi(\EE_1) = \psi(\EE_2)$ if and only if $\EE_1 = \EE_2.$
\end{lemma}

Let $\mathcal{E}$ be an extended multi-segment of type $Z_{k, k + 1}$ and let $\Theta_1'(\mathcal{E}) = r \cup \mathcal{E},$ where $r = ([k + 2, -k - 2], k + 2, -\eta(\mathcal{E})).$ We remark that $\Theta_1'(\EE)$ only differs from the definition of $\Theta_1(\EE)$ by twisting by the appropriate supercuspidal representations. Note that $\Theta_1'(\mathcal{E})$ is of type $Z_{k, k + 2}.$ Then we have the following result.

\begin{lemma}
    We can perform exactly two $dual \circ ui \circ dual$s (up to row exchange) involving the first row of $\Theta_1'(\mathcal{E})$: one with the first row whose support ends at $k + 1,$ and one with the last row whose support ends at $k + 1.$ The resulting multi-segments are not equivalent, and we call them $\Theta_2'(\mathcal{E})$ and $\Theta_4'(\mathcal{E}).$ We can perform an additional $ui^{-1}$ to both $\Theta_2'$ and $\Theta_4'$ separating one circle from the row with support ending at $k + 2$ to obtain the same $\Theta_3'(\EE).$
\end{lemma}

\begin{proof}
    It is immediately obvious that there are only two $dual \circ ui \circ dual$s involving the first row of $\mathcal{E}$, up to row exchange. This is because all but one of the rows ending at $k + 1$ must be repeats. The first row ending at $k + 1$ cannot have support $[k + 1, k + 1]$ because we cannot have some $\mathcal{T}_i^0 = \{k + 1\}.$ Therefore, we can be sure that $\psi_{\Theta_2'(\mathcal{E})} \neq \psi_{\Theta_4'(\mathcal{E})}$ by comparing supports. Lastly, we can be sure that applying an additional $ui^{-1}$ to separate one circle with support $[k + 2, k + 2]$ produces the same $\Theta_4'(\mathcal{E})$ by \cite[Lemma 6.6]{HKT26}, which still holds for type $Z_{\mathcal{M}}.$
\end{proof}

\begin{lemma}
\label{thetas distinct case 3}
    All of the elements of the set $\{\psi_{\Theta_i'(\EE')} \mid \mathcal{E}' \sim \mathcal{E}; i = 1, 2, 3, 4 \}$ are distinct.
\end{lemma}

\begin{proof}
    The proof is analogous to that of \cite[Lemma 6.7]{HKT26}. In light of Lemma \ref{packets distinct Z}, it suffices to show that these extended multi-segments are not equal to each other. Suppose $\psi_{\Theta_i'(\mathcal{E}_1)} = \psi_{\Theta_i'(\Theta_j'(\mathcal{E}_2))}.$ First, we show that $i = j$ by comparing the supports of these multi-segments. Specifically, let $r \in \Theta_i(\mathcal{E}_1)$ be the row whose support ends at $k + 2.$
    \begin{enumerate}
        \item If $supp(r) = [k + 2, -k - 2],$ then $i = 1.$
        \item If $supp(r) = [k + 2, k + 2],$ then $i = 4.$
        \item If $supp(r) = [k + 2, k + 1],$ then $i = 3.$
        \item Otherwise, $i = 2.$
    \end{enumerate}

    Therefore, it is clear that $\psi_{\Theta_i'(\mathcal{E}_1)} = \psi_{\Theta_i'(\Theta_j'(\mathcal{E}_2)} \implies i = j.$ Now, we aim to show that $\mathcal{E}_1 = \mathcal{E}_2.$ If $i \in \{1, 3\}$ or $i = 3,$ then $\supp(\mathcal{E}) = \supp(\Theta_i'(\mathcal{E}) \setminus \{r\}),$ so $\mathcal{E}$ can be uniquely determined from $\Theta_i'(\mathcal{E}).$ If $i \in \{2, 4\},$ then $\Theta_i'(\mathcal{E})$ uniquely determines $\Theta_3'(\mathcal{E}),$ which determines $\mathcal{E}.$
\end{proof}

\begin{lemma}
\label{times 4 relation}
    We have the relation $$\Psi(\pi(\Theta_1'(\mathcal{E}))) = \{\psi_{\Theta_i'(\EE')} \mid \mathcal{E}' \sim \mathcal{E}; i = 1, 2, 3, 4 \}.$$
    Moreover, Lemma \ref{thetas distinct case 3} then implies that $|\Psi(\pi(\Theta_1'(\mathcal{E})))| = 4 |\Psi(\pi(\mathcal{E}))|.$
\end{lemma}

\begin{proof}
    Suppose that  {$\psi_{\mathcal{E}_1} \in \Psi(\pi(\Theta_1(\mathcal{E})))$.}Since $\Theta_1(\mathcal{E})$ is of type $Z_{k, k + 2},$ then $\mathcal{E}_1$ must also be of type $Z_{k, k + 2}$ with the same $\mathcal{M}$ and $\eta.$ We examine the $\mathcal{S}$-data of $\mathcal{E}_1$. In particular, we study the unique set containing $k + 2.$ Let $\mathcal{E}_2$ be the virtual extended multi-segment $\mathcal{E}(\mathcal{M}, \mathcal{S}', \mathcal{T}', \eta(\mathcal{E})$, where $(\mathcal{S}', \mathcal{T}')$ is obtained by removing $k + 2$ from the $\mathcal{S}$ data of $\mathcal{E}_1$. We have the following cases.

    \begin{enumerate}
        \item $\mathcal{S} = (\dots, \{\dots, k + 1\}, \{k + 2\})$ in which case $\mathcal{E}_1 = \Theta_4'(\mathcal{E}_2).$
        \item $\mathcal{S} = (\dots, \{ \dots \overline{\dots, k + 1, k + 2}\}),$ in which case $\mathcal{E}_1= \Theta_2'(\mathcal{E}_2).$
        \item $\mathcal{S} = (\dots, \{ \dots k, k + 1, k + 2\}),$ in which case $\mathcal{E}_1= \Theta_2'(\mathcal{E}_2).$
        \item $\mathcal{S} = (\dots, \{k + 1, k + 2\}),$ in which case $\mathcal{E}_1 = \Theta_4'(\mathcal{E}_2).$
        \item $\mathcal{S} = (\dots, \{\dots, k + 1, \overline{k + 2}\}),$ in which case $\mathcal{E}_1 = \Theta_3'(\mathcal{E}_2).$ \qedhere
    \end{enumerate}
\end{proof}

Now, we prove Case 3 of Theorem \ref{thm-count-theta-temp}, starting with Case 3.1. Let $\mathcal{E}$ be an extended multi-segment with $m_H \equiv 0 \bmod{2}$, $m_N \equiv 1 \bmod{2}$ such that $\EE$ ends at $H_{col}+1$ and let $\EE' = \rc_{H_{col}+1}(\EE).$ Then $\mathcal{E}_{m^{\up, \alpha}}^\up$ is of type $Z_{k, k + 1},$ while $(\mathcal{E}_{m^{\up, \alpha}}^\up)' = \mathcal{E}_{m^{\up, \alpha}}^\up \sm \{r_2, \dots, r_{m_{k + 1}}\},$ where $\{r_2, \dots, r_{m_{k + 1}} \}$ are the rows with support $[k + 1, k + 1]$. The statement of Case 3.1 is equivalent to the assertion that 
$$|\Psi(\pi(\mathcal{E}_{m^{\up, \alpha}}^\up))| = |\Psi(\pi((\mathcal{E}_{m^{\up, \alpha}}^\up)'))|.$$
This, however, follows from the fact that all the operators on $\pi(\mathcal{E}_{m^{\up, \alpha}}^\up))$ are some composition of $S, M, U, D,$ and their inverses. None of these operators affect the rows $\{r_2, \dots, r_{m_{k + 1}}\},$ and they are all still well defined even if these rows are removed. Thus, we have a one-to-one correspondence
$$\{\mathcal{E}_1 \mid \mathcal{E}_1 \sim \mathcal{E}_{m^{\up, \alpha}}\} = \{\mathcal{E}_1' \cup \{r_1, \dots, r_k\} \mid \mathcal{E}_1' \sim \mathcal{E}_{m^{\up, \alpha}}'\}.$$
We thus have a similar one-to-one correspondence between the sets $\Psi(\pi(\mathcal{E}_{m^{\up, \alpha}}))$ and $\Psi(\pi(\mathcal{E}'_{m^{\up, \alpha}})),$ which suffices for the proof.

Next, we prove Case 3.2. Let $\mathcal{E}$ be a multi-segment with $m_H \equiv 0 \bmod{2}$, $m_N \equiv 1 \bmod{2}$ such that $\EE$ ends at $H_{col}+2$, and let $\EE' = \rc_{H_{col}+2}(\EE).$ An argument equivalent to the proof of \cite[Lemma 6.8]{HKT26} allows us to reduce to the case where $m_{k + 2} = 1.$ Again, $\mathcal{E}'_\alpha$ is an extended multi-segment of type $Z_{k, k+ 1}$, and our reduction means that  $\mathcal{E}_\alpha$ is an extended multi-segment of type $Z_{k, k + 2}$ obtained by taking the union $\mathcal{E}'_\alpha \cup r.$ Then it follows from Lemma \ref{times 4 relation} that
$$|\Psi(\pi(\mathcal{E}_\alpha))| = 4 |\Psi(\pi(\mathcal{E}'_\alpha))|.$$

The proof of Case 3.3 is almost exactly the same as the proof of the recursive formula in Theorem \ref{thm-count-block-temp}, so we omit it here.

\begin{rmk}
    The more general statement of Case 3 is:
    \[|\Psi(\theta^\up_{-m^{\up, \alpha}}(\pi(\EE)))| = |\Psi(\theta^\up_{-m^{\up, \alpha}}(\pi(\BB_1 \cup \BB_2)))| \cdot \prod_{i=3}^k |\Psi(\pi(\BB_i))|,\]
    which essentially states that taking the theta correspondence only impacts the first and second block, while multi-segments on the remaining blocks act completely independently of these blocks a la the statement of Proposition \ref{prop-independence-of-blocks}. The details of the proof of this proposition still apply if a pair of blocks of Type $Y_{\mathcal{M}}$ is replaced by a single block of type $Z_{\mathcal{M}}$, which is exactly what happens under the theta correspondence.
\end{rmk}

\subsection{Case 5 of Theorem \ref{thm-count-theta-temp}}\label{sec case 5}

In this section, we prove the formula in Case 5 of Theorem \ref{thm-count-theta-temp}. The approach resembles that in \cite[\S 7]{HKT26}, since in both cases we aim to show that we can treat parts of a certain decomposition separately. Our proof has two parts. First, we prove an independence-type result (Lemma \ref{lem-case5-independence}), which says that operations must involve only one part of a certain decomposition structure. This is analogous to \cite[Lemma 7.12]{HKT26} (stated as Proposition \ref{lem-independence-of-blocks} below). Second, we prove an existence-type result (Lemma \ref{lem-existence-theta-lift}), which says that operations which are valid on one part of the decomposition viewed as its own extended multi-segment can still be applied when it is included in the larger whole. This is analogous to \cite[Lemma 7.6]{HKT26} and \cite[Lemma 7.7]{HKT26}.

Since the approach is broadly similar, let us recall several notions and lemmas from \cite[\S 7]{HKT26}. We first recall the definitions of types of boundaries.

\begin{defn}[{\cite[Definition 7.11]{HKT26}}]
    Suppose $\EE\in\VRep_\rho^\mathbb{Z}(G_n)$ is  equivalent to a tempered virtual extended multi-segment $\EE_{temp}$, and suppose $\EE_{temp}$ has a block decomposition $\BB_{1, temp} \cup \cdots \cup \BB_{k, temp}$. Suppose that $\EE$ has a decomposition $\BB_1 \cup \cdots \cup \BB_k$ with $\BB_i$ equivalent to $\BB_{i, temp}$. In this setup we give a classification of the ways in which $\BB_k$ could start, which depend on how its support overlaps with the support of $\BB_{k-1}$, which we call the \emph{boundary} between $\BB_k$ and $\BB_{k-1}$. Let $H_{col}$ be the last nonempty column in $\BB_{k-1}$ and let $N_{col}$ be the first nonempty column in $\BB_k$.
    \begin{itemize}
        \item A \emph{type 1 boundary} occurs when $N_{col} > H_{col}+1$.
        \item A \emph{type 2 boundary} occurs when $N_{col} = H_{col}$.
        \item A \emph{type 3 boundary} occurs when $N_{col} = H_{col}+1$.
    \end{itemize}
\end{defn}

Second, we recall the following fact about truncations (from the top) of extended multi-segments of type $Y_\mathcal{M}$. We will use this to prove a generalization \ref{lem-truncations-theta-lifts} regarding truncations of a theta lift of an almost-block, which will be used to understand operations of the form $dual \circ ui\inv \circ dual$ of type 3'.

\begin{lemma}[{\cite[Lemma 7.10]{HKT26}}]
\label{lem-truncations}
    Let $\EE$ be of type $Y_\mathcal{M}$, and let $\EE^{\tc}$ be a truncated virtual extended multi-segment containing all but the first $i$ rows of $\EE$. Then regardless of $i$, the extended multi-segment $\EE^{\tc}$ is equivalent to a tempered virtual extended multi-segment where the multiplicity of each column is at most the corresponding multiplicity in $\EE_{temp}$. Moreover, the operations realizing this equivalence can be performed on $\EE$.
\end{lemma}

Finally, we recall the following independence results, which we will use to prove our analogous independence results.

\begin{lemma}[{\cite[Lemma 7.12]{HKT26}}]
\label{lem-independence-of-blocks} 
    Suppose that $\EE\in\VRep_\rho^\mathbb{Z}(G_n)$ is equivalent to a tempered virtual extended multi-segment $\EE_{temp}$. Let $\EE_{temp}$ have block decomposition $\BB_{1, temp} \cup \dots \cup \BB_{k, temp}$. Suppose that $\EE$ has a decomposition $\BB_1 \cup \dots \cup \BB_k$ with $\BB_i$ equivalent to $\BB_{i, temp}$. Then any raising operator $T$ (or inverse) on $\EE$ cannot involve (even by row exchanges) rows from both $\BB_{i_1}$ and $\BB_{i_2}$ for $i_1 < i_2$.
\end{lemma}

\begin{prop}[{\cite[Proposition 7.13]{HKT26}}]
\label{prop-independence-of-blocks} 
    Let $\EE\in\VRep_\rho^\mathbb{Z}(G_n)$ be  equivalent to a tempered virtual extended multi-segment $\EE_{temp}$. If $\EE_{temp}$ has a block decomposition $\BB_{1, temp} \cup \cdots \cup \BB_{k, temp}$, then there exists a decomposition $\EE = \BB_1 \cup \cdots \cup \BB_k$ such that $\BB_i$ is equivalent to $\BB_{i, temp}$ for all $i$.
\end{prop}

We now turn to the proof of Case 5 of Theorem \ref{thm-count-theta-temp}. In all the statements that follow, we assume that $\EE_{temp}$ lies in Case 5 of Theorem \ref{thm-count-theta-temp}.

\subsubsection{Independence-type result}

Shortly we will consider decompositions of the following form. If $\EE_{temp}$ is a tempered extended multi-segment with block decomposition $\BB_{1, temp} \cup \dots \cup \BB_{k, temp}$, and $\EE$ is equivalent to $\Theta_1(\EE_{temp})$, then we want to consider a decomposition $\CC_1 \cup \dots \cup \CC_{k-1}$ such that
\begin{align*}
    \CC_1 &\sim \Theta_1(\BB_{1, temp} \cup \BB_{2, temp}), \\
    \CC_i &\sim \BB_{i+1, temp} \text{ for } i = 2, \dots, k-1.
\end{align*}

Like the decomposition into parts equivalent to blocks, this decomposition has the staircase property (see \cite[Definition 7.5]{HKT26}).

\begin{lemma}
\label{lem-staircase-theta-lift}
    The decomposition $\CC_1 \cup \dots \cup \CC_{k-1}$ described above, if it exists, has the property that for $i_1 < i_2$, and rows $r_1 \in \CC_{i_1}$, $r_2 \in \CC_{i_2}$, we have $A(r_1) \leq B(r_2)$. 
\end{lemma}
\begin{proof}
    The proof is much the same as \cite[Lemma 7.4]{HKT26}, except this time we know $\CC_i$ for $i>1$ is type $Y_\mathcal{M}$ and does not begin at zero, so the minimum of $B(r)$ for $r \in \BB_{i+1, temp}$ is the same as the minimum of $B(r)$ for $r \in \CC_i$.
    
    More precisely, this follows from the facts that
    \begin{itemize}
        \item the analogous statement is true for $\Theta_1(\BB_{1, temp}, \BB_{2, temp}), \BB_{3, temp}, \dots, \BB_{k, temp}$,
        \item equivalent extended multi-segments have the same max of support, and
        \item the minimum of the support of $\CC_{i+1}$ is the same as the minimum of the support of $\BB_i$. \qedhere
    \end{itemize}
\end{proof}

We also need the following fact, which follows from the results on lifts of almost-blocks.
\begin{lemma}
\label{lem-aj-almost-block}
    Let $\BB$ be an almost-block ending at column $k$ with the multiplicity of the last column even, and let $\CC_1$ be any virtual extended multi-segment equivalent to $\Theta_1(\BB)$. Pick any row $r\in\CC_1$ for which we can row exchange $r$ until it becomes the last row. We denote its image under these row exchanges by $r'$. Then either
    \begin{itemize}
        \item $r$ consists of a single circle in column $k$, with the same sign as the last sign of $\BB$ or
        \item $l(r') = 1$, and the sign of the last circle of $r'$ is the opposite of the sign of the last circle in $\BB$.
    \end{itemize}
\end{lemma}
\begin{proof}
    Let $\BB$ have block decomposition $\BB_{1, temp} \cup \BB_{2, temp}$. Note that $\BB_{2, temp}$ is just a single circle. First recall from \cite[Lemma 6.8]{HKT26} that every extended multi-segment equivalent to $\BB_{1, temp} \cup \BB_{2, temp}$ takes the form $\BB_1 \cup \BB_{2, temp}$ for some $\BB_1 \sim \BB_{1, temp}$. By Lemma \ref{duds for almost block}, the only operations on $\Theta_1(\BB_1\cup \BB_{2, temp})$ which involve the hat are $dual \circ ui \circ dual$ operations between the hat and the first and last rows ending at column $k$. So we split into cases depending on whether $\CC_1$ is of the form $\Theta_1(\BB)$, $\Theta_2(\BB)$, or $\Theta_4(\BB)$.
    
    First consider the virtual extended multi-segments of the form $\Theta_1(\BB_1 \cup \BB_{2, temp})$ or $\Theta_2(\BB_1\cup \BB_{2, temp})$. Observe that both of these have a decomposition of the form $\BB_1' \cup \BB_{2, temp}$ where $\BB_1' \sim \Theta_1(\BB_{1, temp})$. If the row $r$ lies in $\BB_{2, temp}$, then it is already the last row. Its sign must be the same as the sign of the last circle in $\BB_{1, temp}$, and it lies in column $k$. So this satisfies the conditions of the first case. 
    
    Otherwise, since $\Theta_1(\BB_{1, temp})$ is a block, when the row $r$ is exchanged to the last position in $\BB_1'$, giving $r'$, we have $l(r') = 0$ by \cite[Lemma 7.9]{HKT26}. Furthermore, the last circle of $r'$ has the same sign as the last circle of $\Theta_1(\BB_{1, temp})$. Since the first circle of $\Theta_1(\BB_{1, temp})$ has the opposite sign as the first circle of $\BB_{1, temp}$, but $\Theta_1(\BB_{1, temp})$ ends in column $k+1$ while $\BB_{1, temp}$ ends in column $k$, the sign of the last circle is the same. So we conclude that the last circle of $r'$ has the same sign as the last circle of $\BB_{1, temp}$, which is the same as the sign of $\BB_{2, temp}$. So the row exchange between $r'$ and the (only) row of $\BB_{2, temp}$ lies in Case 1(b). If $r''$ is the result, then $l(r'') = 1$ and $\eta(r'') = - \eta(r)$. So the sign of the last circle of $r''$ is opposite the sign of the last circle of $r'$, which is the same as the sign of the last circle of $\BB$.

    Now we consider extended multi-segments of the form $\Theta_4(\BB_1 \cup \BB_{2, temp})$. In this case since the last row of $\BB_1 \cup \BB_{2, temp}$ ending at $k$ is just $\BB_{2, temp}$, the $dual \circ ui \circ dual$ occurs between the hat and $\BB_{2, temp}$. By reducing to the case where all rows have one circle by Lemma \ref{cor Alex} and then performing a similar calculation as in the proof of the existence of the $D$ operation, as stated in Lemma \ref{lem SMUD operators} (detailed in \cite[Lemma 5.22]{HKT26}) we see that the resulting row is a pair of triangles with the same sign as the last circle of $\BB$. So the sign of the last circle of this row (recall this is defined even if the row has no circles in the picture) is the opposite of the sign of the last circle of $\BB$.
\end{proof}

In preparation for our main independence results, we generalize Lemma \ref{lem-truncations} to the case of a theta lift of an almost-block.

\begin{lemma}
\label{lem-truncations-theta-lifts}
    Let $\EE$ be equivalent to the theta lift of an almost-block. Suppose the almost-block has block decomposition $\BB_{1, temp} \cup \BB_{2, temp}$, with $\BB_{1, temp}$ ending at column $H_{col}$. Let $\EE^{\tc}$ be a truncated extended multi-segment containing all but the first $i$ rows of $\EE$. Then $\EE^{\tc}$ is equivalent to one of the following. Moreover, in each case the operations realizing this equivalence can be performed on $\EE$.
    \begin{enumerate}
        \item If $\EE = \Theta_1(\BB_1 \cup \BB_{2, temp})$ for some $\BB_1 \sim \BB_{1, temp}$, then
        \begin{itemize}
            \item $\EE^{\tc}$ is equivalent to $\Theta_1(\BB_{1, temp} \cup \BB_{2, temp})$,
            \item or is equivalent to $\BB_1' \cup \BB_{2, temp}$ where $\BB_1'$ is tempered, and the multiplicity of each column is at most the corresponding multiplicity in $\BB_{1, temp}$.
        \end{itemize}
        \item If $\EE = \Theta_2(\BB_1 \cup \BB_{2, temp})$ for some $\BB_1 \sim \BB_{1, temp}$, then
        \begin{itemize}
            \item $\EE^{\tc}$ is equivalent to the second form in the previous case,
            \item or is equivalent to $\BB_1' \cup \BB_{2, temp}$ where $\BB_1'$ is the result after applying a certain $ui$ to a tempered extended multi-segment whose multiplicities are at most the multiplicities in $\Theta_3(\BB_{1, temp})$.
        \end{itemize}
        \item If $\EE = \Theta_4(\BB_1 \cup \BB_{2, temp})$ for some $\BB_1 \sim \BB_{1, temp}$, then $\EE^{\tc}$ is equivalent to $\BB_1' \cup \BB_2'$, where $\BB_1'$ is tempered with the multiplicities of each column at most the multiplicities in $\BB_{1, temp}$, and $\BB_2'$ is the same as the last row of $\EE$.
    \end{enumerate}
\end{lemma}
\begin{proof}
    The proof is similar to the proof of Lemma \ref{lem-truncations} (see \cite[Lemma 7.10]{HKT26}). 
    
    Assume first that was are in Case (1). if $\EE^{\tc} \neq \EE$, then the hat added by $\Theta_1$ must be truncated. Hence a nontrivial truncation of $\EE$ is just a (possibly trivial) truncation of a virtual extended multi-segment equivalent to an almost-block. In this case, the result follows from Lemma \ref{lem-truncations}.

    Assume now that was are in Case (2). Further, we first suppose that the row to which the $dual \circ ui \circ dual$ is applied is truncated. Then whether or not the $dual \circ ui \circ dual$ was applied is irrelevant, so this reduces to the second form of Case (1). Now suppose that the row was not truncated. Recall that $\EE$ takes the form $\EE' \cup \BB_{2, temp}$ where $\EE' \sim \Theta_1(\BB_{1, temp})$, since $\BB_{2, temp}$ cannot be the last row ending at $H_{col}$. By using the operators of Lemma \ref{lem SMUD operators}, we can undualize hats and split rows of circles in $\EE'$, with one exception. Since $\BB_{2, temp}$ is a single circle in $H_{col}$, we cannot perform a $ui\inv$ operation to split off circles in $H_{col}+1$. Hence $\EE$ is equivalent to $\BB_1' \cup \BB_{2, temp}$, where $\BB_1'$ is ``almost tempered,'' except that there cannot be a single circle in $H_{col}+1$. In fact, observe that since the $dual \circ ui \circ dual$ is applied to the first row ending at $H_{col}$, and this row was not truncated, in $\BB_1'$ we must have at least one circle in $H_{col}$, since the dual of this row (if it is a hat, or the row itself if it is a row of circles) has a circle in $H_{col}$. So in $\BB_1'$ there exists a row with support $[H_{col}+1, H_{col}]$ consisting of two circles, which is the result of applying a certain $ui$.

    Finally, we assume that we are in Case (3). We note that the last row of $\BB_1 \cup \BB_{2, temp}$ with support ending in $H_{col}$ is just the last row $\BB_{2, temp}$, so the truncation does not affect this row. Hence a truncation of $\EE$ is just a truncation of $\BB_1$ followed by the last row of $\EE$. Since this last row has support $[H_{col}+1, H_{col}]$, all undualizations and all $ui\inv$ operations of type 3' on $\BB_1$ are valid. Thus by Lemma \ref{lem-truncations}, the truncation of $\BB_1$ is equivalent to an extended multi-segment taking the desired form $\BB_1'$.
\end{proof}

We now state and prove our main independence lemma, which will allow us (in Proposition \ref{prop-case5-decomposition}) to create the decomposition promised at the beginning of the subsection.

\begin{lemma}
\label{lem-case5-independence}
    Let $\EE_{temp}$ be a tempered extended multi-segment with block decomposition $\BB_{1, temp} \cup \dots \cup \BB_{k, temp}$. Let $\EE$ be any extended multi-segment equivalent to $\Theta_1(\EE_{temp})$, and suppose $\EE$ has a decomposition $\CC_1 \cup \dots \cup \CC_{k-1}$ such that $\CC_1 \sim \Theta_1(\BB_{1, temp} \cup \BB_{2, temp})$ and $\CC_i \sim \BB_{i+1, temp}$ for $i = 2, \dots, k-1$. Then any operation on $\EE$ cannot involve (even by row exchanges) a row from $\CC_{i_1}$ and a row from $\CC_{i_2}$ for $i_1 < i_2$.
\end{lemma}

\begin{proof}
    The proof is similar in structure to the proof of \cite[Lemma 7.12]{HKT26} (see Lemma \ref{lem-independence-of-blocks}), although the specific cases are different.
    
    Without loss of generality suppose that each of the $\CC_i$ for $i>1$ are of type $Y_\mathcal{M}$, since any extended multi-segment equivalent to a block is of type $Y_\mathcal{M}$ up to row exchanges. For $\CC_1$, we can assume that it takes the form $\Theta_1(\BB_1 \cup \BB_{2, temp})$, $\Theta_2(\BB_1 \cup \BB_{2, temp})$, or $\Theta_4(\BB_1 \cup \BB_{2, temp})$ for some $\BB_1$ of type $Y_\mathcal{M}$ equivalent to $\BB_{1, temp}$.

    Throughout this proof, let $H_{col}$ be the last column of $\BB_1$. In the picture below we illustrate the decomposition into parts $\CC_i$ for a prototypical example of Case 5 of Theorem \ref{thm-count-theta-temp}. 
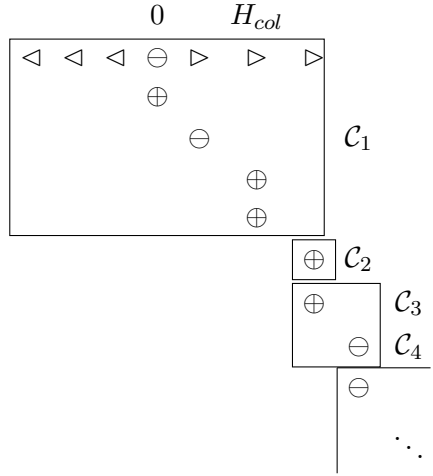
    \[\begin{tikzpicture}[baseline=(current bounding box.north)]
      \matrix (m) [matrix of math nodes, nodes in empty cells] {
        & & & 0 & & {H_{col}} & & & \\
        \lhd & \lhd & \lhd & \ominus & \rhd & \rhd & \rhd & & \\
        & & & \oplus & & & & & \\
        & & & & \ominus & & & \CC_1 & \\
        & & & & & \oplus & & & \\
        & & & & & \oplus & & & \\
        & & & & & & \oplus & \CC_2 & \\
        & & & & & & \oplus & & \CC_3 \\
        & & & & & & & \ominus & \CC_4 \\
        & & & & & & & \ominus & \phantom{\oplus} \\
        & & & & & & & \phantom{\oplus} & \ddots \\
      } ;
      \draw (m-7-7.north west) rectangle (m-7-7.south east);
      \draw (m-8-7.north west) rectangle (m-9-8.south east);
      \draw (m-10-8.north west) -- (m-10-9.north east);
      \draw (m-10-8.north west) -- (m-11-8.south west);
      \draw (m-2-1.north west) rectangle (m-6-7.south east);
    \end{tikzpicture}\]
    
    \underline{Operations of the form $ui\inv$ of type 3'}

    As before, it suffices to consider a $ui\inv$ operation consisting of a sequence of monotonic row exchanges, followed by a $ui\inv$, followed by the inverse sequence. If $i_1, i_2 \neq 1$ then we can view the $ui\inv$ operation as occurring on $\CC_2 \cup \dots \cup \CC_{k-1}$, which are all equivalent to blocks. So by Lemma \ref{lem-independence-of-blocks}, no such operation is possible. So suppose $i_1 = 1$.

    First suppose that the $ui\inv$ operation is applied to a row $r$ that was originally from $\CC_{i_1} = \CC_1$, before row exchanges. Note that each $\CC_i$ for $i>1$ starts at $H_{col}+1$ or later, while $r$ ends at $H_{col}+1$ or earlier. So the only rows $r$ can row exchange with that do not lie in $\CC_1$ are single circles in column $H_{col}+1$, and $r$ must end at $H_{col}+1$. Let $r'$ be the image of $r$ after the row exchanges. Since $r$ lies in $\CC_1$, after the row exchanges $r'$ comes after a row which is a single circle in column $H_{col}+1$. Since $r'$ ends at $H_{col}+1$, this means that if any number of circles are split off, this will result in a row ending before $H_{col}+1$. So the resulting extended multi-segment is not in admissible order, so no $ui\inv$ is possible.

    Second suppose that the $ui\inv$ operation is applied to a row $r$ that was originally from $\CC_{i_2}$. We claim that it is impossible for $r$ to be exchanged into $\CC_1$ if $i_2 \geq 4$. If the boundary between $\CC_3$ and $\CC_4$ is type 1 or type 3\footnote{Note that $\CC_i$ for $i>1$ are equivalent to blocks, so we can speak of the boundary in the same way as before.}, then $\CC_4$ starts at $H_{col}+3$ or later. If the boundary is type 2, then since the boundary between $\CC_2$ and $\CC_3$ is also type 2, and blocks are maximal, it must be that $\CC_3$ has at least two columns. So $\CC_4$ starts at $H_{col}+2$ or later. In either case, since $\CC_1$ ends at $H_{col}+1$, no row exchanges are possible with rows of $\CC_1$.

    Now suppose $i_2 = 3$. Since the supports of $\CC_3$ and $\CC_1$ only overlap in column $H_{col}+1$, if $r$ is exchanged with any rows of $\CC_1$, they must have support only in $H_{col}+1$, so they must be single circles. Moreover, $r$ must have support starting at $H_{col}+1$. But by the same reasoning as before (in the case that $r$ lies in $\CC_1$ and is exchanged down), the image after row exchanges $r'$ has no valid $ui\inv$ operations which leave the order admissible, since it comes before the single circles in column $H_{col}+1$.

    Finally, if $i_2 = 2$ then $r$ has support of length $1$, so certainly no $ui\inv$ is possible. 
    
    \underline{Operations of the form $dual \circ ui\inv \circ dual$ of type 3'}

    There are two cases for a possible $dual \circ ui\inv \circ dual$ operation. Either a row $r$ from $\EE$ corresponding to $\widehat{r}$ in $dual(\EE)$ is exchanged up in $dual(\EE)$, then a $ui\inv$ operation is applied, or $\widehat{r}$ is exchanged down in $dual(\EE)$.

    In the first case, the proof is very similar to the proof in \cite[Lemma 7.12]{HKT26} (Lemma \ref{lem-independence-of-blocks}). It suffices to show that $\EE$ is in (P') order as it is shown in \cite[p. 52, 3rd paragraph]{HKT26} that it is not possible to apply a $dual\circ ui^{-1} \circ dual$ of type 3' involving a row that is being swapped up in the dual in our setting. The decomposition $\CC_1 \cup \dots \cup \CC_{k-1}$ satisfies the staircase property, so it suffices to show that each $\CC_i$ is in (P') order. For $\CC_i$ for $i>1$, this follows from the fact that $\CC_i$ is of the type $Y_\mathcal{M}$. For $\CC_1$, if it takes the form $\Theta_1(\BB_1 \cup \BB_{2, temp})$ then since the hat is wider than all the rows, adding it preserves the (P') order. If $\CC_1$ takes the form $\Theta_2(\BB_1 \cup \BB_{2, temp})$ or $\Theta_4(\BB_1 \cup \BB_{2, temp})$ then we see that the order is still (P') since none of the left endpoints have changed compared to $\BB_1 \cup \BB_{2, temp}$.

    In the second case, we once again imitate the proof of \cite[Lemma 7.12]{HKT26} (Lemma \ref{lem-independence-of-blocks}), but we make use of Lemma \ref{lem-truncations-theta-lifts} instead of Lemma \ref{lem-truncations}, which makes it slightly more involved. 

    If $\widehat{r}$ is not exchanged with any rows of $dual(\CC_1)$, then the claim follows from Lemma \ref{lem-independence-of-blocks}. Indeed, we remove $\mathcal{C}_1$ from $\EE$ and apply Lemma \ref{lem-independence-of-blocks}. Note that the resulting parts agree with $dual(\EE)$ up to a global sign change.
    Thus, we suppose $\widehat{r}$ is exchanged with rows of $dual(\CC_1)$. We get a truncation of $\CC_1$ by considering the rows in $\CC_1$ corresponding to the rows in $dual(\CC_1)$ which are exchanged with $\widehat{r}$. By Lemma \ref{cor Alex}, we can assume that $\widehat{r}$ is exchanged with the dual of extended multi-segments of the form specified in Lemma \ref{lem-truncations-theta-lifts}.

    We observe that $r$ cannot lie in $\CC_2$. This is because $r$ is only supported in one column, so $B(\widehat{r}) = - A(\widehat{r})$, so any $ui\inv$ of type 3' would create we a row $s$ satisfying $|B(s)| > A(s)$, which is impossible.

    By examining the cases of Definition \ref{def row exchange}, we see that any row exchange with a row with with one circle cannot decrease $l(\widehat{r})$ by more than $1$. So in order for $l$ to be $0$ after all the row exchanges, it must decrease once for every column. Also, if $\widehat{r}$ is exchanged across the boundary of two parts of the decomposition equivalent to blocks, then the net result is that $l$ decreases less than once for every column. Since $r$ cannot lie in $\CC_2$, necessarily $\widehat{r}$ is exchanged across such a boundary. Hence it remains to show that when $\widehat{r}$ is exchanged with the rows of $dual(\CC_1)$ corresponding to the truncation of $\CC_1$, $l$ does not decrease more than the number of columns of the truncation, excluding the column $H_{col}+1$. (The reason for the caveat is that the supports of $\CC_1$ and $\CC_2$ have an overlap of one column, namely $H_{col}+1$.)

    First suppose that the truncation of $\CC_1$ is trivial, or in other words $\widehat{r}$ is exchanged all the way to the last row of $dual(\EE)$. Let $\widehat{r}'$ be the image of $\widehat{r}$ after all row exchanges except the last one (i.e. the one with the dual of the hat). We note that $l(\widehat{r}') \geq 2$, since columns $H_{col}$ and $H_{col}+1$ both have an even multiplicity in $\EE_{temp}$. So for the last row exchange, between $\widehat{r}'$ and the dual of the hat, if it is Case 1(b) or Case 1(c) then $l(\widehat{r}'') = l(\widehat{r}') \pm 1$, which in either case is still positive. If it is Case 1(a), then $l(\widehat{r}'') = C(\widehat{r}') + l(\widehat{r}') - 1$, but $l(\widehat{r}') - 1>0$, so this is still positive. So in any case no $ui\inv$ of type 3' is possible.

    Second suppose that the truncation is equivalent to something of the form $\BB_1' \cup \BB_{2, temp}$, where $\BB_1'$ is tempered and the multiplicities are at most the multiplicities of $\BB_{1, temp}$. In this case since the truncation is not supported in $H_{col}+1$, we just want to show that $l$ does not decrease by more than the number of columns of the truncation. This follows directly from the aforementioned fact that any row exchange does not decrease $l$ by more than $1$, and Lemma \ref{lem-multiplicity-cancel}.

    Third suppose that the truncation is equivalent to something of the form $\BB_1' \cup \BB_{2, temp}$, where $\BB_1'$ is the result after applying a $ui$ involving circles in $H_{col}$ and $H_{col}+1$ to a tempered extended multi-segment whose multiplicities are at most the multiplicities in $\Theta_3(\BB_{1, temp})$. This is depicted in the diagram below. By Lemma \ref{lem-multiplicity-cancel}, we can assume that there are no rows of $\BB_1'$ supported in column $H_{col}$ except for the row with support $[H_{col}+1, H_{col}]$, because the multiplicity of $H_{col}$ is even. By the same reasoning as before, for every column before than $H_{col}$, exchanging with the hats corresponding to the circles in that column does not decrease $l$ by more than $1$. So it suffices to show that exchanging with the duals of the row with support $[H_{col}+1, H_{col}]$ and $\BB_{2, temp}$ does not decrease $l$ by more than $1$. Call the first row $s_1$ and the second row $s_2$.
    
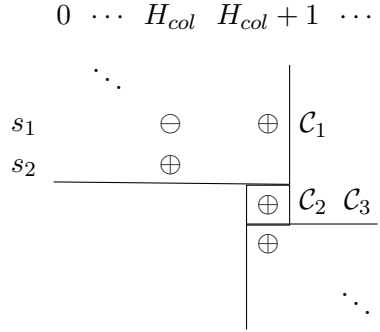
\[\begin{tikzpicture}[baseline=(current bounding box.north)]
      \matrix (m) [matrix of math nodes, nodes in empty cells] {
        & 0 & \cdots & H_{col} & H_{col} + 1 & \cdots\\
        & & \ddots &  & \phantom{\oplus} & \\
        s_1 & & & \ominus & \oplus & \\
        s_2 & & & \oplus & \phantom{\oplus} & \\
        & & & & \oplus & \\
        & & & & \oplus & \phantom{\oplus} \\
        & & & & \phantom{\oplus}& \ddots \\
      } ;
      \draw (m-4-2.south west) -- (m-4-5.south east);
      \draw (m-4-5.south east) -- (m-2-5.north east) node[pos=0.5, right] {$\CC_1$};
      \draw (m-6-5.north west) -- (m-6-6.north east) node[pos=0.85, above] {$\CC_3$} node[pos=0.51, above] {$\CC_2$};
      \draw (m-6-5.north west) -- (m-7-5.south west);
      \draw (m-5-5.north west) rectangle (m-5-5.south east);
    \end{tikzpicture}\]
    
    To show this, we first observe that by examining the cases of Definition \ref{def row exchange}, exchanging $\widehat{r}$ with $\widehat{s_2}$ cannot decrease $l$ by more than $1$, and exchanging it with $\widehat{s_1}$ cannot decrease $l$ by more than $2$. Hence the row exchange with $\widehat{s_2}$ must either keep $l$ unchanged, or decrease it by $1$. In the former case, the row exchange with $\widehat{s_1}$ must decrease $l$ by $2$, and in the latter case the row exchange with $\widehat{s_1}$ must decrease $l$ by $1$.

    In the former case, if $l$ is unchanged, then it must have been Case 1(a) of Definition \ref{def row exchange} and $C(\widehat{r}) = 1$. But note that $s_1$ and $s_2$ fail the alternating sign condition, since the last circle of both of them has the same sign as the last circle of $\BB_1$. So $\widehat{s_2}$ and $\widehat{s_1}$ also fail the alternating sign condition, and since $C(\widehat{s_2})=1$, this means $\eta(\widehat{s_2}) = \eta(\widehat{s_1})$. Since the row exchange of $\widehat{r}$ with $\widehat{s_1}$ is Case 1(a), they fail the alternating sign condition. Moreover, $\eta(\widehat{r}) = \eta(\widehat{r}')$, where $\widehat{r}'$ is the image of $\widehat{r}$ after the first row exchange. So the row exchange of $\widehat{r}$ with $\widehat{s_2}$ is either Case 1(a) or Case 1(b). But the only way for $l$ to decrease by $2$ is if it is Case 1(a) and $C(\widehat{r}') = 0$, which is impossible, as $\widehat{r}$ always has at least one circle.

    In the latter case, if $l$ decreases by $1$ in the first row exchange, then it must have been Case 1(c). By similar reasoning as before, since $\eta(\widehat{r}') = -\eta(\widehat{r})$, we have that $\widehat{r}'$ fails the alternating sign condition with $\widehat{s_2}$. So the row exchange must be Case 1(a) or Case 1(b). But in order for $l$ to decrease by $1$ in the second row exchange, it must be Case 1(a) and $C(\widehat{r}') = 1$. But this is clearly impossible since $l(\widehat{r}')$ is not maximal, as it just decreased by $1$.

    Fourth and finally suppose that the truncation is equivalent to something of the form $\BB_1' \cup \BB_2'$, where $\BB_1'$ is tempered and the multiplicities are at most the multiplicities in $\BB_{1, temp}$, and $\BB_2'$ is a pair of triangles, as shown below. Since $\BB_2'$ is the only row of the truncation supported in $H_{col}+1$, it suffices to show that row exchanges with the dual of this row do not decrease $l$. This is clear by examining the cases of Definition \ref{def row exchange}, since the row has no circles.
    \[\begin{tikzpicture}[baseline=(current bounding box.north)]
      \matrix (m) [matrix of math nodes, nodes in empty cells] {
        0 & \cdots & H_{col} & H_{col} + 1 & \cdots\\
        & \ddots & \phantom{\oplus} & \phantom{\oplus} & \\
        & & \oplus & & \CC_1 \\
        & & \lhd & \rhd & \\
        & & & & \ddots \\
      } ;
      \draw (m-4-1.south west) -- (m-4-4.south east);
      \draw (m-4-4.south east) -- (m-2-4.north east);
    \end{tikzpicture}\]
    
    \underline{Operations of the form $dual \circ ui \circ dual$}

   These are impossible because the decomposition has the staircase property (Lemma \ref{lem-staircase-theta-lift}). Indeed, it implies that for any $r_1 \in \CC_{i_1}$ and $r_2 \in \CC_{i_2}$, $\supp(\widehat{r_2}) \supset \supp(\widehat{r_1})$. So there is no nontrivial union-intersection possible between different blocks.

    \underline{Operations of the form $ui$}

    Observe that if $i_1 > 1$, then the $ui$ operation between $r_1$ and $r_2$ can be viewed as a $ui$ on the truncation $\CC_2 \cup \dots \cup \CC_{k-1}$. All of the parts of this decomposition are equivalent to blocks. So by Lemma \ref{lem-independence-of-blocks}, no $ui$ is possible. Hence we only have to consider the case where $i_1 = 1$.

    Now observe that for $i_2 > 4$, no $ui$ is possible for support reasons. By the same reasoning as when we considered $ui\inv$ operations, the earliest column that $\CC_4$ could start at is $H_{col}+2$, which is achieved when the boundaries between $\CC_2$ and $\CC_3$, and $\CC_3$ and $\CC_4$ are both type 2. If $\CC_5$ started at $H_{col}+2$, then $\CC_4$ would only be supported in one column, which means that its circles should be part of $\CC_3$ or $\CC_4$, by maximality of blocks. So $\CC_5$ must start at least at $H_{col}+3$, which means no $ui$ is possible with a row from $\CC_1$, which ends at $H_{col}+1$. So it remains to consider the cases $i_2 = 4$, $3$, $2$.

    First suppose $i_2 = 4$. Then by the above reasoning it must be that the boundaries between $\CC_2$ and $\CC_3$ and $\CC_3$ and $\CC_4$ are both type 2, and that $\CC_3$ has two columns. It also follows that $r_1$ must end at $H_{col}+1$ and $r_2$ must begin at $H_{col}+2$. Moreover, it must be that every row of $\CC_3$ has exactly one circle. Otherwise, row with support $[H_{col}+2, H_{col}+1]$ cannot be exchanged with $r_1$ or $r_2$, so they can never be exchanged to adjacent positions. By Lemma \ref{lem-multiplicity-cancel} and \ref{lem-multiplicity-cancel-up}, we can assume the multiplicities of $H_{col}+1$ and $H_{col}+2$ in $\CC_2$ and $\CC_3$ are all $1$. Let $r_1'$ be the image of $r_1$ after it is exchanged to the last row of $\CC_1$, and let $r_2'$ be the image of $r_2$ after it is exchanged to the first row of $\CC_4$. In summary, we have the following situation.
    
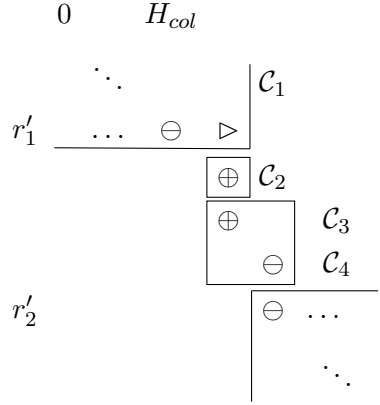
\[\begin{tikzpicture}[baseline=(current bounding box.north)]
      \matrix (m) [matrix of math nodes, nodes in empty cells] {
        & 0 & & H_{col} & & & \\
        & & \ddots & & \phantom{\oplus} & \CC_1 & \\
        r_1' & & \dots & \ominus & \rhd & & \\
        & & & & \oplus & \CC_2 & \\
        & & & & \oplus & & \CC_3 \\
        & & & & & \ominus & \CC_4 \\
        r_2' & & & & & \ominus & \dots \phantom{\oplus} \\
        & & & & & \phantom{\oplus} & \ddots \\
      } ;
      \draw (m-3-2.south west) -- (m-3-5.south east);
      \draw (m-2-5.north east) -- (m-3-5.south east);
      \draw (m-4-5.north west) rectangle (m-4-5.south east);
      \draw (m-5-5.north west) rectangle (m-6-6.south east);
      \draw (m-7-6.north west) -- (m-7-7.north east);
      \draw (m-7-6.north west) -- (m-8-6.south west);
    \end{tikzpicture}\]

    To see that no $ui$ is possible, we note that for support reasons the $ui$ must be performed by exchanging $r_1'$ with the circles in column $H_{col}+1$, exchanging $r_2'$ with the circles in column $H_{col}+2$, and then applying $ui$ to adjacent rows, and exchanging back. By Lemma \ref{lem-multiplicity-cancel}, the image of $r_1'$ after these row exchanges, which we call $r_1''$, is identical to $r_1'$. Let $r_2''$ be the image of $r_2'$ after the row exchange with the circle in column $H_{col}+2$. Since the supports of $r_1$ and $r_2$ do not overlap, the union-intersection cannot be Cases 1 or 2 of Definition \ref{def row exchange}. Also, since $r_1''$ has $l(r_1'') = 1$, \[A(r_1'') - l(r_1'') + 1 = H_{col}+1 < H_{col}+2 \leq B(r_2'') + l(r_2'').\] So no $ui$ of Case 3 is possible either.

    Second suppose $i_2 = 3$. We have the following three cases:
    \begin{enumerate}
        \item $A(r_1) = H_{col}+1$ and $B(r_2) = H_{col}+1$
        \item $A(r_1) = H_{col}+1$ and $B(r_2) = H_{col}+2$
        \item $A(r_1) = H_{col}$ and $B(r_2) = H_{col}+1$
    \end{enumerate}

    In both the first and the second case, since $A(r_1) = H_{col}+1$, by Lemma \ref{lem-aj-almost-block}, we have that if $r_1'$ is the image of $r_1$ after it is exchanged to the last row of $\CC_1$, then $l(r_1') = 1$ and the sign of the last circle of $r_1'$ is opposite the sign of the last circle of $\BB_{1, temp}$.

    In the first case, by Lemma \ref{lem-multiplicity-cancel}, we can assume without loss of generality that $\CC_2$ has multiplicity $1$ in $H_{col}+1$. First suppose that $r_2$ is supported only in $H_{col}+1$. Then the $ui$ could occur in one of two ways: either $r_1'$ is exchanged with the row of $\CC_2$, and then the $ui$ occurs, or $r_2$ is exchanged with the row of $\CC_2$, and then the $ui$ occurs. 
    
    For the first way, since the sign of the last circle of $r_1'$ is opposite the sign of the last circle of $\BB_{1, temp}$, and this is the same as the sign of the circle in $\CC_2$, the row exchange with the row of $\CC_2$ is Case 1(c) of Definition \ref{def row exchange}. If $r_1''$ is the image of $r_1'$ after the row exchange, then $l(r_1'') = 0$ and $\eta(r_1'') = -\eta(r_1')$. Since $\eta(r_2)$ is the same as the sign of the circle in $\CC_2$, it follows that $r_1''$ and $r_2$ fail the alternating sign condition. So the only possible union-intersection would be Case 1 or Case 2 of Definition \ref{def ui}, but $A(r_2) - l(r_2) \neq A(r_1'') - l(r_1'')$ and $B(r_2) + l(r_2) \neq B(r_1'') + l(r_1'')$.

    For the second way the row exchange between $r_2$ and the row in $\CC_2$ is trivial, giving $r_2'$ with $l(r_2') = l(r_2)$ and $\eta(r_2') = \eta(r_2)$. In this case a union-intersection between $r_1'$ and $r_2'$ of Case 3 of Definition \ref{def ui} is possible, but it leaves the rows (and the order) unchanged, so it is trivial.

    Next suppose that $r_2$ is supported in more than one column. We once again have to ways for the $ui$ to occur: either $r_1'$ is exchanged with the circle in $\CC_2$, or $r_2$ is exchanged with the circle in $\CC_2$. The first way is impossible for exactly the same reason as before. For the second way, the row exchange between $r_2$ and the circle in $\CC_2$ is Case 2(b) of Definition \ref{def row exchange}. So the image $r_2'$ after row exchange has $l(r_2') = 1$ and $\eta(r_2') = -\eta(r_2)$. So $r_1'$ and $r_2'$ fail the alternating sign condition, so the union-intersection would have to be Case 1 or Case 2. But $A(r_1') - l(r_1') \neq A(r_2') - l(r_2')$ and $B(r_1') + l(r_1') \neq B(r_2') + l(r_2')$, so union-intersection is impossible.

    For the second case, that $A(r_1) = H_{col}+1$ and $B(r_2) = H_{col}+2$, the only way the union-intersection can occur is if $r_1'$ is exchanged with all the circles in column $H_{col}+1$. Since the multiplicity of $H_{col}+1$ is even, by Lemma \ref{lem-multiplicity-cancel}, the result $r_1''$ is the same as $r_1'$. So $r_1''$ and $r_2$ fail the alternating sign condition, but again neither Case 1 nor Case 2 of Definition \ref{def ui} applies.

    For the third case, that $A(r_1) = H_{col}$ and $B(r_2) = H_{col}+1$, for support reasons the only way the union-intersection could occur is if $r_1$ is exchanged to the last row of $\CC_1$, giving $r_1'$, and $r_2$ is exchanged to the first row of $\CC_2$, giving $r_2'$. For support reasons the $ui$ would have to be Case 3 of Definition \ref{def ui}. Moreover, since $B(r_2') + l(r_2') \geq H_{col}+1$, and $A(r_1') - l(r_1') + 1 \leq H_{col}+1$, it follows that $l(r_1') = 0$ and $l(r_2') = 0$. By \cite[Lemma 7.9]{HKT26}, this implies $r_1'$ is a single circle in column $H_{col}$, with the same sign as the last circle of $\BB_{1, temp}$. 

    If $r_2$ has support only in $H_{col}+1$, then the row exchanges with the circles in $\CC_2$ leave it unchanged. So it fails the alternating sign condition with $r_1'$, so no $ui$ is possible.

    If $r_2$ has support in multiple columns, then after row exchanging with circles in $\CC_2$, the result $r_2'$ has $l(r_2') = 1$ using much the same reasoning as before. So again no $ui$ is possible.

    Third suppose $i_2 = 2$. Since $r_2$ is then a single circle with support in $H_{col}$, this case has already been considered when $i_2 = 3$ and $r_2$ is a single circle with support in $H_{col}$, since row exchanges leave it unchanged.
\end{proof}

Using Lemma \ref{lem-case5-independence}, we show the existence of the decomposition described at the beginning of the subsection.

\begin{prop}
\label{prop-case5-decomposition}
    Let $\EE_{temp}$ be a tempered extended multi-segment with block decomposition $\BB_{1, temp} \cup \dots \cup \BB_{k, temp}$. Let $\EE$ be any extended multi-segment equivalent to $\Theta_1(\EE_{temp})$. Then there exists a decomposition $\CC_1 \cup \dots \cup \CC_{k-1}$ of $\EE$ such that
    \begin{align*}
        \CC_1 &\sim \Theta_1(\BB_{1, temp} \cup \BB_{2, temp}) \\
        \CC_i &\sim \BB_{i+1, temp} \text{ for } i = 2, \dots, k-1.
    \end{align*}
\end{prop}
\begin{proof}
    The proof is analogous to the proof of Proposition \ref{prop-independence-of-blocks} (see \cite[Lemma 7.12]{HKT26}). Observe that $\Theta_1(\EE_{temp})$ has such a decomposition, namely $\CC_1 = \Theta_1(\BB_{1, temp} \cup \BB_{2, temp})$ and $\CC_i = \BB_{i+1, temp}$ for $i > 1$. Since $\EE$ is equivalent to $\Theta_1(\EE_{temp})$, there exists a sequence of basic operations taking $\EE$ to $\Theta_1(\EE_{temp})$. So it suffices to show that each one of these operations preserves the decomposition. By Lemma \ref{lem-case5-independence}, each operation only involves rows from a single part $\CC_i$. Since operations are local (Lemma \ref{lem-local}), they do not affect the other parts of the decomposition. So we can just apply the operation to the relevant $\CC_i$, which gives a new decomposition.
\end{proof}

\subsubsection{Existence-type result}

We prove the following existence type result.

\begin{lemma}
\label{lem-existence-theta-lift}
    Let $\EE_{temp}$ be a tempered extended multi-segment with block decomposition $\BB_{1, temp} \cup \dots \cup \BB_{k, temp}$. Let $\EE$ be any extended multi-segment equivalent to $\Theta_1(\EE_{temp})$ with decomposition $\CC_1 \cup \dots \cup \CC_{k-1}$ such that $\CC_1 \sim \Theta_1(\BB_{1, temp} \cup \BB_{2, temp})$ and $\CC_i \sim \BB_{i+1, temp}$ for $i = 2, \dots, k-1$, which exists by Proposition \ref{prop-case5-decomposition}. Then any operation on $\CC_i$ as its own extended multi-segment is also a valid operation on $\EE$.
\end{lemma}
\begin{proof}
    The proof is analogous to the proof of Lemma \cite[Lemma 7.6]{HKT26}, but we sketch it again for clarity, since the context is slightly more general here.
    
    Any operation $T$ on $\CC_i$ is either a $ui$, a $dual \circ ui \circ dual$, a $ui\inv$ of type 3', or a $dual \circ ui\inv \circ dual$ where the $ui\inv$ is of type 3'. In the case that $T$ is a $ui$, it is clear that whether the conditions of union-intersection are satisfied is purely local. In the case that $T$ is a $dual \circ ui \circ dual$, we note that $dual(\EE) = dual(\CC_{k-1}) \cup \dots \cup dual(\CC_1)$ possibly up to a global sign change on each $dual(\CC_i)$. Hence this does not affect whether the union-intersection is valid.

    In the case that $T$ is a $ui\inv$ of type 3', the only thing we need to check is that the order is still admissible after the $ui\inv$. Since $T$ is a valid operation on $\CC_i$, the rows created satisfy the admissible order condition with other rows from $\CC_i$. Also since the decomposition has the staircase property (Lemma \ref{lem-staircase-theta-lift}), the rows created by applying a $ui\inv$ to $\CC_i$ cannot fail the admissible order condition with a row from $\CC_j$ for $j \neq i$.

    Finally, in the case that $T$ is a $dual \circ ui\inv \circ dual$, again the only thing we need to check is that the order is still admissible. Note that \[\max_{r \in \CC_i} |B(r)| \leq \max_{r \in \CC_i} A(r) \leq \min_{r \in \CC_{i+1}} B(r),\] where the first inequality follows from the definition of $A$ and $B$ and the second inequality follows from the fact that the decomposition has the staircase property. This implies that for any $i<j$, the support of every row in $dual(\CC_j)$ contains the support of every row in $dual(\CC_i)$. Hence the new rows created by applying a $ui\inv$ to $dual(\CC_i)$ cannot fail the admissible order condition with a row from $\CC_j$ for $j \neq i$, since it does not even apply.
\end{proof}

Finally, we prove Case 5 of Theorem \ref{thm-count-theta-temp}, which completes the proof of Theorem \ref{thm-count-theta-temp}.

\begin{proof}[Proof (of Case 5 of Theorem \ref{thm-count-theta-temp})]
    By Proposition \ref{prop-case5-decomposition}, every $\EE$ equivalent to $\EE_{temp}$ has a decomposition $\CC_1 \cup \dots \cup \CC_{k-1}$ with $\CC_1 \sim \Theta_1(\BB_{1, temp} \cup \BB_{2, temp})$ and $\CC_i \sim \BB_{i+1, temp}$ for $i > 1$. By repeatedly applying Lemma \ref{lem-existence-theta-lift}, all such combinations of $\CC_i$ are achievable.
\end{proof}

\subsection{Anti-tempered extended multi-segments}\label{sec Anti-Tempered Extended Multi-Segments}

Finally, calculating $\theta_{-m^{\up, \alpha}}^\up(\pi)$ when $\pi$ is an anti-tempered representation is simpler than the tempered case because $\theta_{-m^{\up, \alpha}}^\up(\pi)$ itself must be an anti-tempered representation. Moreover, given an anti-tempered extended multi-segment $\mathcal{E},$ we can easily identify a multi-segment associated to the theta lift $\theta_{-m^{\up, \alpha}}^\up(\pi(\mathcal{E})).$

\begin{lemma}
\label{theta-antitempered}
    Let $\mathcal{E}$ be an anti-tempered extended multi-segment in (P') order. Then the following holds.

    \begin{enumerate}
        \item The first row of $\mathcal{E}$ is of the form $([k, -k], k, \eta(\mathcal{E})).$
        \item $\mathcal{E}_{m^{\up, \alpha}}$ is of the form $([k + 1, -k - 1], k + 1, -\eta(\mathcal{E})) \cup \mathcal{E}.$
        \item $\mathcal{E}_{m^{\up, \alpha}}$ is itself an anti-tempered multi-segment.
    \end{enumerate}
\end{lemma}

\begin{proof}
    Part (1) follows immediately from the fact that $dual(\mathcal{E})$ must be tempered. Since the first row of $\mathcal{E}$ must dualize to a row with a single circle, this indeed must be a row of the form $([k, -k], n, \eta(\mathcal{E})).$ Furthermore, since we impose a (P') order, we note that no row of $\mathcal{E}$ is supported $\geq k.$

    Therefore, from Algorithm \ref{algo compute theta lift}, it follows that for $\alpha$ such that $\frac{\alpha - 1}{2} \geq n,$ $\theta_{-\alpha}^\eta(\pi(\mathcal{E})$ is associated to an extended multi-segment:
    $$\mathcal{E}' := \left( \left[ \frac{\alpha - 1}{2}, \frac{-\alpha + 1}{2}\right], \frac{\alpha - 1}{2}, \eta   \right) \cup \mathcal{E}.$$

    Forming such an extended multi-segment for the case where $\frac{\alpha - 1}{2} = n$, we note that $\pi(\mathcal{E}') \neq 0$ if and only if $\eta = \eta(\mathcal{E}).$ To see this, via Lemma \ref{lemma dual}, it suffices to show that $\pi(dual(\mathcal{E}')) \neq 0$ if and only if $\eta = \eta(\mathcal{E}').$ This is true because the hat $\left( \left[ \frac{\alpha - 1}{2}, \frac{-\alpha + 1}{2}\right], \frac{\alpha - 1}{2}, \eta   \right) = ([k, -k], k, \eta)$ is equal to the first row of $\mathcal{E},$ except possibly for the sign. If $\eta = \eta(\mathcal{E}),$ then $dual(\mathcal{E})$ will be a tempered extended multi-segment with an extra circle in the $n$th column and therefore will correspond to a non-zero tempered representation. If $\eta = -\eta(\mathcal{E}),$ then $dual(\mathcal{E})$ will have two circles in the same column with opposite signs and therefore the resulting representation vanishes.

    Therefore, it is clear that the sign $\eta(\mathcal{E})$ corresponds with the going-down tower, and that $m^{\up, \alpha}$ must satisfy $\frac{m^{\up, \alpha} - 1}{2} \geq k + 1.$ In fact, we have $\frac{m^{\up, \alpha} - 1}{2} = k + 1,$ since the extended multi-segment $([k + 1, -k - 1], k + 1, -\eta(\mathcal{E})) \cup \mathcal{E}$ clearly has a tempered dual. This suffices to prove Part (2), and also Part (3).
\end{proof}

The proof of Theorem \ref{theta-antitempered-count} now follows immediately from a combination of Lemma \ref{theta-antitempered} and Theorems \ref{thm-count-block-temp} and \ref{thm-count-temp}.

\appendix

\section{Motivation for Conjecture \ref{conj-count-theta-temp}}\label{sec motivation}

The goal of this appendix is to give an idea of how the formula in Conjecture \ref{conj-count-theta-temp} can be obtained. 

Throughout this appendix, we let $\rho$ be an orthogonal supercuspidal representation of some $\GL_d(F)$ (often omitted in the notation). Let $([\frac{\alpha-1}{2},-\frac{\alpha-1}{2}]_{\chi_W}, \frac{\alpha-1}{2}, \eta)$ denote the added extended segment in $\Theta_1(\EE)$ for some $\EE\in\VRep_\rho(G_n).$ From the point of view of Definition \ref{def-hat}, this extended segment is called a hat. Throughout this appendix, we shall refer to it as the hat. Often, we apply a $dual\circ ui \circ dual$ which involves the hat with some other row in $\EE.$ We refer to this operation as ``dualizing the hat'' (see the operations $D^1$ and $D^2$ described in Lemma \ref{lem SMUD operators}).

By a similar argument as in the proof of Theorem \ref{thm-count-temp} (see \cite[\S 7]{HKT26}), we can argue that the blocks after the first two are independent, in the sense that the count for the theta lift of $\EE$ is the same as the count for the theta lift of $\BB_1 \cup \BB_2$ multiplied by $\prod_{i=3}^k |\Psi(\pi(\BB_i))|$. So it suffices to consider the case when $\EE$ decomposes into two blocks $\BB_1$ and $\BB_2$.

Our approach is the following. We consider an arbitrary extended multi-segment $\EE'$ equivalent to a tempered extended multi-segment $\EE$, and ask what operations are possible on $\Theta_1(\EE')$ that were not possible in $\EE'$. Each of these operations gives a distinct new extended multi-segment. We then add the count of the number of these operations to $|\Psi(\pi(\BB_1))| \cdot |\Psi(\pi(\BB_2))|$, which is the number of extended multi-segments $\EE'$ equivalent to the tempered $\EE$.

For each $\Theta_1(\EE')$, we claim that the operations are of the following type.
\begin{enumerate}
    \item We can perform some number of $dual \circ ui \circ dual$ operations between the hat in the first row and another row. The total number of extended multi-segments coming from the this type of operation is given by $D$.
    \item We can perform additional $ui$ operations between rows in extended multi-segments equivalent to $\BB_1$ and $\BB_2$. (In other words, the blocks are not independent in the sense of Lemma \ref{lem-independence-of-blocks}.) The number of extended multi-segments coming from the second type of operation is given by $U$.
\end{enumerate}

To give a count of $D$, we note that there are usually exactly two $dual \circ ui \circ dual$ operations involving the hat in the first row. This is for basically the same reasons as in the almost-block case (see Lemma \ref{duds for almost block}): as long as the first circle in column $H_{col}$ is combined with some row in $\BB_1'$, either in a row of circles or in a hat, we can dualize the hat to that row; moreover, since $m_H$ is even and hence greater than $1$, there is always at least one additional circle in column $H_{col}$ leftover to which the hat can dualize. There cannot be more than two dualizations because only the first circle in $H_{col}$ can interact with $\BB_1'$, and the dualizations to the remaining circles all differ by row exchanges.

The exceptions to the above reasoning are the following. In the case that $m_H-1>1$ the only exception is that the first circle in $H_{col}$ might not be combined with anything in $\BB_1'$. In this case there are $|\Psi(\pi(\BB_1'))| \cdot |\Psi(\pi(\BB_2))|$ possibilities. This explains the second case of the formula for $D$, which is much the same as in the almost block-case (Theorem \ref{thm-count-block-theta-temp}).

The case that $m_H - 1 = 1$ is more subtle. In these cases it is possible that there is only one $dual \circ ui \circ dual$ for another reason: the last circle in $H_{col}$, which belongs to $\BB_2$, might be $ui$'d with another circle in $H_{col}+1$, thereby ``blocking'' the $dual \circ ui \circ dual$ with the hat. For example, in the extended multi-segment below, there is only one $dual \circ ui \circ dual$ possible despite the fact that the first circle in column $3$ has been combined with previous rows. 
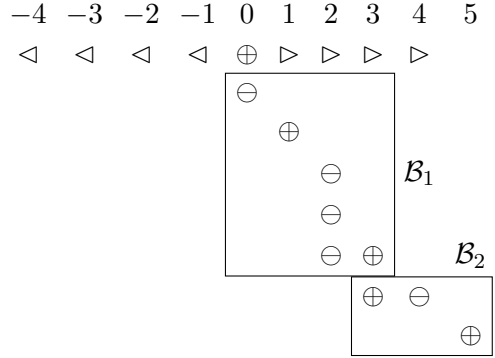
\[\begin{tikzpicture}[baseline=(current bounding box.north)]
  \matrix (m)[matrix of math nodes, nodes in empty cells] {
    -4 & -3 & -2 & -1 & 0 & 1 & 2 & 3 & 4 & 5 \\
    \lhd & \lhd & \lhd & \lhd & \oplus & \rhd & \rhd & \rhd & \rhd & \\
    & & & & \ominus & & & & & \\
    & & & & & \oplus & & & & \\
    & & & & & & \ominus & & \BB_1 & \\
    & & & & & & \ominus & & & \\
    & & & & & & \ominus & \oplus & & \BB_2 \\
    & & & & & & & \oplus & \ominus & \\
    & & & & & & & & & \oplus \\
  } ;
  \draw (m-3-5.north west) rectangle (m-7-8.south east);
  \draw (m-8-8.north west) rectangle (m-9-10.south east);
\end{tikzpicture}\] So, we count as follows:
\begin{itemize}
    \item start with two $dual \circ ui \circ dual$ operations for each of the $\EE'$, giving \[2 |\Psi(\pi(\BB_1))| \cdot |\Psi(\pi(\BB_2))|,\]
    \item subtract the case that there is only one $dual \circ ui \circ dual$ operation because the first circle in $H_{col}$ is unchanged, of which there are \[|\Psi(\pi(\BB_1'))| \cdot |\Psi(\pi(\BB_2))|,\]
    \item subtract the case that there is only one $dual \circ ui \circ dual$ operation because the second circle in $H_{col}$ has been combined with other circles, of which there are \[|\Psi(\pi(\BB_1))| \cdot (|\Psi(\pi(\BB_2))| - |\Psi(\pi(\BB_2'))|),\]
    \item add back the overlap between the previous two cases, of which there are \[|\Psi(\pi(\BB_1'))| \cdot (|\Psi(\pi(\BB_2))| - |\Psi(\pi(\BB_2'))|).\]
\end{itemize}
The first two formulas are straightforward. The third formula can be derived by complementary counting. We want to find the number of ways to form a $\BB_2''$ equivalent to $\BB_2$ such that the first row has more than one circle. The total number of such $\BB_2''$ is $|\Psi(\pi(\BB_2))|$, and the number of such $\BB_2''$ where the first row has only one circle is just $|\Psi(\pi(\BB_2'))|$, since we do not allow operations on the first row. Finally the fourth formula can be derived in a similar way. Combining the formulas and simplifying gives the claimed value of $D$ for $m_H - 1 = 1$ in Conjecture \ref{conj-count-theta-temp}.

Continuing in the above example, we have $\BB_1'$ and $\BB_2'$ are as follows.
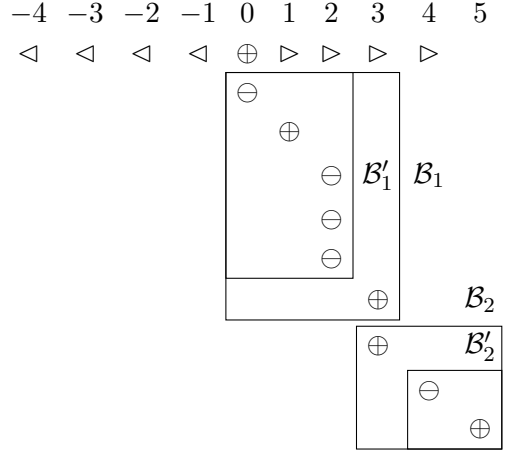
\[\begin{tikzpicture}[baseline=(current bounding box.north)]
  \matrix (m)[matrix of math nodes, nodes in empty cells] {
    -4 & -3 & -2 & -1 & 0 & 1 & 2 & 3 & 4 & 5 \\
    \lhd & \lhd & \lhd & \lhd & \oplus & \rhd & \rhd & \rhd & \rhd & \\
    & & & & \ominus & & & & & \\
    & & & & & \oplus & & & & \\
    & & & & & & \ominus & \BB_1' & \BB_1 & \\
    & & & & & & \ominus & & & \\
    & & & & & & \ominus & & & \\
    & & & & & & & \oplus & & \BB_2 \\
    & & & & & & & \oplus & & \BB_2' \\
    & & & & & & & & \ominus & \\
    & & & & & & & & & \oplus \\
  } ;
  \draw (m-3-5.north west) rectangle (m-8-8.south east);
  \draw (m-3-5.north west) rectangle (m-7-7.south east);
  \draw (m-9-8.north west) rectangle (m-11-10.south east);
  \draw (m-10-9.north west) rectangle (m-11-10.south east);
\end{tikzpicture}\] We have the counts $|\Psi(\pi(\BB_1))| = 33, |\Psi(\pi(\BB_1'))| = 9, |\Psi(\pi(\BB_2))| = 4,$ and $|\Psi(\pi(\BB_2'))| = 2$, so the formula gives \[D = 2 (33 \cdot 4) - (9 \cdot 4 + 33 \cdot (4 - 2) - 9 \cdot (4 - 2)) = 180.\]

Finally we explain the count of $U$, the number of additional $ui$ operations involving rows from extended multi-segments equivalent to $\BB_1$ and $\BB_2$. These occur in the following way: first, we perform a $dual \circ ui \circ dual$ operation between the hat and the last circle in column $H_{col}$, creating a row with a pair of triangles. (We expect this operation is always possible, even if the extended multi-segment is not tempered.) Then, if the row immediately following the pair of triangles has more than one circle, we can perform a union-intersection between the two rows. The number of new extended multi-segments created by this operation is possible is $|\Psi(\pi(\BB_1))| \cdot (|\Psi(\pi(\BB_2'))| - |\Psi(\pi(\BB_2''))|)$. The term $|\Psi(\pi(\BB_1))|$ comes from the fact that the first part of the extended multi-segment can be anything equivalent to $\BB_1$, and the term $|\Psi(\pi(\BB_2'))| - |\Psi(\pi(\BB_2''))|$ comes from the fact that we need the first row with support in $H_{col}+1$ to have more than one circle.

As an example, consider the following theta lift of a tempered extended multi-segment.
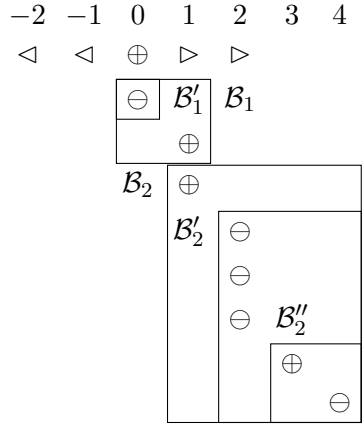
\[\begin{tikzpicture}[baseline=(current bounding box.north)]
  \matrix (m)[matrix of math nodes, nodes in empty cells] {
    -2 & -1 & 0 & 1 & 2 & 3 & 4 \\
    \lhd & \lhd & \oplus & \rhd & \rhd & & \\
    & & \ominus & \BB_1' & \BB_1 & & \\
    & & & \oplus & & & \\
    & & \BB_2 & \oplus & & & \\
    & & & \BB_2' & \ominus & & \\
    & & & & \ominus & & \\
    & & & & \ominus & \BB_2'' & \\
    & & & & & \oplus & \\
    & & & & & & \ominus \\
  } ;
  \draw (m-3-3.north west) rectangle (m-4-4.south east);
  \draw (m-3-3.north west) rectangle (m-3-3.south east);
  \draw (m-5-4.north west) rectangle (m-10-7.south east);
  \draw (m-6-5.north west) rectangle (m-10-7.south east);
  \draw (m-9-6.north west) rectangle (m-10-7.south east);
\end{tikzpicture}\]

An additional $6$ extended multi-segments are created by these $ui$ operations. One of them can be created by the following process. Note that the union-intersection is only valid in the last step because the $4$th row has more than one circle.

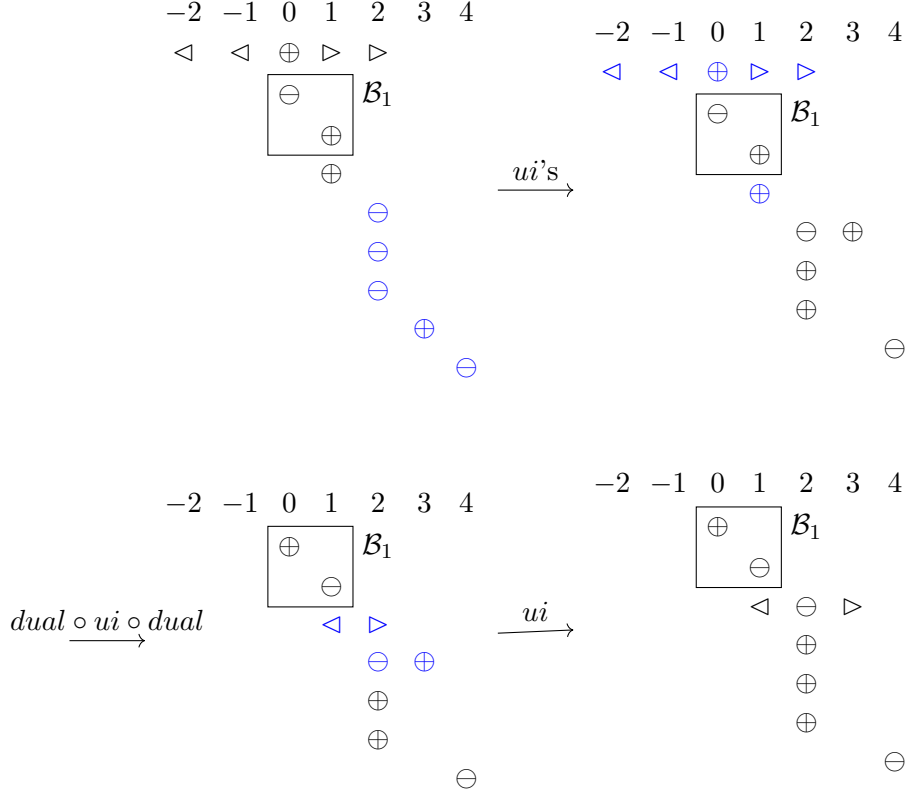
\[\begin{tikzpicture}[auto]
  \matrix (m)[matrix of math nodes, nodes in empty cells, row 6/.style=blue, row 7/.style=blue, row 8/.style=blue, row 9/.style=blue, row 10/.style=blue] {
    -2 & -1 & 0 & 1 & 2 & 3 & 4 \\
    \lhd & \lhd & \oplus & \rhd & \rhd & & \\
    & & \ominus & & \BB_1 & & \\
    & & & \oplus & & & \\
    & & & \oplus & & & \\
    & & & & \ominus & & \\
    & & & & \ominus & & \\
    & & & & \ominus & & \\
    & & & & & \oplus & \\
    & & & & & & \ominus \\
  } ;
  \matrix (m1) [right=of m, matrix of math nodes, nodes in empty cells, row 2/.style=blue, row 5/.style=blue] {
    -2 & -1 & 0 & 1 & 2 & 3 & 4 \\
    \lhd & \lhd & \oplus & \rhd & \rhd & & \\
    & & \ominus & & \BB_1 & & \\
    & & & \oplus & & & \\
    & & & \oplus & & & \\
    & & & & \ominus & \oplus & \\
    & & & & \oplus & & \\
    & & & & \oplus & & \\
    & & & & & & \ominus \\
  } ;
  \matrix (m2) [below=of m, matrix of math nodes, nodes in empty cells, row 4/.style=blue, row 5/.style=blue] {
    -2 & -1 & 0 & 1 & 2 & 3 & 4 \\
    & & \oplus & & \BB_1 & & \\
    & & & \ominus & & & \\
    & & & \lhd & \rhd & & \\
    & & & & \ominus & \oplus & \\
    & & & & \oplus & & \\
    & & & & \oplus & & \\
    & & & & & & \ominus \\
  } ;
  \matrix (m3) [below=of m1, matrix of math nodes, nodes in empty cells] {
    -2 & -1 & 0 & 1 & 2 & 3 & 4 \\
    & & \oplus & & \BB_1 & & \\
    & & & \ominus & & & \\
    & & & \lhd & \ominus & \rhd & \\
    & & & & \oplus & & \\
    & & & & \oplus & & \\
    & & & & \oplus & & \\
    & & & & & & \ominus \\
  } ;
  \node (ghost) [left=of m2] {};
  \draw (m-3-3.north west) rectangle (m-4-4.south east);
  \draw (m1-3-3.north west) rectangle (m1-4-4.south east);
  \draw (m2-2-3.north west) rectangle (m2-3-4.south east);
  \draw (m3-2-3.north west) rectangle (m3-3-4.south east);
  \draw [->] (m) to node {$ui$'s} (m1);
  \draw [->] (ghost) to node {$dual \circ ui \circ dual$} (m2);
  \draw [->] (m2) to node {$ui$} (m3);
\end{tikzpicture}\]

\section{Commutativity Results}\label{sec commutativity}

Throughout this appendix, we assume that $\rho$ is an orthogonal representation of $\GL_d(F)$ and often suppress it in the notation. 
In the language of extended multi-segments, the spirit of the Adams conjecture in the case of the first occurrence of the going-up tower is that for some $\mathcal{E}\in\Rep_\rho(G_n),$ we have
$$\Psi(\theta_{-m^{\up, \alpha}}^\up (\pi(\mathcal{E}))) \supseteq \{\psi_{\mathcal{E}'_{m^{\up, \alpha}}} \mid \mathcal{E}' \sim \mathcal{E}\}.$$

Our previous results show that it is too much to hopeful for an equality between these sets generally. However, in many cases, they offer a meaningful resolution for tempered $\mathcal{E}$. In particular, Theorem \ref{first block m_n = 1} describes a case in which $\Psi(\theta_{-m^{\up, \alpha}}^\up \pi(\mathcal{E}))$ is equal not to $\{\psi_{\mathcal{E}'_{m^{\up, \alpha}}} \mid \mathcal{E}' \sim \mathcal{E}\},$ but to three ``copies" of this set. Theorem \ref{first block m_n > 1} offers a bit more complexity. In this case, $\Psi(\theta_{-m^{\up, \alpha}}^\up \pi(\mathcal{E}))$ is equal to the union of four sets like $\{\psi_{\mathcal{E}'_{m^{\up, \alpha}}} \mid \mathcal{E}' \sim \mathcal{E}\},$ but two of these sets may have some overlap. Here, we offer an elaboration of these results that better elucidates the connection between these sets, rather than simply counting their sizes.

\subsection{A motivating example}

Consider the following.

$$\mathcal{E} = \bordermatrix{& 0 & 1 & 2 & 3 \cr & \oplus \cr & & \ominus \cr& & & \ominus \cr & & & & \oplus}, ~~~~~~~~~~  \mathcal{E}_{m^{\up, \alpha}} = \bordermatrix{& -2 & -1 & 0 & 1 & 2 & 3 \cr & \lhd & \lhd & \ominus & \rhd & \rhd \cr & & & \oplus \cr & & & &  \ominus \cr & & & & & \ominus \cr & & & & & & \oplus}.$$

The multi-segment $\mathcal{E}\in\Rep(G_n)$ can be written in the block decomposition $\mathcal{B}_1 \cup \mathcal{B}_2,$ where $|\Psi(\pi(\mathcal{B}_1))| = 3$ and $|\Psi(\pi(\mathcal{B}_2))| = 2,$ from which it follows that $|\Psi(\pi(\mathcal{E}))| = 6$ via Theorem \ref{thm-count-temp}. These six extended multi-segments and all connecting raising operators are depicted below.

\[\begin{tikzcd}[scale cd = 0.7, sep = small, every label/.append style={font=\tiny}]
	\bordermatrix{& -1 & 0 & 1 & 2 & 3 \cr & \lhd & \oplus & \rhd \cr & & \ominus \cr & & & &  \ominus \cr & & & & & \oplus} && \bordermatrix{& 0 & 1 & 2 & 3 \cr & \oplus & \ominus \cr& & & \ominus \cr & & & & \oplus} && \bordermatrix{& 0 & 1 & 2 & 3 \cr & \oplus \cr & & \ominus \cr& & & \ominus \cr & & & & \oplus} \\
	\\
	\bordermatrix{& -1 & 0 & 1 & 2 & 3 \cr & \lhd & \oplus & \rhd \cr & & \ominus \cr & & & &  \ominus & \oplus} && \bordermatrix{& 0 & 1 & 2 & 3 \cr & \oplus & \ominus \cr& & & \ominus & \oplus} && \bordermatrix{& 0 & 1 & 2 & 3 \cr & \oplus \cr & & \ominus \cr& & & \ominus & \oplus}
	\arrow["{dual \circ ui \circ dual}", from=1-1, to=1-3]
	\arrow["{ui^{-1}}", from=1-3, to=1-5]
	\arrow["{ui^{-1}}", from=3-1, to=1-1]
	\arrow["{dual \circ ui \circ dual}", from=3-1, to=3-3]
	\arrow["{ui^{-1}}", from=3-3, to=1-3]
	\arrow["{ui^{-1}}", from=3-3, to=3-5]
	\arrow["{ui^{-1}}"', from=3-5, to=1-5]
\end{tikzcd}\]

The block $\mathcal{B}_1$ is of type $Y_\mathcal{M_\B}$ with $\mathcal{M}_\mathcal{B} = (m_0, m_1) = (1, 1),$ so Theorem \ref{first block m_n = 1} guarantees that $\Psi(\Theta(\mathcal{B}_1)) = \{\psi_{\Theta_i(\mathcal{B}')} \mid \mathcal{B}' \sim \mathcal{B}_1, i \in \{1, 2, 3\}\}.$ Independence of blocks, as stated in Proposition \ref{prop-independence-of-blocks}, guarantees a similar 3-to-1 correspondence between extended multi-segments equivalent to $\mathcal{E}_{m^{\up, \alpha}}$ and $\mathcal{E}$ respectively. The following three diagrams show the layers of multi-segments $\{\Theta_i(\mathcal{E}') \mid \mathcal{E}' \sim \mathcal{E}\}$ for $i = 1, 2,$ and $3$ respectively.

\begin{center}
\begin{tabular}{|c|c|}
\hline
    $i = 1$ &  \begin{tikzcd}[scale cd = 0.6, sep = small, every label/.append style={font=\tiny}]
	{\let\quad\thinspace\bordermatrix{& -2 & -1 & 0 & 1 & 2 & 3 \cr & \lhd & \lhd & \ominus &  \rhd & \rhd \cr && \lhd & \oplus & \rhd \cr & & & \ominus \cr & & & & & \ominus \cr & & & & & &\oplus}} && {\let\quad\thinspace\bordermatrix{& -2 & -1 & 0 & 1 & 2 & 3 \cr & \lhd & \lhd & \ominus & \rhd & \rhd \cr & & & \oplus & \ominus \cr& & &&& \ominus \cr & & & &&& \oplus}} && {\let\quad\thinspace\bordermatrix{& -2 & -1 & 0 & 1 & 2 & 3 \cr & \lhd & \lhd & \ominus & \rhd & \rhd \cr &&& \oplus \cr & &&& \ominus \cr& & &&& \ominus \cr & & & &&& \oplus}} \\
	\\
	{\let\quad\thinspace\bordermatrix{& -2 & -1 & 0 & 1 & 2 & 3 \cr & \lhd & \lhd & \ominus &  \rhd & \rhd \cr && \lhd & \oplus & \rhd \cr & & & \ominus \cr & & & & & \ominus & \oplus}} && {\let\quad\thinspace\bordermatrix{& -2 & -1 & 0 & 1 & 2 & 3 \cr & \lhd & \lhd & \ominus & \rhd & \rhd \cr & & & \oplus & \ominus \cr& & & & & \ominus & \oplus}} && {\let\quad\thinspace\bordermatrix{& -2 & -1 & 0 & 1 & 2 & 3 \cr & \lhd & \lhd & \ominus & \rhd & \rhd \cr & & & \oplus \cr & & & & \ominus \cr& & & & & \ominus & \oplus}}
	\arrow["{dual \circ ui \circ dual}", from=1-1, to=1-3]
	\arrow["{ui^{-1}}", from=1-3, to=1-5]
	\arrow["{ui^{-1}}", from=3-1, to=1-1]
	\arrow["{dual \circ ui \circ dual}", from=3-1, to=3-3]
	\arrow["{ui^{-1}}", from=3-3, to=1-3]
	\arrow["{ui^{-1}}", from=3-3, to=3-5]
	\arrow["{ui^{-1}}"', from=3-5, to=1-5]
\end{tikzcd} 
\\
     \hline
     $i = 2$ & \begin{tikzcd}[scale cd = 0.6, sep = small, every label/.append style={font=\tiny}]
	{\let\quad\thinspace\bordermatrix{& -1 & 0 & 1 & 2 & 3 \cr & \lhd & \ominus & \oplus & \rhd \cr & & \ominus \cr & & & &  \ominus \cr & & & & & \oplus}} && {\let\quad\thinspace\bordermatrix{& 0 & 1 & 2 & 3 \cr & \ominus & \oplus & \ominus \cr& & & \ominus \cr & & & & \oplus}} && {\let\quad\thinspace\bordermatrix{& 0 & 1 & 2 & 3 \cr & \ominus \cr & & \oplus & \ominus \cr& & & \ominus \cr & & & & \oplus}} \\
	\\
	{\let\quad\thinspace\bordermatrix{& -1 & 0 & 1 & 2 & 3 \cr & \lhd & \ominus & \oplus & \rhd \cr & & \ominus \cr & & & &  \ominus & \oplus}} && {\let\quad\thinspace\bordermatrix{& 0 & 1 & 2 & 3 \cr & \ominus & \oplus & \ominus \cr& & & \ominus & \oplus}} && {\let\quad\thinspace\bordermatrix{& 0 & 1 & 2 & 3 \cr & \ominus  \cr & & \oplus & \ominus \cr& & & \ominus & \oplus}}
	\arrow["{dual \circ ui \circ dual}", from=1-1, to=1-3]
	\arrow["{ui^{-1}}", from=1-3, to=1-5]
	\arrow["{ui^{-1}}", from=3-1, to=1-1]
	\arrow["{dual \circ ui \circ dual}", from=3-1, to=3-3]
	\arrow["{ui^{-1}}", from=3-3, to=1-3]
	\arrow["{ui^{-1}}", from=3-3, to=3-5]
	\arrow["{ui^{-1}}"', from=3-5, to=1-5]
\end{tikzcd}

\\
     \hline
     
     $i = 3$ & \begin{tikzcd}[scale cd = 0.6, sep = small, every label/.append style={font=\tiny}]
	{\let\quad\thinspace\bordermatrix{& -1 & 0 & 1 & 2 & 3 \cr & \lhd & \ominus & \rhd \cr & & \oplus \cr & & & &  \ominus \cr & & & &  \ominus \cr & & & & & \oplus}} && {\let\quad\thinspace\bordermatrix{& 0 & 1 & 2 & 3 \cr & \ominus & \oplus \cr& & & \ominus \cr& & & \ominus \cr & & & & \oplus}} && {\let\quad\thinspace\bordermatrix{& 0 & 1 & 2 & 3 \cr & \ominus \cr & & \oplus \cr& & & \ominus \cr& & & \ominus \cr & & & & \oplus}} \\
	\\
	{\let\quad\thinspace\bordermatrix{& -1 & 0 & 1 & 2 & 3 \cr & \lhd & \ominus & \rhd \cr & & \oplus \cr & & & &  \ominus \cr & & & &  \ominus & \oplus}} && {\let\quad\thinspace\bordermatrix{& 0 & 1 & 2 & 3 \cr & \ominus & \oplus \cr& & & \ominus \cr& & & \ominus & \oplus}} && {\let\quad\thinspace\bordermatrix{& 0 & 1 & 2 & 3 \cr & \ominus \cr & & \oplus \cr& & & \ominus \cr& & & \ominus & \oplus}}
	\arrow["{dual \circ ui \circ dual}", from=1-1, to=1-3]
	\arrow["{ui^{-1}}", from=1-3, to=1-5]
	\arrow["{ui^{-1}}", from=3-1, to=1-1]
	\arrow["{dual \circ ui \circ dual}", from=3-1, to=3-3]
	\arrow["{ui^{-1}}", from=3-3, to=1-3]
	\arrow["{ui^{-1}}", from=3-3, to=3-5]
	\arrow["{ui^{-1}}"', from=3-5, to=1-5]
\end{tikzcd}
\\

\hline
\end{tabular}
\end{center}

Not only does each of these layers have the same number of multi-segments as $\Psi(\pi(\mathcal{E})),$ the arrows between the corresponding multi-segments is exactly the same in each diagram. In other words, raising operators between multi-segments equivalent to $\mathcal{E}$ seem to ``commute" with the $dual \circ ui \circ dual$ and $ui^{-1}$ operators which go between $\Theta_1, \Theta_2,$ and $\Theta_3.$ In a sense, then, each set $\{\Theta_i(\mathcal{E}') \mid \mathcal{E}' \sim \mathcal{E}\}$ is a precise copy of the set $\{\mathcal{E}' \mid \mathcal{E}' \sim \mathcal{E}\}.$ This leads us to the main theorem we aim to prove regarding the structure of row operations on $\mathcal{E}_{m^{\up, \alpha}},$ which we aim to prove at the end of this section.

\begin{thm}
\label{commutativity theta correspondence}
    Let $\mathcal{E}$ be of type $Y_\mathcal{M}$ with $\mathcal{M}$ starting at zero, such that $\Psi(\pi(\mathcal{E})) = \{\psi(\Theta_i(\mathcal{E})) \mid \mathcal{E}' \sim \mathcal{E}, i \in I_\mathcal{E}\},$ for some index set $I_\mathcal{E} = \{1, 2, 3, 4\}$ or $\{1, 2, 3\}$ depending on $\mathcal{E}.$ Then, if $\EE' \sim \EE'' \sim \EE$ but $\EE' \neq \EE''$, there exists a raising operator $T \in \{dual \circ ui \circ dual, ui^{-1}\}$ from $\Theta_i(\EE')$ to $\Theta_j(\EE'')$ if and only if $i = j$ and the same raising operator $T$ can be applied from $\EE'$ to $\EE''.$
\end{thm}

\begin{rmk}
    We have stated this theorem for the case where $\mathcal{E}$ is of type $Y_\mathcal{M}.$ Similar results can be deduced for other cases, such as

    \begin{itemize}
        \item $\mathcal{E}$ is an almost-block, given the structure detailed in Lemma \ref{Almost block classification},
        \item $\mathcal{E}$ if of type $Y_\mathcal{M}$ with $\mathcal{M}$ starting after zero; i.e., $\mathcal{E}$ is equivalent to a tempered multi-segment not starting at zero, or
        \item $\mathcal{E}$ has a block decomposition similar to the one in the above example.
    \end{itemize}

    For the sake of brevity, we only prove the stated theorem.
\end{rmk}

\subsection{Commutativity results for theta correspondence}

All of the operators of type 3'; $S; M; U; D^1$; and $D^2$, are, in a sense, `local.' Each of them affects $1-2$ rows in structurally predictable ways and and leaves the others completely intact, modulo a possible sign change. This means that many of these operators can be performed simultaneously, independently of each other. In a sense, they commute with each other. This idea is the primary mechanism behind Theorem \ref{commutativity theta correspondence}. We make it precise in the following theorem. We also provide an example to illustrate each specific commutativity result in the theorem.

\begin{thm}[Commutativity of Operators]
\label{Commutativity of Operators} Let $\EE$ be of type $Y_\mathcal{M},$ with $\mathcal{M}$ starting at zero.
    \begin{enumerate}
        \item Any two operators of type 3' involving distinct rows commute.
        
        \item Given a chain $r,$ we have:
        $$S_{r', k_2} \circ S_{r, k_1} = S_{r'', k_1} \circ S_{r, k_1 + k_2}.$$
        Here, $r'$ is the row remaining after $k_1$ circles have been taken out of $r,$ while $r''$ is the row of $k_1 + k_2$ circles removed from $r.$

        \item Given consecutive hats $h_1 < h_2 < h_3,$ we have
        $$(h_1 * h_2) * h_3 = h_1 * (h_2 * h_3).$$

        \item Given mergeable hats $h_1, h_2$ and a row of circles $r$, we have
        $$D_{h_1, r'} \circ D_{h_2, r} = D_{h_1 * h_2, r} \circ M_{h_1, h_2},$$
        where $r'$ is the image of $r$ under $D.$

        \item Given mergeable hats $h_1, h_2$, we have
        $$M_{h_1', h_2} \circ U_{h_1, k} = U_{h_1 * h_2, k} \circ M_{h_1, h_2},$$
        where $h_1'$ is the image of $h_1$ after $U.$

        \item Given mergeable hats $h_1, h_2$, we have
        $$D_{h_1, r} \circ U_{h_2, k} = U_{h_1 * h_2, C(h_1) + k} \circ M_{h_1, h_2},$$
        where $r$ is the row of circles in the image of $h_2$ under $U.$

        \item Given a hat $h$, we have
        $$U_{h', k_2} \circ U_{h, k_1} = S_{r, k_1} \circ U_{h, k_1 + k_2},$$
        where $h'$ is the remaining part of of $h$ when $k_1$ circles are pulled out through $U$, and $r$ is the row of $k_1 + k_2$ circles pulled out.

        \item Given a hat $h$ and row of circles $r$, we have
        $$S_{r', C(h) + k} \circ D_{h, r} = D_{h, r_2} \circ S_{r, k},$$
        where $r'$ is the image of $r$ after $D,$ and $r_2$ is the second row in the image of $r$ under $S.$

        \item Given a hat $h$ and row of circles $r$, we have
        $$D_{h', r} \circ U_{h, k} = S_{r', k} \circ D_{h, r},$$
        where $h'$ is the remaining hat after $U,$ and $r'$ is the image of $r$ under $D.$
    \end{enumerate}
\end{thm}

\begin{proof}
    We note that for the duration of this proof, it suffices to show that two combinations of operations yield the same virtual extended multi-segments up to the signs. That is, we need only check the effect on the $\mathcal{S}$-data, and since all these operations preserve $\eta$ and the sign conditions for Definition \ref{defn E(M,S)}, all the signs must match. Parts (2) through (9) thus follow directly from the formulations of the operators in Lemma \ref{lem SMUD operators}. More details can be found in our earlier manuscript, \cite[\S 12]{HKT25}. Meanwhile, Part (1) follows from the fact that all four operations are local. That is, they affect $1-2$ segments in a prescribed manner and leave the other segments unaffected up to sign.
\end{proof}

We now prove Theorem \ref{commutativity theta correspondence}. We do this by proving the theorem for the case where $T$ is a raising operator of type $3'.$ From this, we immediately obtain the result for the case where $T$ is a lowering operator of type $3'.$ All other operators are compositions of such operators of type 3'. The proof for raising operators of type 3' follows from the following series of lemmas.

\begin{lemma}
\label{commute 0 to 1}
    Let $\EE$ be of type $Y_\mathcal{M}$ with $\mathcal{M}$ starting at zero. Suppose there exists a raising operator $T \in \{ dual \circ ui \circ dual, ui^{-1}\}$ of type 3' from $\Theta_1(\EE)$ to $\EE'.$ If $T$ is not a $dual \circ ui \circ dual$ involving the first row of $\Theta_1(\EE),$ then there exists $\EE'' \sim \EE$ such that the following diagram commutes.
\[\begin{tikzcd}
	{\Theta_1(\EE)} & {\EE' = \Theta_1(\EE'')} \\
	{\EE} & {\EE''}
	\arrow["T", from=1-1, to=1-2]
	\arrow["\theta", from=2-1, to=1-1]
	\arrow["T", from=2-1, to=2-2]
	\arrow["\theta", from=2-2, to=1-2]
\end{tikzcd}\]
\end{lemma}

\begin{proof}
    Let $h = ([c_{\max} + 1, -c_{\max} - 1], c_{\max} + 1, -\eta(\EE))$ be the first row of $\Theta_1(\EE).$ Note that there is only one possible row operation, a $dual \circ ui\circ dual$ involving $h.$ Therefore, $T$ must involve rows in $\EE$. 
    It suffices to prove that such an operation can be implemented regardless of whether $h$ is present. This is true for all the operators $S, M, U, D$ by Lemma \ref{lem SMUD operators}.
\end{proof}

The above lemma essentially states that almost all raising operators on $\Theta_1(\EE)$ produce other multi-segments of the form $\Theta_1$ and descend to be operators on $\EE.$ It is also true that case 3' operators on $\EE$ can ascend to $\Theta_1(\EE),$ as stated in the following lemma.

\begin{lemma}
\label{commute 1 to 0}
    Let $\EE$ be of type $Y_\mathcal{M}$ with $\mathcal{M}$ starting at zero. Suppose there exists a raising operator $T \in \{dual \circ ui \circ dual, ui^{-1}\}$ of type 3' from $\EE$ to $\EE'.$ Then the following diagram commutes.
    \[\begin{tikzcd}
	{\Theta_1(\EE)} & {\Theta_1(\EE')} \\
	{\EE} & {\EE'}
	\arrow["T", from=1-1, to=1-2]
	\arrow["\theta", from=2-1, to=1-1]
	\arrow["T", from=2-1, to=2-2]
	\arrow["\theta", from=2-2, to=1-2]
    \end{tikzcd}\]
\end{lemma}

\begin{proof}
    We have already shown in the proof to Lemma \ref{commute 0 to 1} that any raising operator $T$ can be implemented regardless of the presence of $h = [c_{\max} + 1, -c_{\max} - 1], c_{\max} + 1, -\eta(\EE))$. This suffices to prove the lemma.
\end{proof}

In the next two lemmas, we show that operators can similarly be lifted between $\Theta_1$ and $\Theta_2.$

\begin{lemma}
\label{commute 1 to 2}
    Let $\EE$ be of type $Y_\mathcal{M}$ with $\mathcal{M}$ starting at zero and let $h = ([c_{\max} + 1, -c_{\max} - 1], c_{\max} + 1, -\eta(\EE))$ be the first row of $\Theta_1(\EE).$ Let $T \in \{dual \circ ui \circ dual, ui^{-1}\}$ be a raising operator of type 3' on $\Theta_1(\EE)$ which is not equal to the $dual \circ ui \circ dual$ that gives $\Theta_2(\EE)$ (or $\Theta_4(\EE),$ should it exist). Then the following diagram commutes.
    \[\begin{tikzcd}
	{\Theta_2(\EE)} & {\Theta_2(\EE')} \\
	{\Theta_1(\EE)} & {\Theta_1(\EE')}
	\arrow["T", from=1-1, to=1-2]
	\arrow["{dual \circ ui \circ dual}", from=2-1, to=1-1]
	\arrow["T", from=2-1, to=2-2]
	\arrow["{dual \circ ui \circ dual}"', from=2-2, to=1-2]
\end{tikzcd}\]
\end{lemma}

\begin{proof}
    Firstly, we know that there exists $\EE'$ such that $T$ goes from $\Theta_1(\EE)$ to $\Theta_1(\EE')$ by Lemma \ref{commute 0 to 1}. Let $r$ be the row that $h$ is combined with in the $dual \circ ui \circ dual$ going from $\Theta_1(\EE)$ to $\Theta_2(\EE).$ If $T$ does not involve $h$ or $r$ then the lemma follows from Part (1) of Theorem \ref{Commutativity of Operators}. Since there are no other raising operators that involve $h,$ we reduce to the case where $T$ involves $r.$

    If the unique $dual \circ ui \circ dual$ involving $h$ is of the form $M$ and $r$ is a hat, then there are three possibilities for $T$ as follows.
    \begin{itemize}
        \item If $T$ is of the form $U,$ then Part (6) of Theorem \ref{Commutativity of Operators} guarantees the existence of the the following commutative diagram.
        \[\begin{tikzcd}
	{\Theta_2(\EE)} & {\Theta_2(\EE')} \\
	{\Theta_1(\EE)} & {\Theta_1(\EE')}
	\arrow["U", from=1-1, to=1-2]
	\arrow["M", from=2-1, to=1-1]
	\arrow["U", from=2-1, to=2-2]
	\arrow["D"', from=2-2, to=1-2]
\end{tikzcd}\]

        \item If $T$ is of the form $D$, then Part (4) of Theorem \ref{Commutativity of Operators} gives the following commutative diagram.
        \[\begin{tikzcd}
	{\Theta_2(\EE)} & {\Theta_2(\EE')} \\
	{\Theta_1(\EE)} & {\Theta_1(\EE')}
	\arrow["D", from=1-1, to=1-2]
	\arrow["M", from=2-1, to=1-1]
	\arrow["D", from=2-1, to=2-2]
	\arrow["D"', from=2-2, to=1-2]
\end{tikzcd}\]

        \item If $T$ is of the form $M,$ then Part (3) of Theorem \ref{Commutativity of Operators} gives a commutative diagram.
        \[\begin{tikzcd}
	{\Theta_2(\EE)} & {\Theta_2(\EE')} \\
	{\Theta_1(\EE)} & {\Theta_1(\EE')}
	\arrow["M", from=1-1, to=1-2]
	\arrow["M", from=2-1, to=1-1]
	\arrow["M", from=2-1, to=2-2]
	\arrow["M"', from=2-2, to=1-2]
\end{tikzcd}\]
    \end{itemize}

    Alternatively, if the unique $dual \circ ui \circ dual$ is of the form $D$ and $r$ is a chain, then the only possible raising operator involving $r$ is one of the form $S.$ Here, Part (8) of Theorem \ref{Commutativity of Operators} gives a commutative diagram.
    \[\begin{tikzcd}
	{\Theta_2(\EE)} & {\Theta_2(\EE')} \\
	{\Theta_1(\EE)} & {\Theta_1(\EE')}
	\arrow["S", from=1-1, to=1-2]
	\arrow["D", from=2-1, to=1-1]
	\arrow["S", from=2-1, to=2-2]
	\arrow["D"', from=2-2, to=1-2]
\end{tikzcd}\]

    Having considered all possible operators that interact with the same rows as the $dual \circ ui \circ dual,$ we have proved the result.
\end{proof}

Next, we demonstrate the reverse direction.

\begin{lemma}
\label{commute 2 to 1}
    Let $\EE$ be of type $Y_\mathcal{M}$ with $\mathcal{M}$ starting at zero. Let $T$ be a raising operator of type 3' from $\Theta_2(\EE)$ to $\EE'$ which is not equal to the $ui^{-1}$ on $\Theta_2(\EE)$ which gives $\Theta_3(\EE).$ Then there exists $\EE''$ of type $X_n$ such that the following diagram commutes.
\[\begin{tikzcd}
	{\Theta_2(\EE)} & {\EE' = \Theta_2(\EE'')} \\
	{\Theta_1(\EE)} & {\Theta_1(\EE'')}
	\arrow["T", from=1-1, to=1-2]
	\arrow["{dual \circ ui \circ dual}", from=2-1, to=1-1]
	\arrow["T", from=2-1, to=2-2]
	\arrow["{dual \circ ui \circ dual}"', from=2-2, to=1-2]
\end{tikzcd}\]
\end{lemma}

\begin{proof}
    We claim that the following diagram commutes.
    \[\begin{tikzcd}
	{\Theta_2(\EE)} & {\EE'} \\
	{\Theta_1(\EE)} & {T(\Theta_1(\EE))}
	\arrow["T", from=1-1, to=1-2]
	\arrow["{dual \circ ui \circ dual}", from=2-1, to=1-1]
	\arrow["T", from=2-1, to=2-2]
	\arrow["{dual \circ ui \circ dual}"', from=2-2, to=1-2]
\end{tikzcd}\]
    To see this, let $h = ([c_{\max} + 1, -c_{\max} - 1], c_{\max} + 1, -\eta(\EE))$ be the first row of $\Theta_1(\EE),$ let $r \in \Theta_1(\EE)$ be the row combined with $h$ through the unique $dual \circ ui \circ dual,$ and let $r' \in \Theta_2(\EE)$ be the image of $r.$ If $T$ does not involve $r',$ then the existence of the commutative diagram follows from Part (1) of Theorem \ref{Commutativity of Operators}.If $T$ involves $r',$ then the $dual \circ ui \circ dual$ is either $M$ or $D.$ If $M,$ then $r'$ is a merged hat, and $T$ is either $M,$ $D,$ or $U.$ If $D,$ then $r'$ is a row of circles and $T$ is of the form $S.$ In either case, the existence of the four commutative diagrams used in the proof of Lemma \ref{commute 1 to 2} verify the existence of the desired diagram.

    The operation $T$ is not the $dual \circ ui \circ dual$ operation from $\Theta_1(\mathcal{E})$ to $\Theta_2(\mathcal{E}).$ Therefore by Lemma \ref{commute 1 to 0}, it is clear that $T(\Theta_1(\mathcal{E})) = \Theta_1(\mathcal{E}'')$ for some $\mathcal{E}'' \sim \mathcal{E}.$ The fact that $(dual \circ ui \circ dual)(\Theta_1(\mathcal{E}'')) = \Theta_2(\mathcal{E}'')$ is evident.
\end{proof}

Next, we verify a similar result for $\Theta_1(\EE)$ and $\Theta_4(\EE).$

\begin{lemma}
\label{commute 1 to 4}
    Let $\EE$ be of type $Y_\mathcal{M}$ with $\mathcal{M}$ beginning at zero, and let $h = ([n + 1, -n - 1], n + 1, -\eta(\EE))$ be the first row of $\Theta_1(\EE).$ Let $T \in \{dual \circ ui \circ dual, ui^{-1}\}$ be a raising operator of type 3' on $\Theta_1(\EE)$ not equal to the $dual \circ ui \circ dual$(s) going to $\Theta_2(\EE)$ or $\Theta_4(\EE)$. Then the following diagram commutes.
    \[\begin{tikzcd}
	{\Theta_4(\EE)} & {\Theta_4(\EE')} \\
	{\Theta_1(\EE)} & {\Theta_1(\EE')}
	\arrow["T", from=1-1, to=1-2]
	\arrow["{D}", from=2-1, to=1-1]
	\arrow["T", from=2-1, to=2-2]
	\arrow["{D}"', from=2-2, to=1-2]
\end{tikzcd}\]
\end{lemma}

\begin{proof}
    Again, we know that there exists $\EE'$ such that $T$ goes from $\Theta_1(\EE)$ to $\Theta_1(\EE')$ by Lemma \ref{commute 0 to 1}. Let $r$ be the last row of  $\Theta_1(\EE)$. The operator $T$ cannot involve $h$, so we can reduce to the case where $T$ involves $r$. Indeed, otherwise, the lemma follows from Part (1) of Theorem \ref{Commutativity of Operators}. But $r$ is a multiple, and the definitions of $S, M, U,$ and D; the only type 3' raising operators, imply that they do not involve multiples.
\end{proof}

Conversely, raising operators in the $\Theta_4$ layer can descend to the $\Theta_1$ layer.

\begin{lemma}
\label{commute 4 to 1}
    Let $\EE$ be of type $Y_\mathcal{M}$ with $\mathcal{M}$ beginning at zero and $m_{c_{\max}} \geq 3.$ Let $T$ be a raising operator of type 3' from $\Theta_4(\EE)$ to $\EE'$ not equal to the $ui^{-1}$ from $\Theta_4(\EE)$ to $\Theta_3(\EE).$ Then there exists $\EE''$ of type $Y_\mathcal{M}$ such that the following diagram commutes.
\[\begin{tikzcd}
	{\Theta_4(\EE)} & {\EE' = \Theta_4(\EE'')} \\
	{\Theta_1(\EE)} & {\Theta_1(\EE'')}
	\arrow["T", from=1-1, to=1-2]
	\arrow["{D}", from=2-1, to=1-1]
	\arrow["T", from=2-1, to=2-2]
	\arrow["{D}"', from=2-2, to=1-2]
\end{tikzcd}\]
\end{lemma}

\begin{proof}
    Let $h = ([c_{\max} + 1, -c_{\max} - 1], n + 1, -\eta(\EE))$ be the first row of $\Theta_1(\EE)$ and let the multiple $r \in \Theta_1(\EE)$ be the last row. Let $r'$ be the image of $r$ under the dualization: a row of two circles. If $T$ does not involve $r',$ then the existence of the commutative diagram follows from Part (1) of Theorem \ref{Commutativity of Operators}. The only row operation involving the row of two circles $r'$ is $S_{r, 1},$ which we have assumed $T$ is not.
\end{proof}

Finally, we show that most operations $T$ can ascend and descend between $\Theta_2$ and $\Theta_3.$

\begin{lemma}
\label{commute 2 to 3}
    Let $\EE$ be of type $Y_\mathcal{M}$ with $\mathcal{M}$ beginning at zero. Let $T \in \{dual \circ ui \circ dual, ui^{-1}\}$ be a raising operator of type 3' on $\Theta_2(\EE)$ not equal to the $ui^{-1}$ from $\Theta_2(\EE)$ to $\Theta_3(\EE).$ Then the following commutative diagram commutes.
    \[\begin{tikzcd}
	{\Theta_3(\EE)} & {\Theta_3(\EE')} \\
	{\Theta_2(\EE)} & {\Theta_2(\EE')}
	\arrow["T", from=1-1, to=1-2]
	\arrow["{ui^{-1}}", from=2-1, to=1-1]
	\arrow["T", from=2-1, to=2-2]
	\arrow["{ui^{-1}}"', from=2-2, to=1-2]
\end{tikzcd}\]
\end{lemma}

\begin{proof}
    Let $r$ be the row that is split by the $ui^{-1}$ from $\Theta_2(\EE)$ to $\Theta_3(\EE).$ If $T$ does not directly involve $r,$ then the result follows from Part (1) of Theorem \ref{Commutativity of Operators}. Otherwise, we consider a number of cases.

    If the $ui^{-1}$ from $\Theta_2(\EE)$ to $\Theta_3(\EE)$ is of the form $U$ and $r$ is a hat, then there are three possibilities for $T$. If $T$ is of form $U,$ then Part (7) of Theorem \ref{Commutativity of Operators} guarantees the existence of the commutative diagram. If $T$ is of form $D,$ then Part (9) of Theorem \ref{Commutativity of Operators} guarantees the existence of the diagram. If $T$ is of form $M$, then Part (5) gives the diagram.

    Alternatively, if the $ui^{-1}$ is of the form $S$ and $r$ is a chain, then the only possible raising operator involving $r$ is one of the form $S.$ Here, Part (2) of Theorem \ref{Commutativity of Operators} gives the commutative diagram. Having considered all possible operators that interact with the same rows as the $ui^{-1}$ we have proved the result.
\end{proof}

Conversely, all operations on the $\Theta_3$ layer descend to the $\Theta_2$ layer.

\begin{lemma}
\label{commute 3 to 2}
    Let $\EE$ be of type $Y_\mathcal{M},$ with $\mathcal{M}$ starting at zero. Let $T$ be a raising operator of type 3' from $\Theta_3(\EE)$ to $\EE.'$ Then there exists $\EE''$ of type $X_n$ such that the following diagram commutes.
\[\begin{tikzcd}
	{\Theta_3(\EE)} & {\EE' = \Theta_3(\EE'')} \\
	{\Theta_2(\EE)} & {\Theta_2(\EE'')}
	\arrow["T", from=1-1, to=1-2]
	\arrow["{ui^{-1}}", from=2-1, to=1-1]
	\arrow["T", from=2-1, to=2-2]
	\arrow["{ui^{-1}}"', from=2-2, to=1-2]
\end{tikzcd}\]
\end{lemma}

\begin{proof}
    Let $r \in \Theta_2(\EE_1)$ affected by the $ui^{-1}$ from $\Theta_2(\EE)$ to $\Theta_3(\EE),$ and let $r'$ be its image after the last circle is separated by the $ui^{-1}.$ If $T$ does not involve $r',$ then the existence of the commutative diagram follows from Part (1) of Theorem \ref{Commutativity of Operators}.

    We may then assume that $T$ involves $r'$. The $ui^{-1}$ from $\Theta_2(\EE)$ to $\Theta_3(\EE)$ is either a $U$ or an $S.$ If $U,$ then $r'$ is a hat and $T$ is either $M, D,$ or $D.$ If $S,$ then $r'$ is a row of circles and $T$ is an $S.$ In either case, the existence of the four commutative diagrams used in the proof of Lemma \ref{commute 3 to 2} verify the existence of the desired diagram.

    Once we know that the diagram commutes, it is clear from Lemma \ref{commute 2 to 1} that $T$ must go from $\Theta(\EE)$ to $\Theta_2(\EE'')$ for some $\EE''.$ This then implies that $\EE' = \Theta_2(\EE''),$ completing the proof.
\end{proof}

Theorem \ref{commutativity theta correspondence} now follows directly from the above lemmas.

\begin{rmk}
    These commutativity results have practical use apart from providing a conceptual understanding of Theorems \ref{first block m_n = 1}, \ref{first block m_n > 1}, and \ref{Almost block classification}. They also provide a mechanism by which one could prove that sets like the ones detailed in this theorem are closed under row operations. In the above proofs of these theorems, we proved this closure property with less difficulty by examining the effect of row operations on $\mathcal{S}$-data. However, cases more general than those considered in this article (e.g., non-tempered and non-anti-tempered) may still involve operations like $S, M, U,$ and $D$ while not admitting a structural description like the $\mathcal{S}$-data. In such cases, results similar to the above eight lemmas and Theorem \ref{Commutativity of Operators} (which may be proved without reference to the $\mathcal{S}$-data) may prove useful.
\end{rmk}

\bibliographystyle{amsplain}
\bibliography{refs}

\end{document}